\documentclass[a4paper,11pt,reqno]{amsart}
\usepackage{amsfonts}
\usepackage{amssymb,amsmath,amssymb}
\usepackage[foot]{amsaddr}
\usepackage{times}
\usepackage{bm}
\usepackage{booktabs}
\usepackage{xspace}
\usepackage{natbib}
\usepackage{amsfonts}
\usepackage{amsmath}
\usepackage{natbib}
\usepackage{graphicx}
\usepackage{subfigure}
\usepackage{multirow,rotating}
\usepackage{epstopdf}
\usepackage{xcolor}
\usepackage{xr}
\usepackage[a4paper]{geometry}
\usepackage[plain,noend]{algorithm2e}
\graphicspath{{pic/}} \makeatletter
\def\T{{\mathrm{\scriptscriptstyle T} }}

\newtheorem{theorem}{Theorem}

\newtheorem{corollary}{Corollary}

\newtheorem{definition}{Definition}
\newtheorem{example}{Example}
\newtheorem{lemma}{Lemma}

\newtheorem{proposition}{Proposition}
\newtheorem{remark}{Remark}

\allowdisplaybreaks
\begin{document}

\title[Confidence Regions in Wasserstein DRO]{{\large
\MakeLowercase{\uppercase{C}onfidence \uppercase{R}egions in
    \uppercase{W}asserstein \uppercase{D}istributionally
    \uppercase{R}obust \uppercase{E}stimation}}}
\author[Blanchet]{Jose Blanchet}
\address[A1]{Management Science and Engineering, Stanford University}
\author[Murthy]{Karthyek Murthy}
\address[A2]{Engineering Systems \& Design, Singapore University of
Technology \& Design}
\author[Si]{Nian Si}
\address[A3]{Management Science and Engineering, Stanford University}
\email{jose.blanchet@stanford.edu,
karthyek\_murthy@sutd.edu.sg,niansi@stanford.edu}
\maketitle




\maketitle

\begin{abstract}
  Wasserstein distributionally robust optimization estimators are
  obtained as solutions of min-max problems in which the statistician
  selects a parameter minimizing the worst-case loss among all
  probability models within a certain distance (in a Wasserstein
  sense) from the underlying empirical measure. While motivated by the
  need to identify optimal model parameters or decision choices that
  are robust to model misspecification, these distributionally robust
  estimators recover a wide range of regularized estimators, including
  square-root lasso and support vector machines, among others, as
  particular cases. This paper studies the asymptotic normality of
  these distributionally robust estimators as well as the properties
  of an optimal (in a suitable sense) confidence region induced by the
  Wasserstein distributionally robust optimization formulation. In
  addition, key properties of min-max distributionally robust
  optimization problems are also studied, for example, we show that
  distributionally robust estimators regularize the loss based on its
  derivative and we also derive general sufficient conditions which
  show the equivalence between the min-max distributionally robust
  optimization problem and the corresponding max-min formulation.
\end{abstract}

%







\section{Introduction}
\label{sec:Intro}
In recent years, distributionally robust optimization  formulations
based on Wasserstein distances have sparked a substantial amount of
interest in the literature. One reason for this interest, as
demonstrated by a range of examples in statistical learning
and operations research,
is that these formulations provide a flexible way to quantify and
hedge against the impact of model misspecification.  Motivated by those
applications, this paper aims to understand the fundamental
statistical properties, such as asymptotic normality of the
distributionally robust estimators and the associated confidence
regions deemed optimal in a suitable sense to be described shortly.

Before providing a review of Wasserstein distributionally robust
optimization and its connections to several areas, such as artificial
intelligence, machine learning and operations research, we set the
stage by first introducing the elements of a typical data-driven
distributionally robust estimation problem.

 Suppose that
$\{X_k: 1 \leq k \leq n\} \subset \mathbb{R}^m$ are independent and
identically distributed samples 
from an unknown distribution $P_\ast.$ A typical non-robust stochastic
optimization formulation informed by $P_{n}$ focuses on minimizing
empirical expected loss of the form,
$E_{P_{n}}\left\{ \ell (X;\beta )\right\} = n^{-1} \sum_{i=1}^n
\ell(X_i;\beta)$, over the parameter choices
$\beta \in B \subseteq \mathbb{R}^d.$ In this paper, we take $B$ to be
a closed, convex subset of $\mathbb{R}^d.$
Let the empirical risk minimization estimators be
\begin{equation}
  \beta _{n}^{ERM}\in \text{arg}\min_{\beta \in B}E_{P_{n}}\left\{ \ell (X;\beta
    )\right\}.  \label{ERM_1}
\end{equation}%

On the other hand, a distributionally robust formulation recognizes
the distributional uncertainty inherent in $P_{n}$ being a noisy
representation of an unknown distribution. Therefore, it enriches the
empirical risk minimization (\ref{ERM_1}) by considering an estimator
of the form,
\begin{equation}
  \beta _{n}^{DRO}(\delta )\in \arg \min_{\beta \in B}\sup_{P\,\in \,\mathcal{U}%
    _{\delta }(P_{n})}E_{P}\left\{ \ell (X;\beta )\right\} ,
\label{Wass_DRO_A}
\end{equation}%
where the set $\mathcal{U}_{\delta }(P_{n})$ is called the
distributional uncertainty region and $\delta $ is the size of the
distributional uncertainty. Here, given a measurable function
$f(\cdot),$ the notation $E_P\{f(X)\}$ denotes expectation
with respect to a probability distribution $P.$ Wasserstein
distributionally robust formulations advocate choosing,
\begin{equation*}
  \mathcal{U}_{\delta }(P_{n})=\{P \in \mathcal{P}(\Omega):
  W(P_{n},P)\leq \delta^{1/2}\},
\end{equation*}%
where $W(P_{n},P)$ is the Wasserstein distance between distributions
$P_{n}$ and $P$ defined below, and $\mathcal{P}(\Omega)$ is the set of probability
distributions supported on a closed set
$\Omega \subseteq \mathbb{R}^m.$ 
\begin{definition}[Wasserstein distances]
  \textnormal{Given a lower semicontinuous function
    $c:\Omega \times \Omega \rightarrow [0,\infty],$ the optimal
    transport cost $D_c(P,Q)$ between any two distributions
    $P,Q \in \mathcal{P}(\Omega)$ is defined as,
    \begin{align*}
      D_c(P,Q) = \min_{\pi \in \Pi(P,Q)}  E_{\pi}\left\{c(X,Y)\right\}
    \end{align*}
    where $\Pi(P,Q)$ denotes the set of all joint distributions of the
    random vector $(X,Y)$ with marginal distributions $P$ and $Q,$
    respectively. If we specifically take $c(x,y) = d(x,y)^2,$ where
    $d(\cdot)$ is a metric, we obtain the Wasserstein distance of
    order 2 by setting
    $
      W(P,Q) = \left\{D_c(P,Q)\right\}^{1/2}.
    $
  }
  \label{defn:WD}
\end{definition}


The quantity $W(P_{n},P)$ may be interpreted as the cheapest way to
transport mass from the distribution $P_{n}$ to the mass of another
probability distribution $P,$ while measuring the cost of
transportation from location $x \in \Omega$ to location $y \in \Omega$
in terms of the squared distance between $x$ and $y.$ In this paper,
we shall work with Wasserstein distances of order 2, which explains
why it is natural to use $\delta^{1/2}$ to specify the distributional
uncertainty region $\mathcal{U}_{\delta }(P_{n})$ as above.  Since
$W(P_n,P_n) = 0,$ the empirical risk minimizing estimator in
(\ref{ERM_1}) can be seen as a special case of the formulation
(\ref{Wass_DRO_A}) by setting $\delta = 0.$

  The need for selecting model parameters or making decisions using a data driven approach which is robust to model uncertainties has sparked a rapidly growing literature on Wasserstein distributionally robust optimization, via formulations such as  (\ref{Wass_DRO_A}); see, for example,
\citet{MohajerinEsfahani2017,ZHAO2018262,blanchet2016quantifying,gao2016distributionally,gao2018robust,wolfram2018}
for applications in operations research and
\cite{yang2017convex,yang2018Wasserstein} for examples
specifically in stochastic control.

In principle, the min-max formulation (\ref{Wass_DRO_A}) is
``distributionally robust'' in the
sense that its solution guarantees a uniform performance over all
probability distributions in $%
\mathcal{U}_{\delta_{n}}(P_{n}).$ Roughly speaking, for every choice
of parameter or decision $\beta,$ the min-max game type formulation in
(\ref{Wass_DRO_A}) introduces an adversary that chooses the most
adversarial distribution from a class of distributions
$\mathcal{U}_{\delta_{n}}(P_{n}).$ The goal of the procedure is to
then choose a decision that also hedges against these adversarial
perturbations, thus introducing adversarial robustness into settings
where the quality of optimal solutions are sensitive to incorrect
model assumptions.

Interestingly, the min-max formulation (\ref{Wass_DRO_A}), which is
derived from the above robustness viewpoint, has been shown to recover
many machine learning estimators when applied to suitable loss
functions $\ell(\cdot)$; some examples include the square-root lasso and
support vector machines \citep{blanchet2016robust}, the group lasso
\citep{blanchet2017distributionally2}, adaptive regularization
\citep{volpi2018generalizing,blanchet2017data}, among others
\citep{shafieezadeh-abadeh_distributionally_2015,gaostatlearn,
  duchidistributionally,chen2018robust}.
The utility of the distributionally robust formulation
(\ref{Wass_DRO_A}) has also been explored in adversarial training of
Neural Networks; see, for example
\citet{sinha2018certifiable,staib2017distributionally}.

Generic formulations such as (\ref{Wass_DRO_A}) are becoming increasingly
tractable; see, for example, %
\citet{MohajerinEsfahani2017,luo2017decomposition} for convex programming
based approaches and \citet{sinha2018certifiable,blanchet2018optimal} for
stochastic gradient descent based iterative schemes. 


Motivated by these wide range of applications, we investigate the asymptotic behaviour of the optimal value and optimal solutions of (\ref{Wass_DRO_A}). In order to
specifically describe the contributions, let us introduce the following
notation. For any positive integer $n$ and $\delta _{n}>0,$ let
\begin{equation*}
\Psi _{n}(\beta )=\sup_{P\,\in \,\mathcal{U}_{\delta
_{n}}(P_{n})}E_{P}\left\{ \ell (X;\beta )\right\} 
\end{equation*}%
denote the distributionally robust objective function in
\eqref{Wass_DRO_A}. Suppose that $\beta _{\ast }$ uniquely minimizes
the population risk. According to (\ref{ERM_1}) - (\ref{Wass_DRO_A}),
we have $\beta_n^{DRO}$ and $\beta_n^{ERM}$ minimize, respectively,
the distributionally robust loss $\Psi_n(\beta)$ and the empirical
loss in (\ref{ERM_1}). Next, let
\begin{equation}
\Lambda _{\delta _{n}}(P_{n})=\big\{ \beta \in B:\beta \in
\arg \min_{\beta  \in B }E_{P}\left\{ \ell (X;\beta )\right\} \text{ for some }P\in
\mathcal{U}_{\delta _{n}}(P_{n})\big\}
\label{NaturalCR}
\end{equation}
denote the set of choices of $\beta \in B$ that are ``compatible'' with the distributional uncertainty
region, in the sense that for every
$\beta \in \Lambda _{\delta _{n}}(P_{n}),$ there exists a probability
distribution $P\in \mathcal{U}%
_{\delta _{n}}(P_{n})$ for which $\beta $ is optimal. In other words,
if $\mathcal{U}_{\delta _{n}}(P_{n})$ represents the set of
probabilistic models which are, based on the empirical evidence,
plausible representations  of the
underlying phenomena, then each of such representations induces an
optimal decision and $\Lambda _{\delta _{n}}(P_{n})$ encodes the set of plausible decisions. Let
  $\Lambda^+ _{\delta _{n}}(P_{n})$ be the closure of $\cap_{\epsilon>0}\Lambda_{\delta_n+\epsilon}(P_n)$. Typically,  $\Lambda^+_{\delta _{n}}(P_{n}) = \Lambda_{\delta _{n}}(P_{n})$, but this is not always true as illustrated in Example \ref{eg:Lambda-plus}. Asymptotically, as $\delta_n$ decreases to zero, the distinction is negligible. However, choosing a set such as $\Lambda^+ _{\delta _{n}}(P_{n})$ as a natural set of plausible decisions is sensible because we guarantee that a distributionally robust solution belongs to this region. Our main result also implies that all distributionally robust solutions are asymptotically equivalent; within $o_p(n^{-1/2})$ distance from each other.

{
With the above notation, the key contributions of this article can be
described as follows.

We first establish the convergence in distribution of the triplet,
\begin{equation}
\big( n^{1/2}\{\beta_{n}^{ERM}-\beta_{\ast}\},\ n^{\bar{\gamma}/2}\{\beta
_{n}^{DRO}(\delta_n)-\beta_{\ast}\},\ n^{1/2}\big\{
\Lambda^+_{\delta_{n}}(P_{n} )-\beta_{\ast}\big\} \big) ,
\label{MLT-triplet}
\end{equation}
for a suitable $\bar{\gamma}\in(0,1/2]$ that depends
  on the rate at which the size of the distributional uncertainty,
$\delta_{n}$%
, is decreased to zero; see Theorem
  \ref{thm:levelsets-master}. We identify the joint limiting
  distributions of the triplet \eqref{MLT-triplet}. The third
component of the triplet in \eqref{MLT-triplet}, namely, $%
n^{1/2}\{\Lambda^+_{\delta_{n}}(P_{n})-\beta_{\ast}\}$, considers a
suitably scaled and centered version of the choices of $\beta \in B$
which are compatible with the respective distributional uncertainty
region $\mathcal{U}_{\delta_{n}}(P_{n})$ in the sense described
above. Therefore,  $\Lambda^+_{\delta_{n}}(P_{n})$ is a natural choice of the confidence region. We further develop an approximation for $\Lambda^+_{\delta_{n}}(P_{n})$; see Section \ref{Sec_Confidence_Regions}.

Second, we utilize
the limiting result of \eqref{MLT-triplet} to examine how the choice
of the size of distributional ambiguity, $\delta_{n}$, affects the
qualitative properties of the distributionally robust estimators and
the induced confidence regions. Specifically, choosing
$\delta_{n}=\eta n^{-\gamma}$, we characterize the behaviour of the
solutions for different choices of $\eta,\gamma\in (0,\infty),$ as
$n\rightarrow\infty$. It emerges that the canonical, $%
O(n^{-1/2})$, rate of convergence is achieved only if $\gamma\leq1$
and the limiting distribution corresponding to the distributionally
robust estimator and that of the empirical risk minimizer
are different only if $\gamma\geq1$. Hence to both
obtain the canonical rate and tangible benefits from the distributionally robust optimization
formulation, we must choose $\gamma=1$, which corresponds to the
resulting $\bar{\gamma}$ in \eqref{MLT-triplet} to be equal to 1.
Moreover, given any $\alpha\in(0,1)$, utilizing the limiting distribution of
the triplet in \eqref{MLT-triplet}, we are able to identify a positive
constant $\eta_{\alpha}\in(0,+\infty)$ such that whenever $\eta\geq
\eta_{\alpha}$ in the choice $\delta_{n}=\eta/n$, the set $\Lambda^+
_{\delta_{n}}(P_{n})$ is an asymptotic $(1-\alpha)$-confidence region for $%
\beta_{\ast}$. 

Finally, we establish the existence of an equilibrium game value. The distributionally robust optimization formulation assumes that the adversary selects a probability model after the statistician chooses a parameter. The equilibrium value of the game is attained if inf-sup equals sup-inf in (\ref{Wass_DRO_A}), namely, if we allow the statistician to choose a parameter optimally after the adversary selects a probability model. We show in great generality that the equilibrium value of the game exists; see Theorem \ref{minmax_prop}.}
We end the introduction with a discussion of related statistical results.
The asymptotic normality of M-estimators which minimize an empirical risk of
the form, $E_{P_{n}}\{\ell(X;\beta)\},$ was first established in the
pioneering work of \citet{huber1967}.  Subsequent asymptotic
characterizations in the presence of constraints on the choices of
parameter vector $\beta$ have been developed in
\citet{dupacova1988,shapiro1989,Shapiro1991,shapiromoor,shapiro2000},
again in the standard M-estimation setting.  Our work here is
different because of the presence of the adversarial perturbation to
the loss represented by the inner maximization in (\ref{Wass_DRO_A}).

Asymptotic normality in the related context of regularized estimators for
least squares regression has been established in %
\citet{knight2000asymptotics}. As mentioned earlier, distributionally robust
estimators of the form (\ref{Wass_DRO_A}) recover lasso-type estimators as
particular examples \citep{blanchet2016robust}. In these cases, the
inner max problem involving the adversary can be solved in closed form,
resulting in the presence of regularization. However, our results can be
applied even in the general context in which no closed form solution to the
inner maximization can be obtained. Therefore, our results in this paper can
be seen as extensions of the results by \citet{knight2000asymptotics}, from
a distributionally robust optimization perspective.

We comment that some of our results involving convergence of sets may
be of interest to applications in the area of empirical likelihood %
\citep{owen_empirical_1988,owen_empirical_1990,owen2001empirical}. This
is because $\Lambda_{\delta_{n}}( P_{n}) $ can be characterized in
terms of a function, namely, the robust Wasserstein profile function,
which resembles the definition of the empirical likelihood profile
function. We refer the reader to \citet{blanchet2016robust} for more discussion on the robust Wasserstein profile function and its connections to empirical likelihood. We also refer to \citet{cisneros2019distributionally} for additional applications, including graphical lasso, which could benefit from our results.


\label{sec:assumption_results}

\section{Preliminaries and Assumptions}

\subsection{Convergence of closed sets}
\label{SectConv_Sets}
We begin with a brief introduction to the notion of convergence of
closed sets before introducing the assumptions required to state our
main results.  For a sequence $\{A_k: k \geq 1\}$ of closed subsets of
$\mathbb{R}^d,$ the inner and outer limits are defined, respectively,
by
\begin{align*}
\text{Li}_{n \rightarrow \infty}\ A_n = &\big\{ z \in \mathbb{R}^d: \text{
there exists a sequence } (a_n)_{n \geq 1} {\text{with } a_n \in A_n} \text{ convergent to } z\},
 \text{ and } \\
\text{Ls}_{n \rightarrow \infty}\ A_n = &\big\{ z \in \mathbb{R}^d: \text{
there exist positive integers } n_1 < n_2 < n_3 < \cdots \text{ and } a_k
\in A_{n_k}  \\
&\qquad\qquad\qquad\qquad\qquad\qquad  \text{ such that the sequence }
(a_k)_{k \geq 1} \text{ is convergent to } z\big\}.
\end{align*}
We clearly have 
$\text{Li}_{n \rightarrow \infty}\ A_{n}\subseteq \text{Ls}_{n \rightarrow
\infty}\ A_{n }.$ 
The sequence $\{A_{n}:n\geq 1\}$ is said to converge to a set $A$ in the
Painlev\'{e}-Kuratowski (PK) sense if
\begin{equation*}
A= \text{Li}_{n\rightarrow \infty}\ A_{n}=\text{Ls}_{n \rightarrow \infty} \
A_{n},
\end{equation*}%
in which case we write PK-$\lim_{n}A_{n}=A$. Since $\mathbb{R}^{d}$ is a
locally compact Hausdorff space, the topology induced by Painlev\'{e}%
-Kuratowski convergence on the space of closed subsets
of $\mathbb{R}^{d}$ is completely metrisable, separable and coincides
with the well-known topology of closed convergence, also known as Fell
topology; see \citet[Chapter 1]{molchanov2005theory}. The notion of
convergence of sets we utilize here will be the above defined
Painlev\'{e}-Kuratowski convergence. { After equipping the space of closed subsets with the Borel $\sigma$-algebra, we are able to define probability measures and further define the usual weak convergence of measures; see, for example, \citet[Chapter 1]{billingsley2013convergence}.}

\subsection{Assumptions and notation}
Throughout the paper, we use $A\succ 0$ to denote that a given
symmetric matrix $A$ is positive definite and the notation $C^\circ$ and ${\rm cl} (C)$
to denote the interior and closure of a subset $C$ of Euclidean space, respectively. In the case
of taking expectations with respect to the data-generating distribution
$P_\ast,$ we drop the subindex in the expectation operator as in,
$E_{P_\ast} \left\{ f(X)\right\} = E\left\{ f(X)\right\}.$ We use $\Rightarrow$ to denote weak convergence and $\rightarrow$ to denote convergence in probability. We let $\mathbb{I}(\cdot)$ be the indicator function. { Let $\|\cdot\|_p$ be the dual norm of  $\|\cdot\|_q$ where $1/p + 1/q = 1$ for $q\in (1,\infty),$ and $p = \infty$ or $1$ for $q=1$ or $\infty$, respectively.}

As mentioned in Section \ref{sec:Intro}, suppose that $\Omega$ is a
closed subset of $\mathbb{R}^m$ and $B$ is a closed, convex subset of
$\mathbb{R}^d.$ Assumptions A1 and A2 below are taken to be satisfied
throughout the development, unless indicated otherwise.

(A1) The transportation cost
$c:\Omega \times \Omega \rightarrow [0,\infty]$ is of the form
$c(u,w)=\Vert u-w\Vert_{q}^{2}$.

(A2)  The function
$\ell:\Omega \times B \rightarrow \mathbb{R}$ satisfies the following
properties:
\begin{itemize}
\vspace*{-3mm}
\item[a)] The loss function $\ell(\cdot)$ is twice continuously differentiable,
  and for each $x$, $\ell(x,\cdot)$ is convex.
\item[b)] Let $h(x,\beta)=D_{\beta}\ell(x,\beta)$, and there
  exists $\beta_{\ast}\in B^\circ$ satisfying the optimality condition
  $E\{ h(X,\beta_{\ast })\} = 0.$ In addition,
  $E \{\Vert h(X,\beta_{\ast })\Vert _{2}^{2}\} <\infty,$ the
   symmetric matrix
  $C=E\left\{ D_{\beta }h(X,\beta_{\ast})\right\} \succ 0,$
  $E\left\{ D_{x}h(X,\beta_{\ast})D_{x}h(X,\beta_{\ast})^\T\right\}
  \succ 0$, and { $\mathrm{pr}\{\|D_x\ell(X,\beta_*)\|_p >0\}>0$}.

\item[c)] For every $\beta \in \mathbb{R}^d,$
  $\| D_{xx}\ell(\,\cdot\,;\beta) \|_p$ is  uniformly continuous and bounded by a continuous
  function $%
  M(\beta)$. {Further, there exists a positive constant $M'<\infty$ such that $\|D_xh(x,\beta)\|_q \leq M'(1+\|x\|_q) $  for $\beta$ in a neighborhood of $\beta_*$}. In addition, $D_xh(\cdot)$ and
  $D_\beta h(\cdot)$ satisfy the following locally Lipschitz
  continuity:
\begin{align*}
\left\Vert D_{x}h(x+\Delta,\beta_{\ast}+u)-D_{x}h(x,\beta_{\ast})\right\Vert
_{q} & \leq\kappa^{\prime}(x)\big( \left\Vert \Delta\right\Vert
_{q}+\left\Vert u\right\Vert _{q}\big), \\
\left\Vert D_{\beta}h(x+\Delta,\beta_{\ast}+u)-D_{\beta}h(x,\beta_{\ast
})\right\Vert _{q} & \leq\bar{\kappa}(x)\big(\left\Vert \Delta\right\Vert
                     _{q}+\left\Vert u\right\Vert _{q}\big),
\end{align*}
for $ \left\Vert\Delta\right\Vert
_{q}+\left\Vert u\right\Vert _{q} \leq 1$, where $\kappa^{\prime},\bar{\kappa}:\mathbb{R}^{m}\rightarrow\lbrack0,%
\infty) $ are such that $E[\{\kappa^{\prime}(X_{i})\}^{2}]<\infty$ and $E\{%
\bar {\kappa}^{2}(X_{i})\}<\infty.$
\end{itemize}


Assumption A1 covers most of the cases in the literature described in
Section \ref{sec:Intro}. One exception that does not immediately
satisfy Assumption A1, but which can be easily adapted after a simple
change-of-variables, is the weighted $l_{2}$ norm (also known as
Mahalanobis distance), namely
$c( x,y) =( x-y) ^{{\ \mathrm{\scriptscriptstyle T} }}A( x-y) $, where
$A \succ 0,$ see \citet{blanchet2018optimal}. 
The requirement that $\ell (\cdot )$ is twice differentiable in
Assumption A2.a is useful in the analysis to identify a second-order
expansion for the objective in (\ref{Wass_DRO_A}), which helps
quantify the the impact of adversarial perturbations. Convexity of
$\ell (x,\cdot ),$ together with $C$ being positive definite in A2.b,
implies uniqueness of $\beta _{\ast }$. The uniqueness
  of $\beta_*$ is a standard assumption in the derivation of rates of
  convergence for estimators; see, for example,
  \citet{huber1967}, \citet[Section 3.2.2]{van1996weak}.  Assumption
A2.b also allows us to rule out redundancies in the underlying source
of randomness (e.g. colinearity in the setting of linear regression).
 The first part of Assumption A2.c ensures that the inner maximization in  (\ref{Wass_DRO_A}) is finite by controlling the magnitude of the adversarial perturbations. {
 The local Lipschitz continuity requirement in $x$ arises with the optimal transportation analysis technique in \cite{blanchet2016robust}, c.f. Assumption A6. Analogous regularity in $\beta$ is useful in proving the confidence region limit theorem; see the discussion following Theorem \ref{thm:master}. }
Limiting results which study the impact of relaxing some of these assumptions are given immediately after describing the main result in Section
\ref{Sec_main_limit} below.

\section{Main results}
\subsection{The main limit theorem\label{Sec_main_limit}}
In order to state our main results we introduce a few more definitions. Define
\begin{equation*}
  \varphi (\xi )= 4^{-1}E\big[ \big\Vert \big\{ D_{x}h(X,\beta _{\ast
        })\big\} ^{ \mathrm{\scriptscriptstyle T} }\xi \big\Vert _{p}^{2}\big],
\end{equation*}%
and its convex conjugate,
$\varphi ^{\ast }(\zeta )=\sup_{\xi \in \mathbb{R}^{d}}\left\{ \xi ^{{
\mathrm{\scriptscriptstyle T} }}\zeta -\varphi (\xi )\right\}.$
In addition, define
\begin{align}
  S(\beta )&=\left[E \left\{\Vert D_{x}\ell (X;\beta )\Vert _{p}^{2}
             \right\}\right]^{1/2},
    \label{defn:sensitivity-term}\\
  f_{\eta ,\gamma }(x) &= x\mathbb{I}{(\gamma \geq 1)}-\eta^{1/2} D_{\beta
                         }S(\beta _{\ast })\mathbb{I}{(\gamma \leq 1)},
                         \label{family-f-eta-gamma}
\end{align}
for $\eta \geq 0,\gamma \geq 0$. { By Assumption A2.b, we have $S(\beta)$ is differentiable at $\beta_*$.} Recall the matrix
$C=E\left\{ D_{\beta }h(X,\beta _{\ast })\right\} $ introduced in
Assumption A2.b and
\begin{equation}
  \Lambda^+_{\delta _{n}}(P_{n}) = \mathrm{cl}\left\{\cap_{\epsilon>0} \Lambda_{\delta _{n}+\epsilon}(P_{n})\right\},
  \label{NaturalCRplus}
\end{equation}
which is the right limit of $\Lambda_{\delta_n}(P_n)$ defined in
(\ref{NaturalCR}). Finally, define the sets,
\begin{align}
  \Lambda_{\eta } =\left\{ u:\varphi ^{\ast }(Cu)\leq \eta \right\},
  \qquad
  \Lambda_{\eta ,\gamma} =
\begin{cases}
  \Lambda_\eta \quad &
  \text{ if }  \gamma =1, \\
  \mathbb{R}^{d} \quad  & \text{ if }\gamma <1, \\
  \{0\} \quad   &\text{ if }\gamma >1.%
\end{cases}
\label{Lim_Set}
\end{align}

We now state our main result.

\begin{theorem}
  \label{thm:levelsets-master}
  \sloppy{Suppose that Assumptions A1 - A2 are satisfied with {$q\in(1,\infty)$},
    $\Omega = \mathbb{R}^m$ and $E\left(\Vert X \Vert_2^2\right) < \infty.$ If
    $H\sim\mathcal{N}(0,\mathrm{cov}\{h(X,\beta_{\ast})\})$ and
    $\delta_{n}=n^{-\gamma}\eta$ \ for some
    $\gamma,\eta\in(0,\infty),$} then we have the following joint
  convergence in distribution:
\begin{align*}
  & \big( n^{1/2}\big\{\beta_{n}^{ERM}-\beta_{\ast}\big\},\ n^{\bar{\gamma}/2}
    \big\{\beta_{n}^{DRO}(\delta_n)-\beta_{\ast}\big\},\ n^{1/2}\big\{
    \Lambda^+_{\delta_{n}}(P_{n})-\beta_{\ast}\big\} \big) \\
  &\hspace{200pt} \Rightarrow \big( C^{-1}H,\ C^{-1}f_{\eta,\gamma}(H),\ \Lambda
    _{\eta,\gamma}+C^{-1}H\big) ,
\end{align*}
where $\bar{\gamma}=\min\{\gamma,1\}$ and $\Lambda_{\eta,\gamma}$ is
defined as in (\ref{Lim_Set}).
\end{theorem}

The proof of Theorem \ref{thm:levelsets-master} is presented in
Section \ref{sec:levelsets-master}. { For $q=1$ or $\infty$, which corresponds to $%
p=\infty $ or 1, $S(\beta )$ may not be differentiable at $\beta_*$, then the limiting
distribution  presents a discontinuity which
makes it difficult to use in practice.  Hence, we prefer not to cover
this here.} Theorem \ref{thm:levelsets-master}
can be used as a powerful conceptual tool. For example, let us examine
how a sensible choice for the parameter $\delta_{n}$ can be obtained
as an application of Theorem \ref{thm:levelsets-master} by considering
the following cases:

{Case 1, where $\gamma > 1$:} If $n\delta_{n} \rightarrow 0$
corresponding to the case $\gamma > 1,$ we have $f_{0,\gamma}(H)=H$
from the definition of the parametric family in
\eqref{family-f-eta-gamma}. Therefore, from Theorem
\ref{thm:levelsets-master},
\begin{align*}
 & \big( n^{1/2}\{\beta_{n}^{ERM}-\beta_{\ast}\},\ n^{\bar{\gamma}/2}\{\beta
    _{n}^{DRO}(\delta_n)-\beta_{\ast}\},\ n^{1/2}\big\{
    \Lambda^+_{\delta_{n}}(P_{n})-\beta_{\ast}\big\} \big) \\& \hspace{150pt}\Rightarrow
    \big(C^{-1}H,C^{-1}H,\{C^{-1}H\}\big),
\end{align*}
which implies that the influence of the robustification vanishes in
the limit when $\delta_{n}=o(n^{-1})$.

{Case 2, where $\gamma < 1$:} If $n\delta_{n}\rightarrow\infty$
corresponding to the case $\gamma<1$, the rate of convergence for the
distributionally robust estimator is slower than the canonical than
$O(n^{-1/2})$ rate:
\begin{equation}
\beta_{n}^{DRO}(\delta_n) = \beta_{\ast}-\eta^{1/2} n^{-\gamma/2}
C^{-1}D_{\beta}S(\beta _{\ast})+o_{p}\big(
n^{-\gamma/2}\big),  \label{BDRO}
\end{equation}
where $n^{\gamma/2}o_{p}( n^{-\gamma/2}) \rightarrow0,$ in
probability, as $n\rightarrow\infty$. The relationship (\ref{BDRO})
reveals an uninteresting limit,
$n^{1/2}\{ \Lambda^+_{\delta_{n}} (P_{n})-\beta_{\ast}\}
\Rightarrow\mathbb{R}^{d}$, exposing a slower than $O(n^{-1/2})$ rate
of convergence $\Lambda^+_{\delta_n }(P_{n}).$ In fact, (\ref{BDRO})
indicates that $O( n^{-\gamma /2})$ scaling will result in a
non-degenerate limit.

{Case 3, where $\gamma = 1$:}  when $\delta_{n}=\eta/n$, we have that all components in
the triplet in Theorem \ref{thm:levelsets-master} have non-trivial
limits.

  Theorem \ref{minmax_prop} below provides a geometric
insight relating $\beta_{n}^{DRO}(\delta_n)$, $\beta_{n}^{ERM}$ and
$\Lambda^+_{\delta_{n}}(P_{n})$, which justifies a picture describing
$%
\Lambda^+_{\delta_{n}}(P_{n})$ as a set containing both
$\beta_{n}^{DRO}(%
\delta_n)$ and $\beta_{n}^{ERM}$. The observation that
$\beta_{n}^{ERM}\in\Lambda_{ \delta_{n}}(P_{n})$ is immediate because
$\Lambda_{\delta}(P_{n})$ is increasing in $\delta$, so
$\beta_{n}^{ERM}\in\Lambda_{0}(P_{n})\subset%
\Lambda^+_{\delta_{n}}(P_{n})$. On the other hand, the observation that
$\beta_{n}^{DRO}(\delta_n)\in\Lambda^+_{\delta_{n}}(P_{n})$ is
non-trivial and it relies on the exchangeability of inf and sup in
Theorem \ref{minmax_prop} below. An appropriate choice of $\eta$
which results in the set $\Lambda_{\delta_n}^+(P_{n})$ also possessing
desirable coverage for $\beta_\ast$ is prescribed in Section
\ref{Sec_Confidence_Regions}.
\begin{theorem}
  \label{minmax_prop}
  { Suppose that Assumption A1 is enforced. We further assume the loss function $\ell(\cdot)$ is
  continuous and non-negative, for each $x$, $\ell(x,\cdot)$ is convex, and $E_{P_*}\{\ell(X,\beta)\}$ has a unique optimizer $\beta_*\in B^\circ$.}
  Then for any
  $\delta >0,$
\begin{equation}
  \inf_{\beta \in B}\sup_{P \in\, \mathcal{U}_{\delta }(P_{n}) }
  E_{P}\left\{ \ell(X;\beta )\right\} =
  \sup_{P \in\, \mathcal{U}_{\delta }(P_{n}) }\inf_{\beta \in B} E_{P}
  \left\{ \ell (X;\beta )\right\},
  \label{minmax_eqn}
\end{equation}
and there exists a distributionally robust estimator choice
$\beta_{n}^{DRO}(\delta) \in \Lambda^+_{\delta}(P_{n})$.
\end{theorem}
The proof of Theorem \ref{minmax_prop} is presented in Section
\ref{ssec:proof_prop_corr_1} of the supplementary material.  Example
\ref{eg:Lambda-plus} below demonstrates that the set of minimizers of
the distributionally robust formulation (\ref{Wass_DRO_A}) is not
necessarily unique and that the set $\Lambda_{\delta}(P_{n})$ may not contain Distributionally robust solutions.
Theorem~\ref{minmax_prop} indicates that the right-limit
$ \Lambda^+_{\delta}(P_{n})$  contains a distributionally robust solution. Theorem
  \ref{thm:levelsets-master} implies that the minimizers of
  (\ref{Wass_DRO_A}) differ by at most $o_p(n^{-1/2})$ in magnitude,
  which indicates that they are asymptotically equivalent and the
  inclusion of one solution of (\ref{Wass_DRO_A}) in
  $\Lambda^+_{\delta}(P_{n})$ is sufficient for the scaling
  considered.

\begin{example}
  Let the loss function
be
\[\ell(x,\beta)=f(\beta)+\{x^2-\log(x^2+1)\} f(\beta-4),\]
where $f(\beta)= 3\beta^2/4-1/8\beta^4 +3/8$ for $\beta\in [-1,1]$,
and $f(\beta)=|\beta|$, otherwise. $\ell(x,\beta)$ is twice-differentiable and convex satisfying Assumptions A1 - A2. Then, if the empirical measure
$P_n$ is a Dirac measure centered at zero with $n=1$, and
$\delta = 1$, we have the distributionally robust estimators
$\beta_{n}^{DRO}(\delta) \in [1,3]$. Further,
$[1,3] \subset \Lambda^+_{\delta}(P_{n})$ but
$[1,3] \cap \Lambda_{\delta}(P_{n}) =  \varnothing$.
\label{eg:Lambda-plus}
\end{example}


Next, we turn to the relationship between $\beta _{n}^{ERM}$ and
$\beta _{n}^{DRO}(\delta _{n}),$ when $\delta _{n}=\eta / n$. From
the first two terms in the triplet, we have,
\begin{align}
  \beta _{n}^{DRO}(\delta _{n})\
  & =\ \beta _{n}^{ERM}-\eta^{1/2} 
    {C^{-1}D_{\beta }S(\beta_{\ast })}{n^{-1/2}}
    + o_{p}\big( n^{-1/2}\big)
    \label{Rel_DRO_ERM} \\
  & =\ \beta _{n}^{ERM}-\delta_n^{1/2}{C^{-1}D_{\beta }S(\beta _{n}^{ERM})}
    + o_{p}\big( \delta_n\big) .  \notag
\end{align}%
The right hand side of (\ref{Rel_DRO_ERM}) points to the canonical
$O\left( n^{-1/2}\right) $ rate of convergence of the Wasserstein distributionally robust
estimator and it can readily be used to construct confidence regions,
as we shall explain in Section \ref{Sec_Confidence_Regions} below.

Relation (\ref{Rel_DRO_ERM}) also exposes the presence of an
asymptotic bias term, namely,
$S(\beta)= [E \{\Vert D_{x}\ell(X;\beta)\Vert_{p}^{2}\}]^{1/2}$, which
points towards selection of optimizers possessing reduced sensitivity
with respect to perturbations in data.  A precise mathematical
statement of this sensitivity-reduction property is given in Corollary
\ref{Cor_Sensitivity} below and its proof is presented in Section
\ref{ssec:proof_prop_corr_1} of the supplementary material.

\begin{corollary}
  \label{Cor_Sensitivity}Suppose that A1 - A2 are in force and
  consider
\begin{equation}
\bar{\beta}_{n}^{DRO} \in \arg\min_{\beta \in B}\left[E_{P_{n}}\left\{ \ell
(X;\beta)\right\} +   n^{-1/2} \left[ \eta E_{P_{n}}\left\{ \Vert
  D_{x}\ell(X;\beta )\Vert_{p}^{2}\right\} \right]^{1/2}\right].
\label{eq:sensitivity-red}
\end{equation}
Then, if $\delta_{n}=\eta/n$, we have that $\beta _{n}^{DRO}(\delta_n) =
\bar{\beta}_{n}^{DRO}+o_{p}(n^{-1/2})$.
\end{corollary}

While the formulation on the right-hand side of
(\ref{eq:sensitivity-red}) is conceptually appealing, it may not be
desirable from an optimization point of view due to the potentially
nonconvex nature of the objective involved. On the other hand, under
Assumption A2, the distributionally robust objective
$\Psi _{n}(\beta )$ is convex; see, for example, the reasoning in
\citet[Theorem 2a]{blanchet2018optimal} while also enjoying the
sensitivity-reduction property of the formulation in
(\ref{eq:sensitivity-red}).

A similar type of result to Corollary \ref{Cor_Sensitivity} is given
in \citet{gaostatlearn}, but the focus there is on the objective
function of (\ref{Wass_DRO_A}) being approximated by a suitable
regularization. The difference between this type of result and
Corollary \ref{Cor_Sensitivity} is that our focus is on the asymptotic
equivalence of the actual optimizers.  Behind a result such as
Corollary \ref{Cor_Sensitivity}, it is key to have a more nuanced
approximation which precisely characterizes the second order term of
size $O(\delta_n)$; see Proposition A1 in the supplementary material.

We conclude this section with results which examine the effects of
relaxing some assumptions made in the statement of Theorem
\ref{thm:levelsets-master} above. Proposition
\ref{prop:constrained-support} below asserts that convergence of the
natural confidence region $\Lambda^+_{\delta_n}(P_n),$ as identified in
Theorem \ref{thm:levelsets-master}, holds even if the support of the
probability distributions in the uncertainty region
$\mathcal{U}_{\delta_n}(P_n)$ is constrained to be a strict subset
$\Omega$ of $\mathbb{R}^d.$ For this purpose, we introduce the
following notation: For any set $C\in \mathbb{R}^{m},$ let
$C^{\epsilon }=\left\{ x\in C: B_{\epsilon }\left( x\right) \subset
  C\right\} ,$ where $B_{\epsilon }\left( x\right) $ is the
neighborhood around $x$ defined as
$B_{\epsilon }\left( x\right) =\left\{ y: \left\Vert y-x\right\Vert
  _{2}\leq \epsilon \right\} .$ Thus, for any probability measure $P$,
we have
$\lim_{\epsilon \rightarrow 0} {P}\left( C^{\epsilon }\right)
={P}\left( C^{\circ }\right).$

\begin{proposition}
  Suppose that Assumptions A1 - A2 are satisfied with {$q\in[1,\infty]$} and
  $E\left(\Vert X \Vert_2^2\right) < \infty$. In addition, suppose that the data
  generating measure $P_*$ satisfies
  $P_{\ast}(\Omega^\circ) =1$. If we take
  $H\sim\mathcal{N}(0,\mathrm{cov}\{h(X,\beta_{\ast})\})$ and
  $\delta_{n}=n^{-\gamma}\eta$ \ for some $\gamma,\eta\in(0,\infty),$
  then the following convergence holds as $n \rightarrow \infty$:
  \begin{align*}
    n^{1/2} \left\{ \Lambda_{\delta_n}(P_n) - \beta_\ast \right\}
    \Rightarrow \Lambda_{\eta,\gamma} + C^{-1}H.
  \end{align*}
  \label{prop:constrained-support}
\end{proposition}
The steps involved in proving Proposition
\ref{prop:constrained-support} are presented in Section
\ref{sec:proof}. A discussion on the validity of a central limit theorem
for the estimator $\beta_n^{DRO},$ in the presence of constraints
restricting transportation within the support set $\Omega$, is
presented in Section \ref{sec:discussion}.

In the case where the unique minimizer $\beta_\ast$ may not
necessarily lie in interior of the set $B$ (as opposed to the
requirement in Assumption A2.b, one may obtain the extension in
Proposition \ref{prop:DRO-boundary} as the limiting result for the
estimator $\beta^{DRO}_n(\delta_n)$. As in the previous results, we
take $h(x,\beta) = D_\beta\ell(x;\beta).$ The proof of Proposition
\ref{prop:DRO-boundary} is in Section \ref{ssec:proofs_dro_bound_prop}
of the supplementary material.
\begin{proposition}
  Suppose that Assumptions A1, A2.a, A2.c are satisfied and $\beta_\ast$
  is the unique minimizer of $\min_{\beta \in B} E\{\ell(X,\beta)\}.$
  Suppose that the set $B$ is compact and there exist
  $\varepsilon > 0$ and twice continuously differentiable functions
  $g_i(\beta)$ such that,
  \begin{align*}
    B \cap B_\varepsilon(\beta_\ast) =
    \left\{ \beta \in B_\varepsilon(\beta_\ast): g_i(\beta) = 0,
    i \in I, \ g_j(\beta) \leq 0, j \in J\right\},
  \end{align*}
  where $I,J$ are finite index sets and $g_i(\beta_\ast) = 0$ for all
  $i \in J.$ With this identification of the set $B,$ suppose that the following so-called Mangasarian-Fromovitz constraint qualification is satisfied at $\beta_\ast$: the
  gradient vectors $\{D g_i(\beta_\ast): i \in I\}$ are linearly
  independent and there exists a vector $w$ such that
  $w^\T D g_i(\beta_\ast) = 0$ for all $i \in I$ and
  {$w^\T D g_j(\beta_\ast) < 0$} for all $j \in J.$

  Suppose that $\Lambda_0$ is the set of Lagrange multipliers satisfying
  the first-order optimality conditions and the following
  second-order sufficient conditions: $\lambda \in \Lambda_0$ if and
  only if $D_\beta L(\beta_\ast,\lambda) = 0,$ $\lambda_i \geq 0$ for
  $i \in J,$ and
  $\max_{\lambda \in \Lambda_0} w^\T D_{\beta \beta}
  L(\beta_\ast,\lambda)w > 0$ for all $w \in \mathcal{C},$ where
  \[L(\beta,\lambda) = E\{\ell(X,\beta)\} + \sum_{i \in I \cup J}
    \lambda_i g_i(\beta)\] is the Lagrangian function associated with
  the minimization $\min_{\beta \in B} E\{\ell(X,\beta)\}$ and
  \[ \mathcal{C} = \left\{ w: w^\T Dg_i(\beta_\ast) = 0, i \in I, \
      w^\T Dg_j(\beta_\ast) \leq 0, j \in J, \ w^\T E\{h(X,\beta_\ast)\}
      \leq 0 \right\}\]
  is the non-empty cone of critical directions.  In addition, suppose
  that $\omega(\xi)$ is the unique minimizer of
  $\min_{u \in \mathcal{C}}\big\{\xi^\T u + 2^{-1}q(u) \big\},$ where
  $q(u) = \max \left\{u^\T D_{\beta \beta}L(\beta_\ast,\lambda) u :
    \lambda \in \Lambda_0 \right\}.$ Then, if $\delta_n = \eta n^{-1}$
  for $\eta \in (0,\infty)$,
  $E \{\Vert h(X,\beta_\ast) \Vert_2^2\} < \infty$ and
  $E\{D_\beta h(X,\beta_\ast)\} \succ 0,$ we have the following
  convergence as $n \rightarrow \infty$:
  \begin{align*}
    n^{1/2}\left\{\beta^{DRO}_n(\delta_n)  - \beta_\ast\right\} \Rightarrow
    \omega\left\{-H + \eta^{1/2}D_\beta S(\beta_\ast)\right\},
  \end{align*}
  where $H \sim \mathcal{N}(0,\mathrm{cov}\{h(X,\beta_\ast)\}).$
  \label{prop:DRO-boundary}
\end{proposition}

The Mangasarian-Fromovitz constraint qualification conditions and the
necessary and sufficient conditions in the statement of Proposition
\ref{prop:DRO-boundary} are standard in the literature if the optimal
$\beta_\ast$ lies on the boundary of the set $B$; see, for example,
\cite{shapiro1989}. Please refer the discussion following Theorem 3.1
in \cite{shapiro1989} for sufficient conditions under which
$\omega(\xi)$ is unique.

Proposition \ref{prop:DRO-property} extends the sensitivity reduction
property in Corollary \ref{Cor_Sensitivity} to settings where the
minimizer for $\min_{\beta \in B} E_{P_*}\{\ell(X;\beta)\}$ is not
unique. The proof of Proposition \ref{prop:DRO-property} is presented
in Section \ref{ssec:proofs_dro_bound_prop} of the supplementary
material.

\begin{proposition}
  Suppose that Assumptions A1, A2.a and A2.c are satisfied, the set
  $B$ is compact, and the choice of the radii $(\delta_n: n \geq 1)$ is
  such that $n\delta_n \rightarrow \eta \in (0,\infty).$ Let the set $B_* $ be $\arg\min_{\beta \in B}  E_{P_*}\{\ell(X;\beta)\}$. Then, the distributionally robust optimization
  objective $\Psi_n(\beta)$ satisfies,
  \begin{align}
    n ^{1/2} \left[\Psi_n(\beta) - E\{\ell(X;\beta)\} \right]
    \Rightarrow Z(\beta) + \eta^{1/2} S(\beta),
    \label{rhs-nonunique}
  \end{align}
  where $Z(\cdot)$ is a zero mean Gaussian process with covariance
  function $\text{cov}\{Z(\beta_1),Z(\beta_2)\}
  =\text{cov}\{\ell(X,\beta_1),\ell(X,\beta_2)\}.$ The above
  weak convergence holds, as $n \rightarrow \infty,$ on the space of
  continuous functions equipped with the uniform topology on compact sets.
    %
  Consequently, if
  $\arg \min_{\beta \in B_\ast}\{Z(\beta) +
  \eta^{1/2} S(\beta)\}$ is singleton with probability one, we have as
  $n\rightarrow \infty,$
  \begin{equation*}
    \beta_n^{DRO}(\delta_n) \Rightarrow \textnormal{arg}\,\textnormal{min}_{\beta \in B_\ast}\big\{
     Z(\beta )+ \eta^{1/2} S(\beta ) \big\}.
\end{equation*}%

\label{prop:DRO-property}
\end{proposition}

\subsection{Construction of Wasserstein distributionally robust confidence regions}
  \label{Sec_Confidence_Regions}
As mentioned in the Introduction, for suitably chosen $\delta_n,$ the
set $\Lambda^+_{\delta_{n}}(P_{n})$ represents a natural confidence
region. In particular, $\Lambda^+_{\delta_{n}}(P_{n})$ possesses an
asymptotically desired coverage, say at level at least $1-\alpha$, if
and only if
\begin{equation*}
1-\alpha\leq\lim_{n\rightarrow\infty}\mathrm{pr}\left\{ \beta_{\ast}\in
\Lambda^+_{\delta_{n}}(P_{n})\right\} =\mathrm{pr}[-C^{-1}H\in\{u:\varphi^{%
\ast }(Cu)\leq\eta\}],
\end{equation*}
or, equivalently, if $\eta\geq\eta_{\alpha}$, where $\eta_{\alpha}$ is
the $%
(1-\alpha)$-quantile of the random variable $\varphi^{\ast}(H)$.

Recall the earlier geometric insight describing
$\Lambda^+_{\delta_{n}}(P_{n})$ as a set containing both
$\beta_{n}^{DRO}(\delta_n)$ and $\beta_{n}^{ERM},$ as a consequence of
Theorem \ref{minmax_prop}.  Following this, if we let
$\eta\geq\eta_{\alpha},$ we then have,
\begin{align*}
  \lim_{n\rightarrow\infty}\mathrm{pr}\big\{\beta_{\ast}\in\Lambda^+_{\delta_{n}}(P_{n}),\
  \beta_{n}^{DRO}\in\Lambda^+_{\delta_{n}}(P_{n}),\
  \beta_{n}^{ERM}\in\Lambda^+_{\delta_{n}}(P_{n})\big\}
  &= \lim_{n\rightarrow\infty}\mathrm{pr} \big\{ \beta_{\ast}\in\Lambda^+_{\delta_{n}}(P_{n})\big\}\\
  &\geq  1-\alpha,
\end{align*}
which presents the picture of $\Lambda^+_{\delta_n}(P_n)$ as a
confidence region simultaneously containing $\beta_\ast,\beta_n^{ERM}$
and $\beta_{n}^{DRO}(\delta_n),$ with a desired level of confidence.

The function $\varphi^{\ast}(H)$ can be computed in closed form in
some settings. But, typically, computing $\varphi^{\ast}(\cdot)$ may
be challenging. We now describe how to obtain a consistent estimator
for $\eta_{\alpha}$. Define the empirical version of $\varphi (\xi )$,
namely%
\begin{equation*}
  \varphi _{n}(\xi ) = \frac{1}{4}E_{P_{n}}\left[ \left\Vert \left\{
        D_{x}h(X,\beta _{\ast })\right\} ^{ \mathrm{\scriptscriptstyle T} }\xi
    \right\Vert _{p}^{2}\right] = \frac{1}{4n}\sum_{i=1}^{n}\left\Vert \left\{
      D_{x}h(X,\beta _{\ast })\right\} ^{ \mathrm{\scriptscriptstyle T} }\xi
  \right\Vert_{p}^{2},
\end{equation*}%
and the associated empirical convex conjugate,
$\varphi _{n}^{\ast }(\zeta )=\sup_{\xi \in \mathbb{R}^{d}}\left\{ \xi
  ^{{ \mathrm{\scriptscriptstyle T} }}\zeta -\varphi _{n}(\xi
  )\right\}.$
Proposition \ref{Prop_Reg_Plug} below, whose proof is in Section
\ref{ssec:proof_prop_reg} of the supplementary material, provides a
basis for computing a consistent estimator for $\eta _{\alpha }$.

\begin{proposition}
  \label{Prop_Reg_Plug}Let ${\Xi}_{n}$ be any consistent estimator of
  $%
  \text{cov}\{ h\left( X,\beta\right) \} $ and write $\bar{\Xi}_{n}$
  for any factorization of $\Xi_{n}$ such that
  $\bar{\Xi}_{n}\bar{\Xi}_{n}^{\T}=\Xi_{n}$%
  . Let $Z$ be a $d$-dimensional standard Gaussian random vector
  independent of the sequence $\left( X_{n}:n\geq1\right) $. Then, i)
  the distribution of $%
  \varphi^{\ast}(Z)$ is continuous, ii)
  $\varphi_{n}^{\ast }(\cdot)\Rightarrow\varphi^{\ast}(\cdot)$ as
  $n\rightarrow\infty$ uniformly on compact sets, and iii)
  $\varphi_{n}^{\ast}(\bar{\Xi}_{n}Z)\Rightarrow\varphi ^{\ast}(H)$.
\end{proposition}

Given the collection of samples $\left\{ X_{i}\right\} _{i=1}^{n}$, we
can generate independent and identically distributed copies of $Z$ and
use Monte Carlo to estimate the quantile
$\left( 1-\alpha\right) $-quantile, $%
\eta_{\alpha}\left( n\right) $, of
$\varphi_{n}^{\ast}(\bar{\Xi}_{n}Z)$. The previous proposition implies
that
$\eta_{\alpha}\left( n\right) =\eta_{\alpha}+o_{p}\left( 1\right),$ as
$n\rightarrow\infty$ This is sufficient to obtain an implementable
expression for $\beta_n^{DRO}\{\eta_\alpha(n)/n\}$ which is
asymptotically equivalent to (\ref{Rel_DRO_ERM}), as it differs only
by an error of maginutude $o_{p}(n^{-1/2}).$

Next, we provide rigorous support for the approximation
\begin{equation*}
\Lambda^+_{\delta_{n}}(P_{n})\approx\beta_{n}^{ERM}+n^{-1/2}\Lambda_{\eta},
\end{equation*}
which can be used to approximate $\Lambda^+_{\delta_{n}}(P_{n})$, providing we
can estimate $\Lambda_{\eta}$.

\begin{corollary}
\label{cor:CIcenter-plugin}Under the assumptions of Theorem \ref%
{thm:levelsets-master} and $\gamma = 1$, we have
\begin{equation*}
n^{1/2}\left\{ \Lambda^+_{\delta_{n}}(P_{n})-\beta_{n}^{ERM}\right \}
\Rightarrow\Lambda_{\eta}.
\end{equation*}
Moreover, if $\eta\left( n\right) =\eta+o\left( 1\right) $, and $%
C_{n}\rightarrow C$, then
\begin{equation*}
\Lambda_{\eta\left( n\right) }^{n}=\{u:\varphi_{n}^{\ast}(C_{n}u)\leq
\eta\left( n\right) \}\rightarrow\Lambda_{\eta}.
\end{equation*}
\end{corollary}

\begin{proof}[Proof  of Corollary \ref{cor:CIcenter-plugin}]
By Following directly from Theorem \ref{thm:levelsets-master} and an application
of continuous mapping theorem as in,
\begin{equation*}
n^{1/2}\left\{ \Lambda^+_{\delta_{n}}(P_{n})-\beta_{n}^{ERM}\right\}
=n^{1/2}\left\{ \Lambda^+_{\delta_{n}}(P_{n})-\beta_{\ast}\right\} -
n^{1/2}\left\{ \beta_{n}^{ERM}-\beta_{\ast}\right\} \Rightarrow
\Lambda_{\eta}+C^{-1}H-C^{-1}H.
\end{equation*}
The second part of the result follows from the regularity results in
Proposition \ref{Prop_Reg_Plug}.
\end{proof}

The next result,  as we shall explain, allows us to obtain
computationally efficient approximations of the set $\Lambda_{\eta}.$
A completely analogous result can be used to estimate
$\Lambda_{\eta\left( n\right) }^{n}$, simply replacing
$\varphi^{\ast}(\cdot)$, $\varphi(\cdot)$ and $C$ by $%
\varphi_{n}^{\ast}(\cdot)$, $\varphi_{n}(\cdot)$ and $C_{n}.$

\begin{proposition}
The support function of the convex set $\Lambda _{\eta }=\{u:\varphi ^{\ast
}(Cu)\leq \eta \}$ is,
\begin{equation*}
h_{\Lambda _{\eta }}(v)=2\{\eta \varphi (C^{-1}v)\}^{1/2},
\end{equation*}%
where the support function of a convex set $A$ is defined as $h_{A}(x)=\sup
\{x\cdot a:a\in A\}.$ \label{prop:level-sets-char}
\end{proposition}
The proof of Proposition \ref{prop:level-sets-char} is in Section \ref{ssec:proof_prop_reg} of the supplementary material.
\begin{remark}
\textnormal{\ Proposition \ref{prop:level-sets-char} can be used to obtain a
tight envelope of the set $\Lambda _{\eta }$ by evaluating an intersection
of hyperplanes that enclose $\Lambda _{\eta }$. Recall from the definition
of support function that
\begin{equation*}
\Lambda _{\eta }=\cap _{u}\{v:u\cdot v\leq h_{\Lambda _{\eta }}\left(
u\right) \}.
\end{equation*}%
Therefore for any $u_{1},...,u_{m},$ we have
$\Lambda _{\eta }$ is contained in
$\cap _{u_{1},...u_{m}}\{v:u_{i}\cdot v\leq h_{\Lambda _{\eta }}\left(
  u_{i}\right) \},$
and
$\Lambda _{\eta \left( n\right) }^{n}$ is contained in
$\cap _{u_{1},...u_{m}}\{v:u_{i}\cdot v\leq h_{\Lambda _{\eta \left(
      n\right) }^{n}}\left( u_{i}\right) \}.$
}
\end{remark}

\section{Numerical examples: Geometry and coverage probabilities}

\label{sec:geo_insignts} 

\subsection{Distributionally robust linear regression}
We first offer a brief introduction to the distributionally robust
version of the linear regression problem considered in
\citet{blanchet2016robust}. Specifically, the data is generated by
$Y=\beta_{\ast}^{{ \mathrm{\scriptscriptstyle T} }}X+\epsilon,$ where
$X\in\mathbb{R}^{d}$ and $%
\epsilon$ are independent, $C=E(XX^{{\mathrm{\scriptscriptstyle T%
    } }}) $ and $\epsilon\sim\mathcal{N}(0,\sigma^{2})$. We consider
square loss
$\ell(x,y;\beta)= 1/2(y-\beta^{{\mathrm{\scriptscriptstyle T}
  }}x)^{2}$ and take the cost function
$c:\mathbb{R}^{d+1} \times \mathbb{R}%
^{d+1} \rightarrow [0,\infty]$ to be
\begin{equation}
c\{(x,y),(u,v)\}=\left\{
\begin{array}{c}
\left\Vert x-u\right\Vert _{q}^{2} \\
\infty%
\end{array}
\right.
\begin{array}{c}
\text{if }y=v, \\
\text{otherwise.}%
\end{array}%
\label{tr-cost-LinReg}
\end{equation}
Then, from \citet[Theorem 1]{blanchet2016robust}, we have
\begin{equation}
  \min_{\beta\in\mathbb{R}^{d}}\sup_{P:D_{c}(P,P_{n})\leq\delta_{n}}E_{P}\left[
    \ell(X,Y;\beta)\right] = \frac{1}{2}\min_{\beta\in\mathbb{R}^{d}}\big[ E_{P_{n}}
  \left\{ (Y-\beta^{{\mathrm{\scriptscriptstyle T} }}X)^{2}\right\}^{1/2}
  +\delta_{n}^{1/2}\left\Vert \beta\right\Vert _{p}\big]^{2},
\label{Sqrt_Lasso_Est}
\end{equation}
where $p$ satisfies $1/p+1/q=1.$ Following Corollary
\ref{cor:CIcenter-plugin}, an approximate confidence region is
\begin{equation*}
\Lambda^+_{\delta_{n}}(P_{n})\approx n^{-1/2}\Lambda_{\eta_{\alpha}}+\beta
_{n}^{ERM},
\end{equation*}
where
$\Lambda_{\eta_\alpha}=\{\theta:\varphi^{\ast}( C\theta) \leq
\eta_{\alpha}\},$
$\varphi(\xi)= 4^{-1}E \{\Vert e\xi-( \xi^{{\mathrm{%
      \scriptscriptstyle T} }}X) \beta_\ast\Vert_{p}^{2}\},$ the
constant $\eta_\alpha$ is such that
$\mathrm{pr} \{\varphi^\ast (H)\leq1-\alpha\}=\eta_{\alpha}$ for
$H\sim\mathcal{N}(0,C\sigma^{2}),$ and $\delta_{n}=\eta_{\alpha}/n.$
By performing a change of variables via linear transformation in the
analysis of the case $c(x,y) = \Vert x - y \Vert_2^2,$ Theorem \ref%
{thm:levelsets-master} can be directly adapted to the choice $c(x,y)$
being a Mahalanobis metric as in,
\begin{align}
c\left( x,y\right) =\left( x-y\right) ^{{\mathrm{\scriptscriptstyle T} }%
}A\left( x-y\right),  \label{Mahalanobis-metric}
\end{align}
for some matrix $A \succ 0.$ The respective
$\Lambda_{\eta_\alpha} = \{\theta:\varphi^{\ast}( C\theta) \leq
\eta_{\alpha}\}$ is computed in terms of
\begin{equation*}
  \varphi(\xi)=4^{-1}E\big\{ \Vert
  \xi^{{\mathrm{\scriptscriptstyle T} }}
  D_{x}h(X,\beta_{\ast })A^{-1/2}\Vert _{2}^{2}\big\}.
\end{equation*}
For the choice $c\left( x,y\right) =\left( x-y\right) ^{{\mathrm{%
\scriptscriptstyle T} }}A\left( x-y\right), $ the relationship between distributionally robust
and regularized estimators, as in (\ref{Sqrt_Lasso_Est}), is
\begin{equation*}
  \min_{\beta\in\mathbb{R}^{d}}\sup_{P:D_{c}(P,P_{n})\leq\delta_{n}}E_{P}\left\{
    l(X,Y;\beta)\right\} = \frac{1}{2}\min_{\beta\in\mathbb{R}^{d}}\left[
    E_{P_{n}}
    \left\{ (Y-\beta^{{\mathrm{\scriptscriptstyle T} }}X)^{2}\right\}^{1/2}
    +\delta_{n}^{-1/2}\big\Vert A^{-1/2}\beta\big\Vert _{2}\right] ^{2}.
\end{equation*}
See \citet{blanchet2017data} for an account of improved out-of-sample
performance resulting from Mahalanobis cost choices.



\subsection{Shape of confidence regions}

The goal of this section is to provide some numerical implementations
to gain intuition about the geometry of the set $\Lambda_{\eta}$ for
different transportation cost choices. We use the empirical set
\begin{equation*}
  \Lambda_{\eta_{\alpha}}^n=\{\theta:{\varphi}^{\ast}_n( {C}_n
  \theta) \leq n^{-1/2}\tilde{\eta}_{\alpha}\},
\end{equation*}
to approximate the desired confidence region as in Corollary
\ref{cor:CIcenter-plugin}. In the above expression,
${\varphi}_n(\xi)= 4^{-1}E_{P_{n}} \{\Vert e\xi-\left(
  \xi^{\mathrm{\scriptscriptstyle T} } X\right)
\beta_{n}^{ERM}\Vert_{p}^{2}\} $, $\eta_{\alpha}(n)$ is such that
$\mathrm{pr}(\tilde{\varphi}^{\ast}_n(H)\leq1-\alpha)=\eta_{\alpha}(n)$ for
$H\sim%
\mathcal{N}(0,{C}_n{\sigma}_n^{2}),$
${C}_n=E_{P_{n}}\left[ XX^{ \mathrm{\scriptscriptstyle T}
  }\right] $, and
${\sigma}^{2}_n =E_{P_{n}}[ \{(Y-( \beta_{n}^{ERM})^{\T } X\}^{2}] $.

In the following numerical experiments, the data is sampled from a linear
regression model with parameters $\sigma^{2}=1$, $\beta_{*}=[0.5,0.1]^{{
\mathrm{\scriptscriptstyle T} }},n=100$ and
\begin{equation}
X\sim\mathcal{N}\left( 0,\left[
\begin{array}{cc}
  1 & \rho \\
  \rho & 1%
\end{array}
\right] \right),
\label{eq:cov-matrix}
\end{equation}
with $\rho = 0.7.$ In Figures \ref{1-norm}-\ref{inf-norm}, we draw the
$95\%$ confidence region corresponding to the \sloppy{choices
  $p=1,3/2,2,3,\infty $,} ($%
q=\infty,3,2,3/2,$ respectively) by means of support functions defined
in Proposition %
\ref{prop:level-sets-char}. In addition, a confidence region for
$\beta _{\ast }$ resulting from the asymptotic normality of the
least-squares estimator,
${n}^{1/2}(\beta _{n}^{ERM}-\beta ^{\ast })\Rightarrow \mathcal{N}%
(0,C^{-1}\sigma ^{2}),$
is
\begin{equation*}
\Lambda _{CLT}(P_{n})=n^{-1/2}\{\theta :\theta ^{{\mathrm{\scriptscriptstyle %
T}}}C\theta /\sigma ^{2}\leq \chi _{1-\alpha }^{2}(d)\}+\beta _{n}^{ERM},
\end{equation*}%
where $\chi _{1-\alpha }^{2}(d)$ is the $1-\alpha $ quantile of the
chi-squared distribution with $d$ degrees of freedom. One can select the
matrix $A$ in the Mahalanobis metric \eqref{Mahalanobis-metric} such that
the resulting confidence region coincides with $\Lambda _{CLT}(P_{n})$.
Namely, $A$ is chosen by solving the equation%
\begin{equation}
E\left[ \left\{ e\xi -\left( \xi ^{{\mathrm{\scriptscriptstyle T}}}X\right)
\beta _{\ast }\right\} A^{-1}\left\{ e\xi -\left( \xi ^{{\mathrm{%
\scriptscriptstyle T} }}X\right) \beta _{\ast }\right\} ^{{\mathrm{%
\scriptscriptstyle T} }}\right] =C\sigma ^{2}.  \label{eqn_A2}
\end{equation}

Figure \ref{clt} gives the confidence region for the choice $p=2$ and
$%
\Lambda_{CLT}(P_n)$ superimposed with various distributionally robust
minimizer along with the empirical risk minimizer.
It is evident from the figures that $p=1$ gives a diamond shape, $p=2$
gives an elliptical shape and $p=\infty$ gives a rectangular shape.
Furthermore, we see that the distributionally robust optimization solutions all reside in their
respective confidence regions but may lie outside of the confidence
regions of other norms.
\begin{figure}[htbh]
\centering
\subfigure[$p=1$]{
\label{1-norm} \includegraphics[width=1.81in]{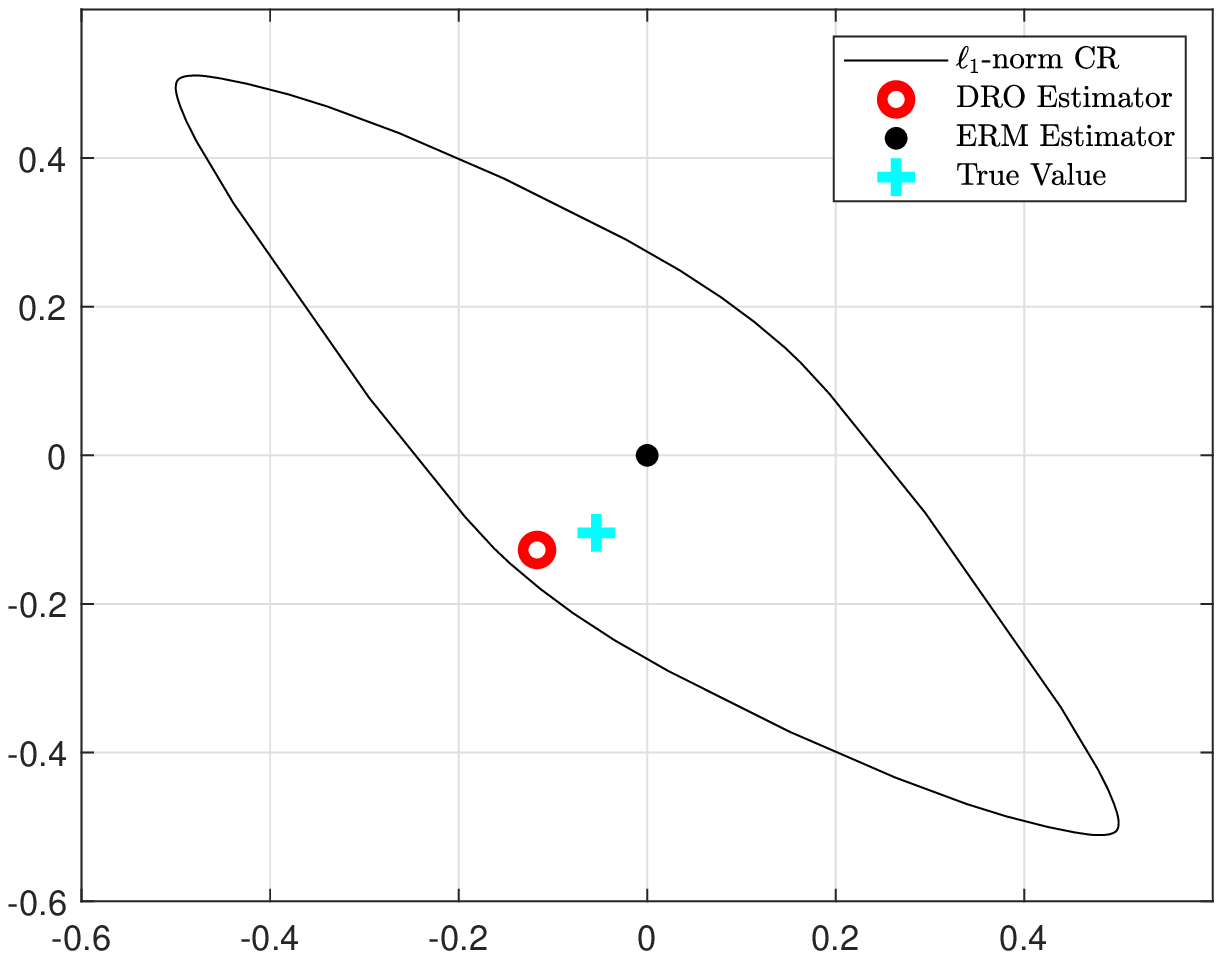}}
\subfigure[$p=1.5$]{
\label{2-norm} \includegraphics[width=1.81in]{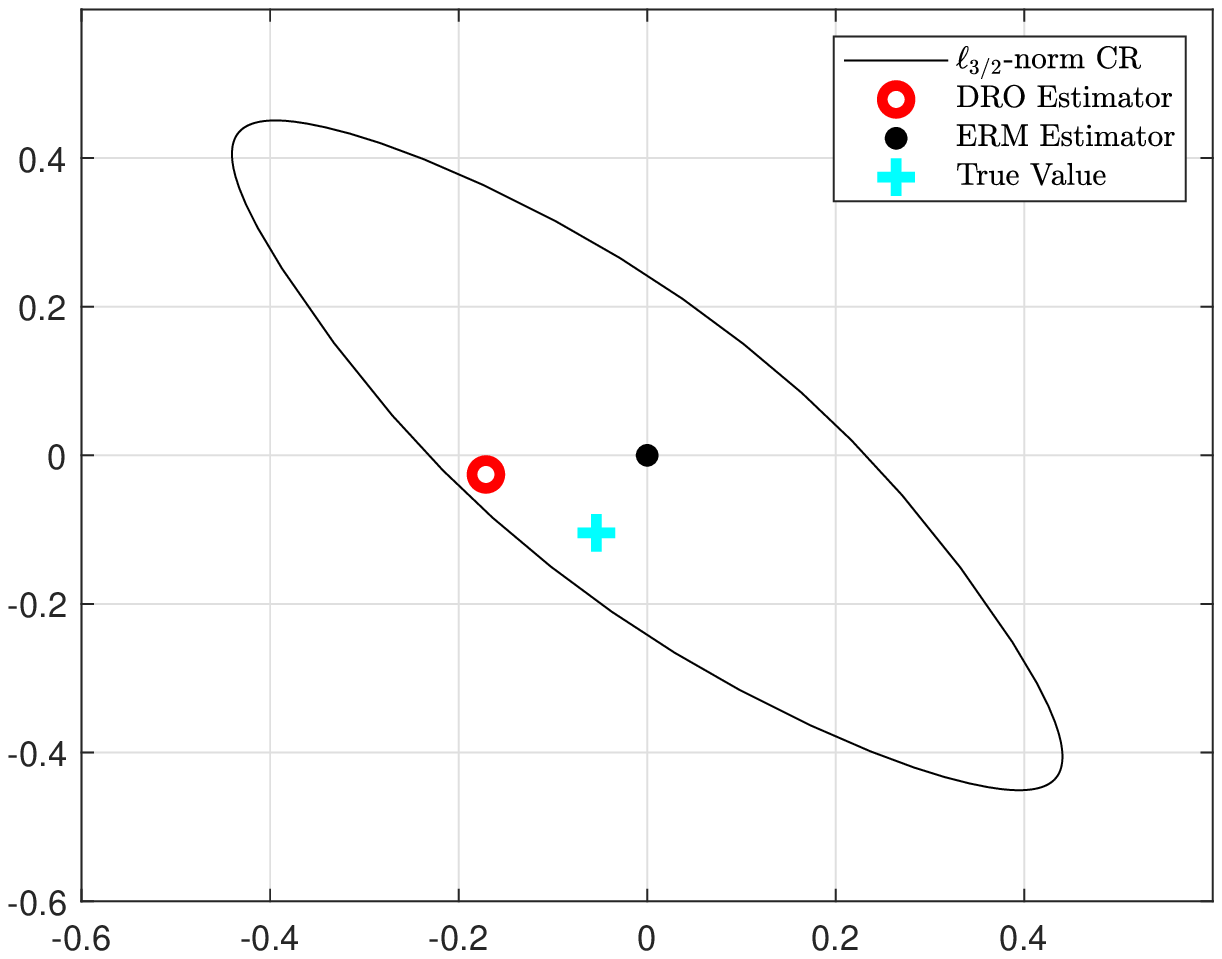}}
\subfigure[$p=2$]{
\label{2-norm} \includegraphics[width=1.81in]{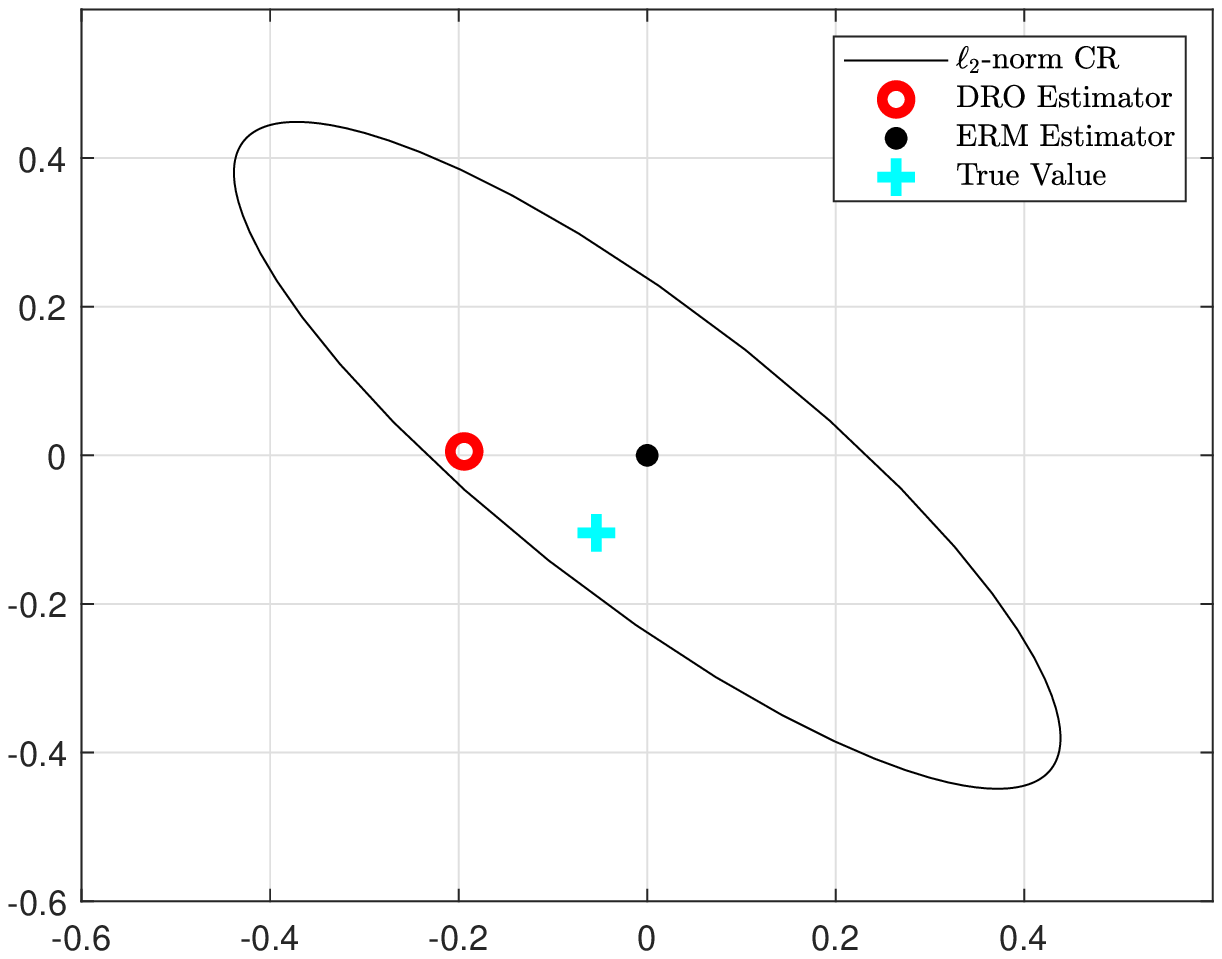}}
\subfigure[$p=3$]{
\label{2-norm} \includegraphics[width=1.81in]{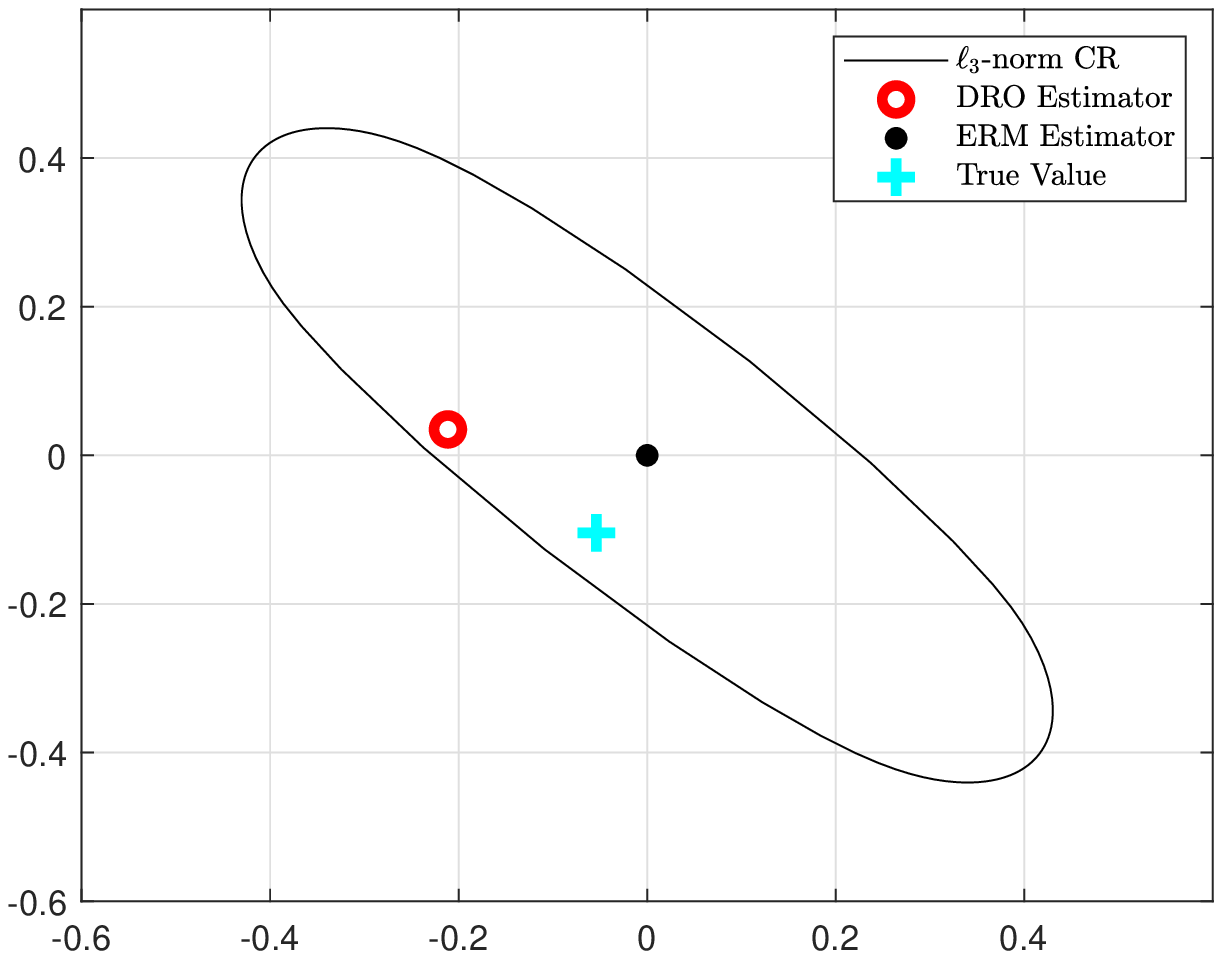}}
\subfigure[$p=\infty$]{
\label{inf-norm} \includegraphics[width=1.81in]{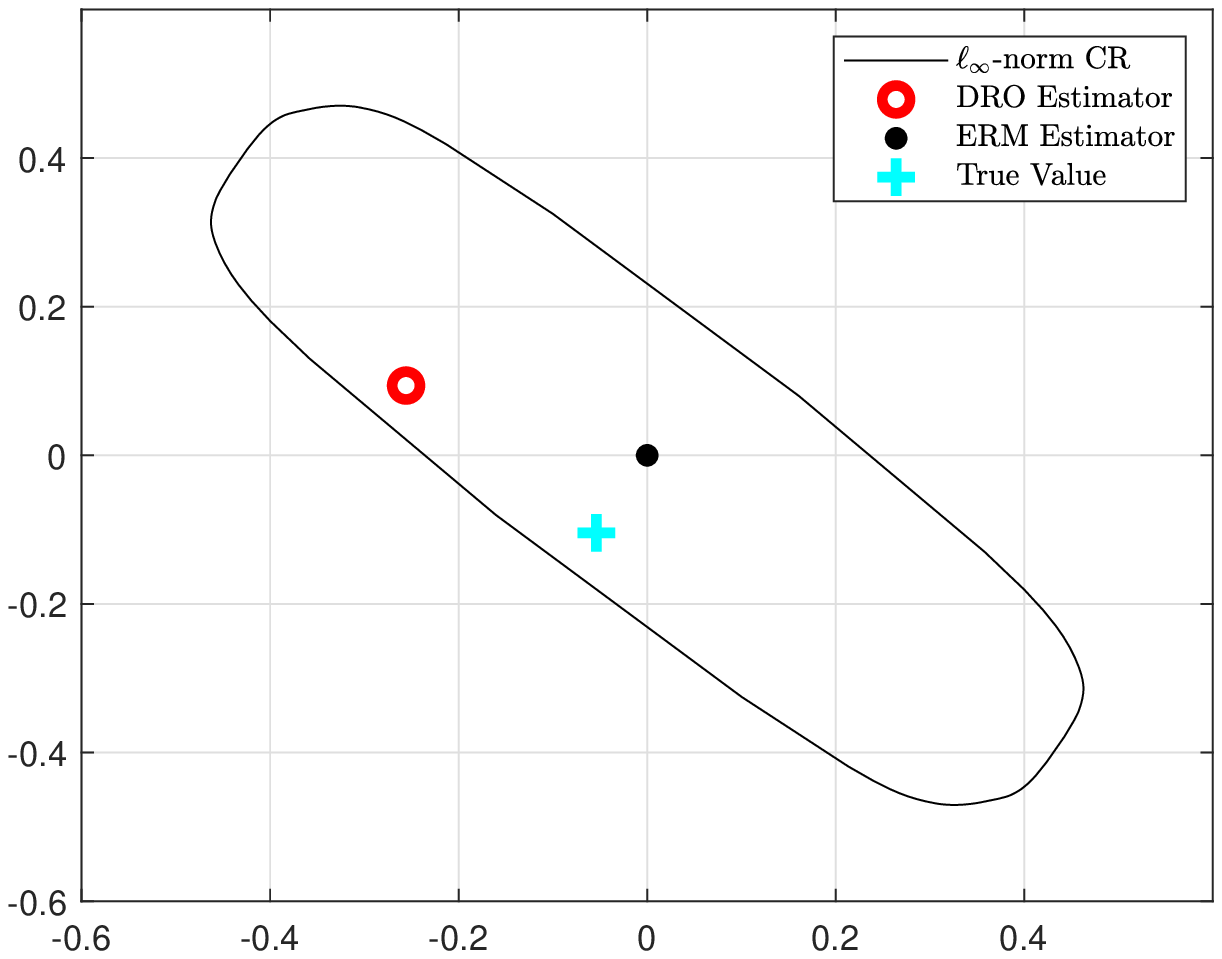}}
\subfigure[CLT]{
\label{clt} \includegraphics[width=1.81in]{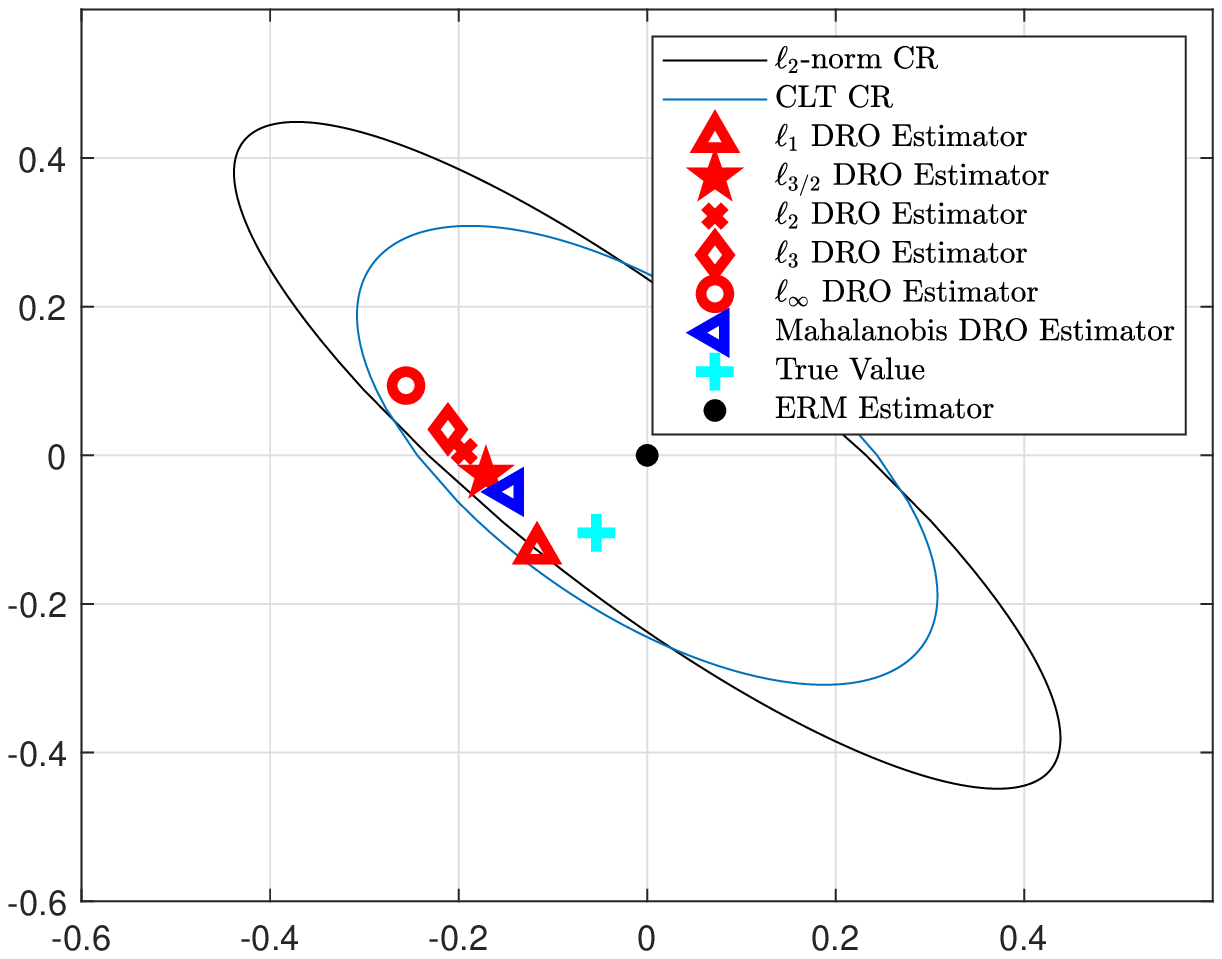}}
\caption{Confidence regions for different norm choices and central
  limit theorem based confidence region plotted together with the
  respective $\beta_n^{DRO}$ estimators and $\beta_n^{ERM}$}.
\end{figure}
We find the induced confidence regions constructed by the Wasserstein
distributionally robust optimization formulations are somewhat similar across the various $l_{p}$
norms, but they are all different to the standard central limit
theorem based confidence region. As noted, the Mahalanobis cost can be
calibrated to exactly match the standard central limit theorem
confidence region.
\subsection{Coverage probabilities and distributionally robust optimization solutions}
In this section, we test the scenario in which the covariates are
highly correlated. Specifically, the data is sampled from a linear
regression model with parameters $\sigma ^{2}=1$, $n=100$, $p=2$. The
random vector $X$ is taken to be distributed in (\ref{eq:cov-matrix}),
considering three different values for $\rho:$ we choose
$\rho = 0.95,0,-0.95.$ We consider the following two cases for the
underlying parameter $\beta_\ast$: $\beta_{*}=[0.5,0.5]^\T$ and
$\beta_{*}=[1,0]^\T.$ In Table \ref{tab:coverage} below, we report the
coverage probabilities of the underlying $\beta _{*}$ and
$\beta_n^{DRO}(\delta_n)$ in both the $\ell_2$-confidence region and
the central limit theorem based confidence regions.  Specifically, we
report the following four probabilities:
 \[
{\rm pr}\{\beta^{DRO}_n \in \Lambda^+_{\delta_{n}}(P_{n})\}, \quad {\rm pr}\{\beta_* \in \Lambda^+_{\delta_{n}}(P_{n})\}, \quad {\rm pr}\{\beta^{DRO}_n \in \Lambda_{CLT}(P_{n})\}, \quad {\rm pr}\{\beta_* \in \Lambda_{CLT}(P_{n})\}.
 \]
 We sample 1000 datasets and report the coverage probabilities in
 Table \ref{tab:coverage}. From Table \ref{tab:coverage}, we observe
 that for $\beta_*$, both the $\ell_2$ confidence region and the
 central limit theorem based confidence region achieve the target 95\%
 coverage. Furthermore, the coverage for the distributionally robust
 estimator of the $\ell_2$ confidence region is 100\%, which validates
 our theory. However, when $\rho = -0.95$ and
 $\beta_* = [0.5,0.5]^\T$, the coverage for the distributionally
 robust estimator in the central limit theorem based confidence region
 is only $75.8 \%$. In this example, the asymptotic results developed
 indicate that this coverage probability converges to zero, when $n$
 tends to infinity.

\begin{table}[htbp]
  \centering
  \caption{Coverage Probability}
    \begin{tabular}{cccccc}
      \toprule
      $\beta_0$    &    $\rho$   & \multicolumn{2}{c}{$\ell_2$-confidence region} & \multicolumn{2}{c}{CLT confidence region} \\
      \midrule
                   &       &  Coverage for $\beta_n^{DRO}$ & Coverage for $\beta_*$ & \multicolumn{1}{l}{Coverage for $\beta_n^{DRO}$} & Coverage for $\beta_*$ \\
      \midrule
      \multirow{3}[0]{*}{$\begin{bmatrix} 0.5 \\ 0.5 \end{bmatrix}$}
                   & 0.95  & 100.0\% & 94.5\% & 99.4\% & 94.6\% \\
                   & 0     & 100.0\% & 94.0\% & 97.1\% & 93.5\% \\
                   & -0.95 & 100.0\% & 94.8\% & 75.8\% & 94.4\% \\

          \midrule
    \multirow{3}[0]{*}{$\begin{bmatrix} 1.0 \\ 0.0 \end{bmatrix}$}
          & 0.95  & 100.0\% & 94.6\% & 93.7\% & 95.4\% \\
          & 0     & 100.0\% & 94.6\% & 100\% & 94.1\% \\
          & -0.95 & 100.0\% & 95.3\% & 91.2\% & 94.9\% \\

          \bottomrule
    \end{tabular}%
  \label{tab:coverage}%
\end{table}%
Figures \ref{fig:scatter_0.5} and \ref{fig:scatter_1} show the scatter
plots of the estimators, $\beta_n^{ERM}$ and $\beta_n^{DRO},$ when the
underlying $\beta_\ast$ takes the values $[0.5,0.5]^\T$ and
$[1,0]^\T$, respectively. In the near-collinearity
  cases where $\rho = 0.95$ or $-0.95$, the lower spreads for the
  distributionally robust estimators reveal their better performance
  over the empirical risk minimizing solutions. The utility of the
  proposed $\ell_2$-confidence region emerges in light of the better
  performance of the distributionally robust estimator $\beta_n^{DRO}$
  and its aforementioned lack of membership in $\Lambda_{CLT}(P_n).$
\begin{figure}[ptbh]
\centering
\subfigure[$\rho=0.95$]{
 \includegraphics[width=1.81in]{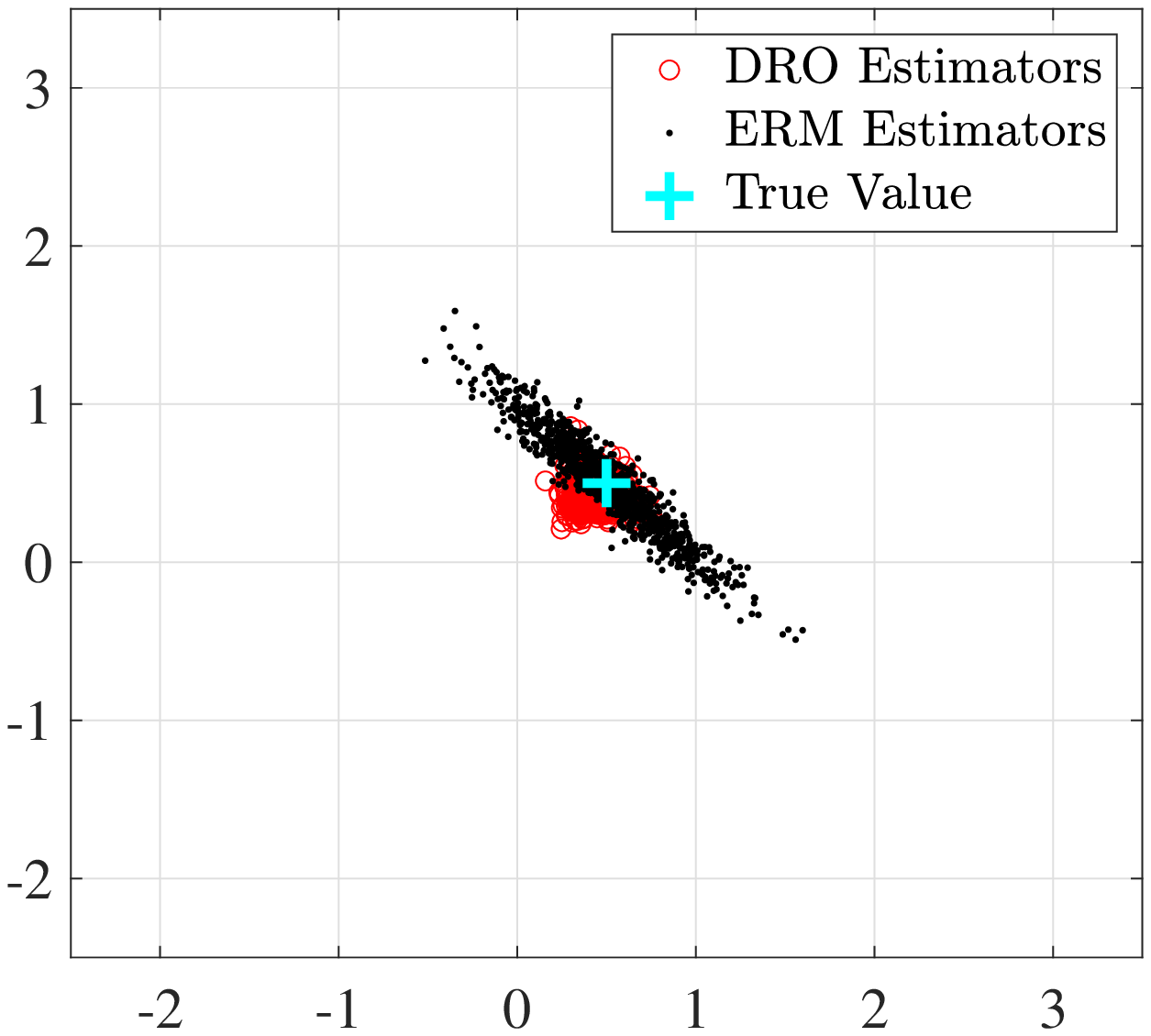}}
\subfigure[$\rho=0$]{
 \includegraphics[width=1.81in]{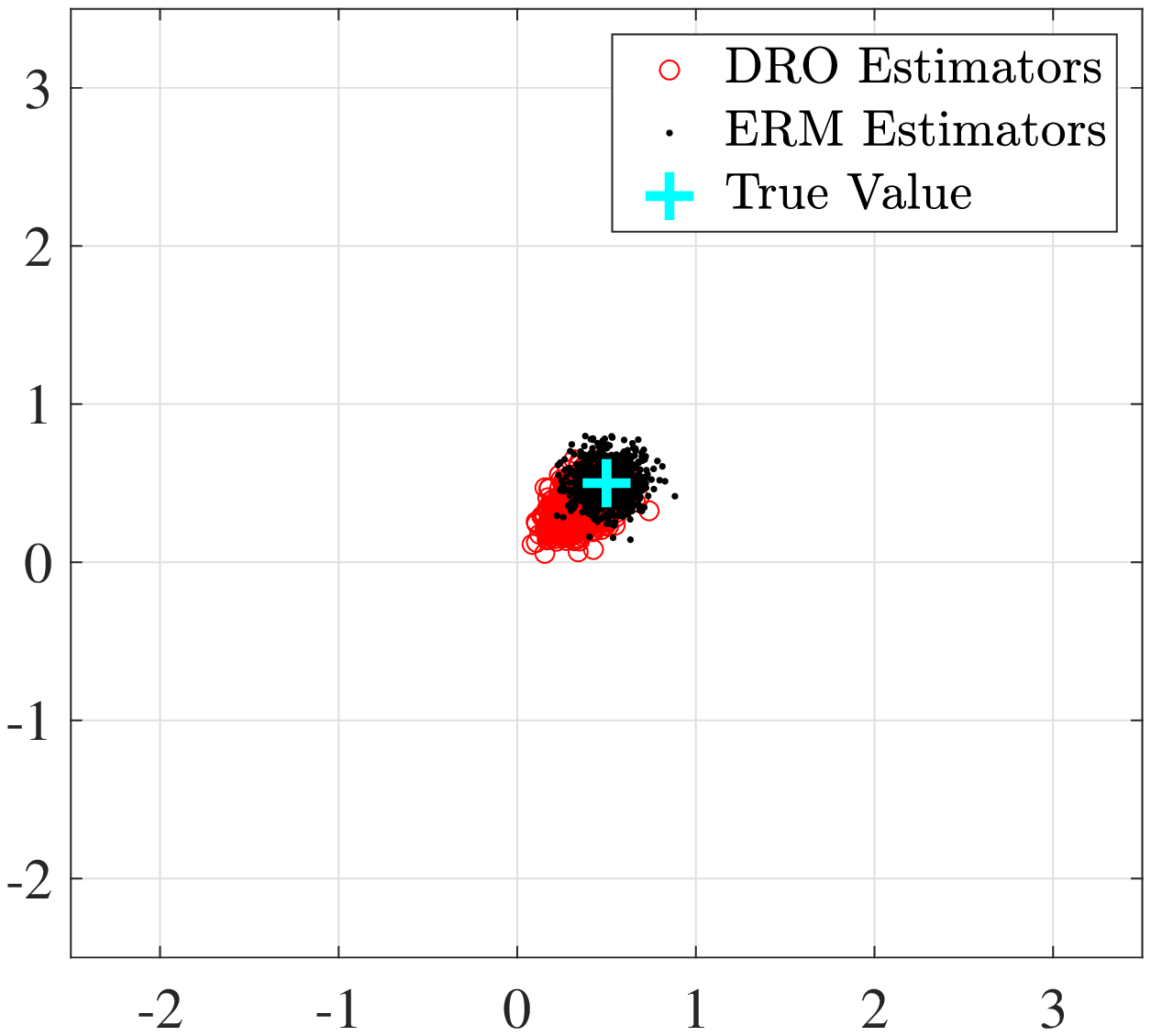}}
\subfigure[$\rho=-0.95$]{
\includegraphics[width=1.81in]{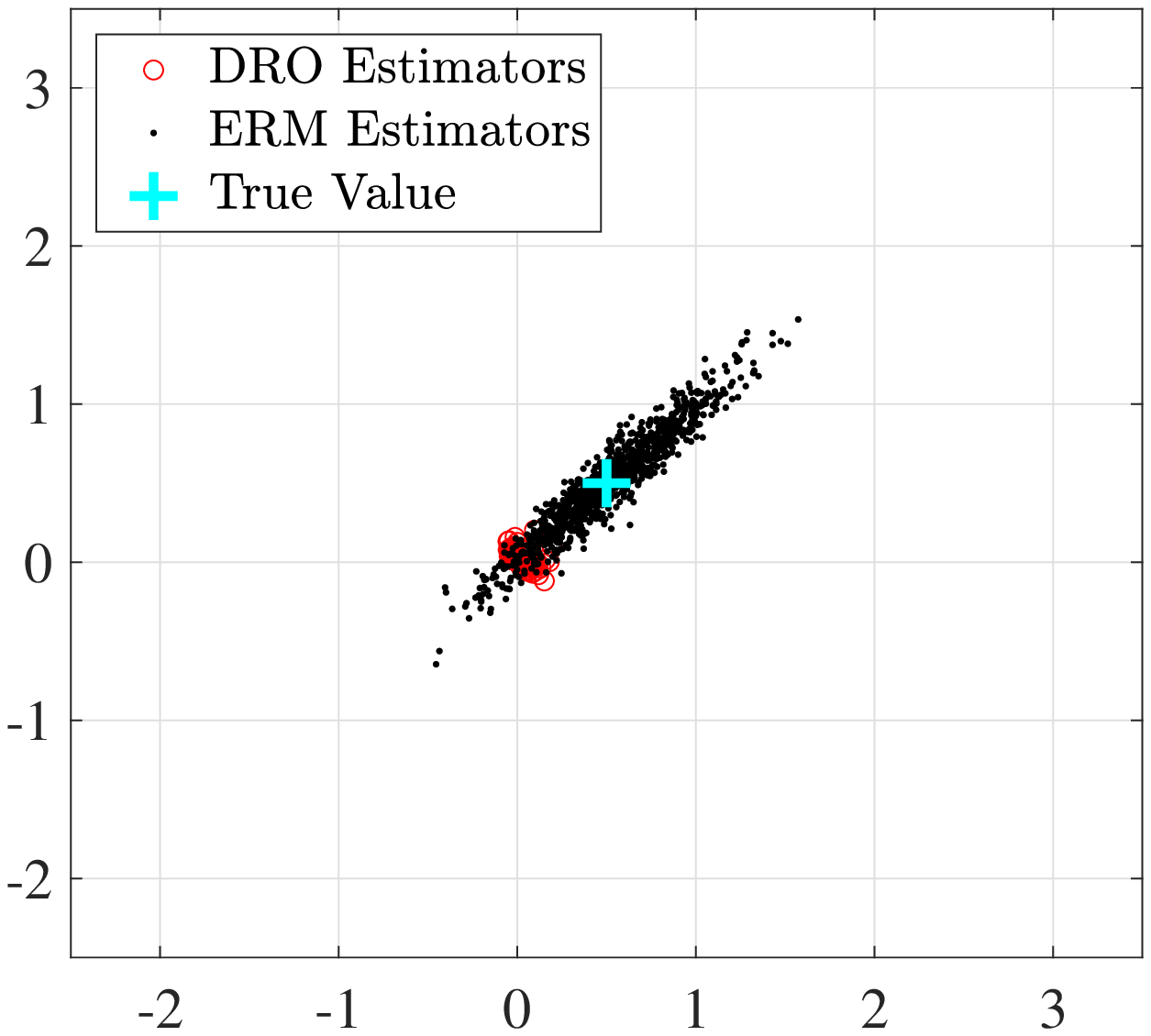}}
\caption{Scatter plots of $\beta_n^{ERM}$ (black circles) and
  $\beta_n^{DRO}$ (red circles) for $\beta_0 = [0.5,0.5]^\T$. }
\label{fig:scatter_0.5}
\end{figure}
\begin{figure}[ptbh]
\centering
\subfigure[$\rho=0.95$]{
 \includegraphics[width=1.81in]{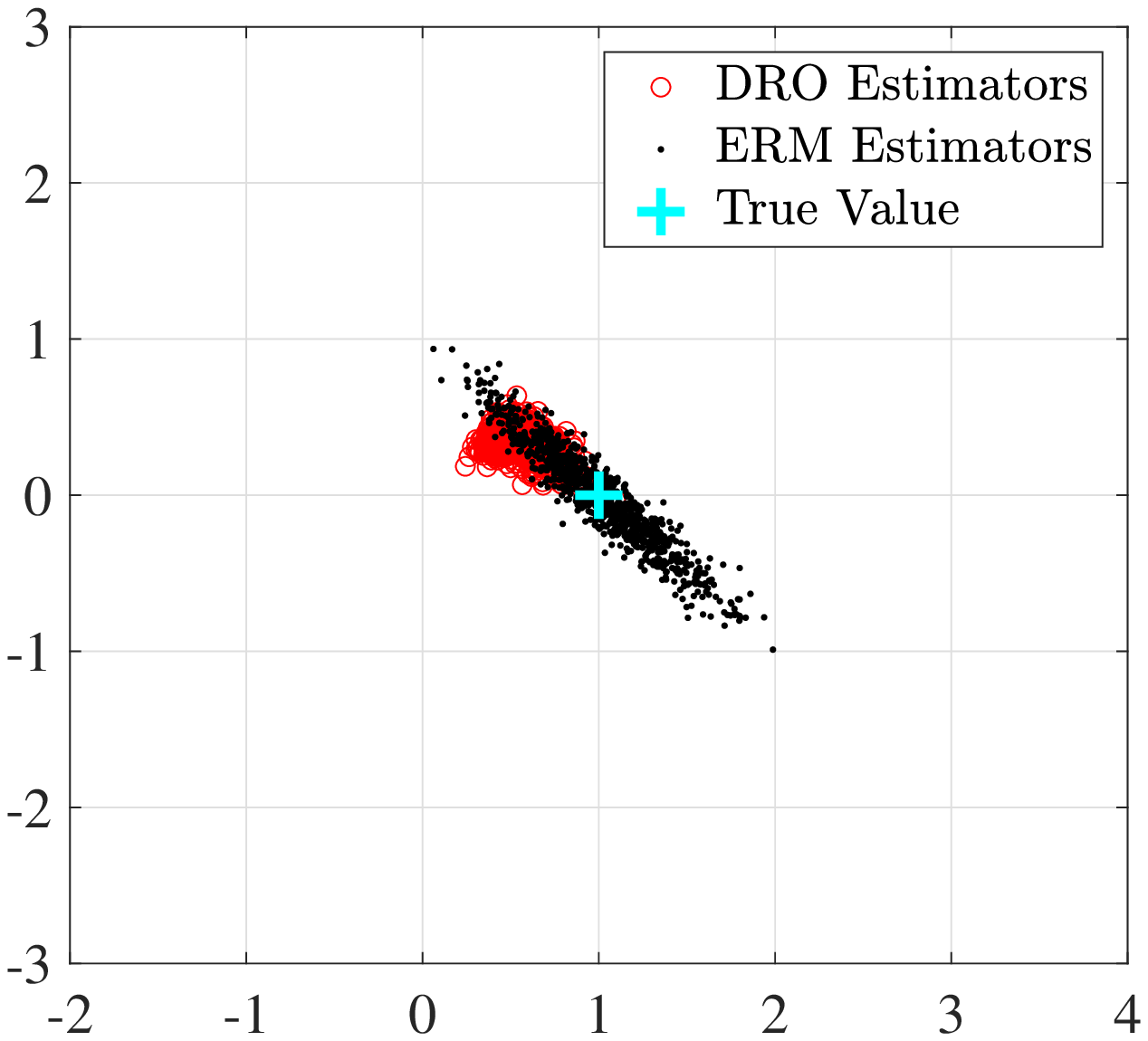}}
\subfigure[$\rho=0$]{
 \includegraphics[width=1.81in]{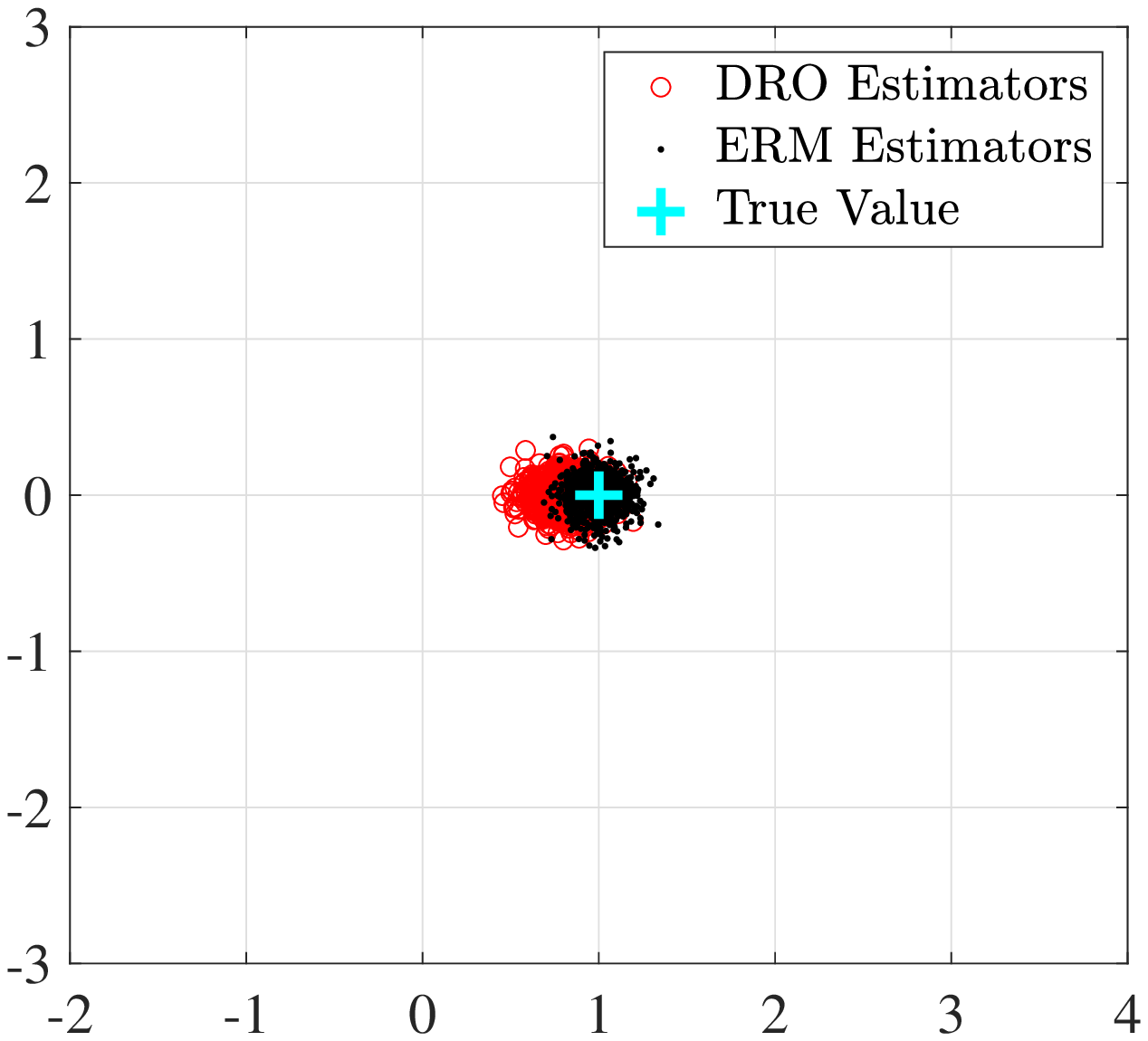}}
\subfigure[$\rho=-0.95$]{
\includegraphics[width=1.81in]{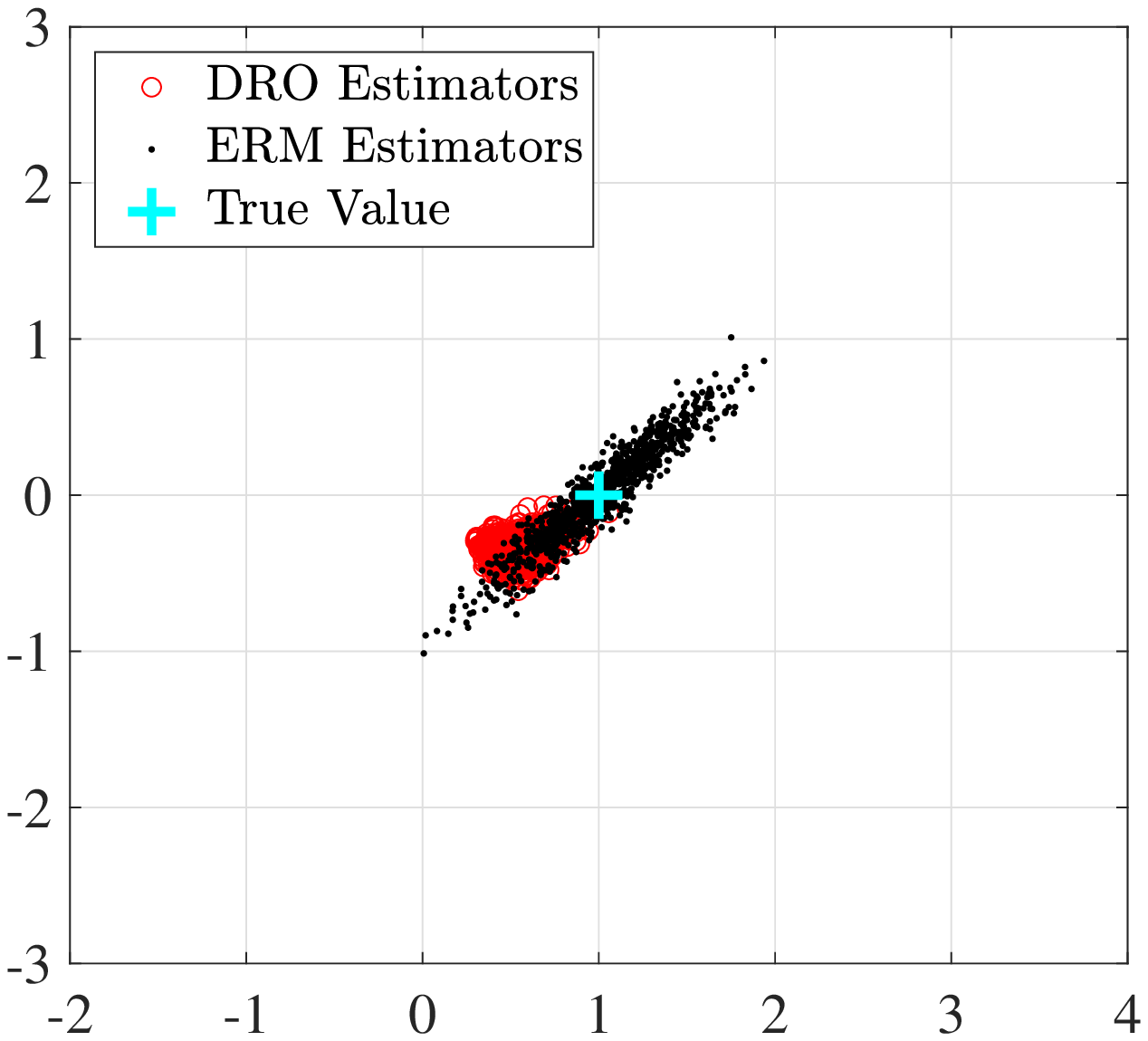}}
\caption{Scatter plots of $\beta_n^{ERM}$ (black circles) and
  $\beta_n^{DRO}$ (red circles) for $\beta_0 = [1.0,0.0]^\T$.}
\label{fig:scatter_1}
\end{figure}
\section{Proofs of main results}

\label{sec:proof} Theorem \ref{thm:levelsets-master} is obtained by
considering appropriate level sets involving auxiliary functionals
which we define next. Following \citet{blanchet2016robust}, we define
the robust Wasserstein profile function, associated with the
estimation of $\beta_\ast$ by solving
$E_{P_n}\{D_\beta h(X,\beta)\} = 0,$ as follows:
\begin{equation*}
R_{n}( \beta) =\inf_{P \in \mathcal{P}(\Omega)} \big[D_c(P,P_{n}) : \beta \in
\arg \min_{\beta \in B }E_{P}\left\{ \ell (X;\beta )\right\} \big].
\end{equation*}
This definition, as noted in \citet{blanchet2016robust}, allows to
characterize the set $\Lambda^+_{\delta}\left( P_{n}\right)$ in terms
of an associated level set; in particular, we have,
\begin{equation}
\Lambda^+_{\delta}( P_{n}) = \mathrm{cl}{\left\{\beta:R_{n}( \beta) \leq\delta \right\}},
\label{CR_based_on_RWP}
\end{equation}
where $\mathrm{cl}(\cdot)$ denotes closure. Indeed, this is because
\begin{equation*}
  \Lambda^+_{\delta}( P_{n})= \mathrm{cl} \big[\cap_{\epsilon>0}\big\{ \beta \in B:\beta \in
  \arg \min_{\beta \in B }E_{P}\{ \ell (X;\beta )\} \text{ for some }P\in
  \mathcal{U}_{\delta _{n}+\epsilon}(P_{n})\big\} \big].
\end{equation*}
If $\beta \in B^\circ$, we have
$R_{n}( \beta) =\inf_{P \in \mathcal{P}(\Omega)}[D_c(P,P_{n}) :E_{P}
\left\{h( X,\beta)\right\} = 0 ].$

Next, for the sequence of radii $\delta _{n}=n^{-\gamma }\eta $, for some
positive constants $\eta ,\gamma $, define functions $V_{n}^{DRO}:\mathbb{R}%
^{d}\rightarrow \mathbb{R}$ and $V_{n}^{ERM}:\mathbb{R}^{d}\rightarrow
\mathbb{R}$, as below, by considering suitably scaled versions of the
distributionally robust and empirical risk objective functions, namely
\begin{align*}
  V_{n}^{DRO}(u)& =n^{\bar{\gamma}}\big \{ \Psi _{n}\big( \beta _{\ast
                  }+n^{-\bar{\gamma}/2}u\big) -\Psi _{n}(\beta _{\ast })\big\} \text{ and }
  \\
  V_{n}^{ERM}(u)& =n\big[ E_{P_{n}}\big\{ \ell (X;\beta _{\ast
                  }+n^{-1/2}u)\big\} -E_{P_{n}}\big\{ \ell (X;\beta _{\ast })\big\} \big],
\end{align*}%
where $\bar{\gamma}=\min \left\{ \gamma ,1\right\} $ is defined in Theorem %
\ref{thm:levelsets-master}. Moreover, define $V:\mathbb{R}^{d}\times \mathbb{%
R}^{d}\rightarrow \mathbb{R}$ via
\begin{equation*}
  V(x,u)=x^{{\mathrm{\scriptscriptstyle T} }}u+ {2}^{-1}u^{{\mathrm{\scriptscriptstyle T} }}Cu.
\end{equation*}%
The following result, as we shall see, can be used to establish Theorem \ref%
{thm:levelsets-master} directly.

\begin{theorem}
  \label{thm:master}
  Suppose that the assumptions made in Theorem
  \ref{thm:levelsets-master} hold. Then we have,
\begin{equation*}
\big\{ V_{n}^{ERM}(\cdot),\ V_{n}^{DRO}(\cdot),\ nR_{n}\big( \beta_{\ast
}+n^{-1/2}\times\cdot\ \big) \big\} \Rightarrow\left\{ V(-H,\cdot ),\
V\{-f_{\eta,\gamma}(H),\cdot\},\ \varphi^{\ast}(H-C\times\cdot\,)\right\} ,
\end{equation*}
on the space $C(\mathbb{R}^{d};\mathbb{R})^{3}$ equipped with the topology
of uniform convergence in compact sets.
\end{theorem}
{ Ensuring  smoothness of $D_{\beta}h(x+\Delta,\beta)$ and $D_{x}h(x+\Delta,\beta)$ around $\beta=\beta_{\ast},$ as in Assumption A2.c, is useful towards investigating the behavior of $nR_{n}\big( \cdot \big)$ in the neighborhood of $\beta^*,$ as required in the third component in the triplet in Theorem \ref{thm:master}.}

\subsection{Proof of Theorem \protect\ref{thm:master}}
Throughout this section, we suppose that the assumptions imposed in
Theorem \ref{thm:levelsets-master} hold. Let
\begin{equation*}
H_{n} = {n}^{-1/2}\sum_{i=1}^{n}h\left( X_{i},\beta_{\ast}\right)
\end{equation*}
The following sequence of results will be useful in proving Theorem
\ref{thm:master} and Proposition
\ref{prop:constrained-support}. Propositions \ref{prop:ERM} and \ref{prop:DRO} hold true for $\Omega = \mathbb{R}^d$; while propositions \ref{prop:RWPEq} - \ref{prop:RWPcont} hold true for general $\Omega$ under the assumption $P_{\ast}(\Omega^\circ) =1$ in Proposition \ref{prop:constrained-support}.

\begin{proposition}
Fix $\alpha\in\lbrack0,1].$ Given $\varepsilon,\varepsilon^{\prime},K>0,$ there exists a
positive integer $n_{0}$ such that
\begin{equation*}
{\rm pr} \left[\big\vert n^{\alpha-1}V_{n}^{ERM}\{n^{( 1-\alpha)
/2}u\}-n^{\alpha/2}H_{n}^{\T}u-2^{-1}u^{\T}Cu\big\vert \leq
\varepsilon^{\prime} \right] \geq 1-\varepsilon,
\end{equation*}
for every $n>n_{0}$ and $\Vert u\Vert_2\leq K.$ Specifically, if $\alpha=1,$
we have%
\begin{equation}
{\rm pr} \left\{\left\vert V_{n}^{ERM}(u)-H_{n}^{\T}u- 2^{-1}u^{\T}Cu\right\vert
\leq\varepsilon^{\prime}\right\} \geq 1-\varepsilon.
\label{eqn:V_ERM}
\end{equation}
\label{prop:ERM}
\end{proposition}
\begin{proposition}
Given $\varepsilon,\varepsilon ^{\prime },K>0,$ there exists a positive integer $n_{0}$
such that
\begin{equation}
{\rm pr} \left\{ \left\vert V_{n}^{DRO}(u)+f_{\eta ,\gamma }(-H_{n})^{\T}u- 2^{-1}
u^{\T}Cu\right\vert \leq \varepsilon ^{\prime } \right\} \geq 1-\varepsilon,
\label{eqn:V_DRO}
\end{equation}%
for every $n>n_{0}$ and $\Vert u\Vert_2 \leq K.$ \label{prop:DRO}
\end{proposition}
\begin{proposition}
  Define the set $\Theta \subset \mathbb{R}^d$ as
\[
  \Theta=\{ \beta \in B^\circ : 0 \in {\rm conv} [\{h(x,\beta)\mid x \in
  \Omega\}]^\circ \},
\]
where ${\rm conv} (S)$ denotes the convex hull of the set $S$.
For $\beta_* + n^{-1/2} u \in \Theta$, We have,
\begin{align*}
n R_{n}\big( \beta_{\ast}+ n^{-1/2}u\big) = \max_{\xi} \left\{ -\xi^{\T}
H_{n} - M_{n}(\xi,u) \right\} ,
\end{align*}
where
\begin{align*}
  M_{n}(\xi,u)
  &= \frac{1}{n} \sum_{i=1}^{n} \max_{\Delta:X_i + n^{-1/2}\Delta \in \Omega} \left\{ \xi^{\T} \int_{0}^{1}
    D_{x}h\big( X_{i} + n^{-1/2}t\Delta, \beta_{\ast}+n^{-1/2} tu
    \big) \Delta {\rm d} t \right .  \\
  &\qquad\qquad + \xi^{\T} \int_{0}^{1}D_{\beta}h\big( X_{i} + n^{-1/2}t\Delta,
    \left. \beta_{\ast}+ n^{-1/2} t u\big) u{\rm d} t - \Vert\Delta\Vert
    _{q}^{2} \right\}.
\end{align*}
Furthermore, there exists a neighborhood of $\beta*$, $B_\epsilon(\beta_*)$ such that $B_\epsilon(\beta_*) \subset \Theta$.
\label{prop:RWPEq}
\end{proposition}

\begin{proposition}
Consider any $\varepsilon ,\varepsilon ^{\prime },K>0.$ Then there exist $%
b_{0}\in (0,\infty )$ such that for any $b\geq b_{0},c_0> 0,\epsilon_0>0$, we
have a positive integer $n_{0}$ such that,
\begin{equation*}
\mathrm{pr}\left[ \sup_{\Vert u\Vert _{2}\leq K}\big \{ nR_{n}\big( \beta
_{\ast }+n^{-1/2}u\big) -f_{up}(H_{n},u,b,c)\big \} \leq \varepsilon
^{\prime }\right] \geq 1-\varepsilon ,
\end{equation*}%
for all $n\geq n_{0},$ and $f_{up}(H_{n},u,b,c)$ equals
\begin{equation*}
  \max_{\Vert \xi \Vert _{p}\leq b}\big\{-\xi ^{\T}H_{n}-%
  E \big[ 4^{-1}\Vert \{ D_{x}h(X,\beta _{\ast
        })\} ^{\T}\xi \Vert _{p}^{2}+\xi ^{\T}D_{\beta }h(X,\beta _{\ast
      })u\big] \mathbb{I}(X \in C_0^{\epsilon_0}) \big\},
\end{equation*}
with $C_0=\{x \in \Omega: \|x\|_p \leq c_0\}.$
\label{prop:RWPUB}
\end{proposition}
\begin{proposition}
  For any $\varepsilon ,\varepsilon ^{\prime },K,b>0,$ there exists a
  positive integer $n_{0}$ such that,
\begin{equation*}
{\rm pr}\left[ \sup_{\Vert u\Vert _{2}\leq K}\big\{ nR_{n}\big( \beta
_{\ast }+n^{-1/2}u\big) -f_{low}(H_{n},u,b)\big\} \geq -\varepsilon
^{\prime }\right] \geq 1-\varepsilon ,
\end{equation*}%
for all $n>n_{0},$ where
\begin{equation*}
  f_{low}(H_{n},u,b)=\max_{\Vert \xi \Vert _{p}\leq b}\big\{-\xi ^{\T}H_{n}-%
  E\left[ 4^{-1}\Vert \left\{ D_{x}h(X,\beta _{\ast })\right\}
    ^{\T}\xi \Vert _{p}^{2}+\xi ^{\T}D_{\beta }h(X,\beta _{\ast })u\right] \big\}.
\end{equation*}
\label{prop:RWPLB}
\end{proposition}

\begin{proposition}
  \label{prop:tightness-1} For any $\varepsilon>0,$ there exist
  constants $a,n_{0} > 0$ such that for every
  $n\geq n_{0},$
\begin{equation*}
{\rm pr}\left\{ nR_{n}(\beta_{\ast})\leq a\right\} \geq1-\varepsilon,
\end{equation*}

\end{proposition}

\begin{proposition}
For any $\varepsilon,\varepsilon^{\prime}, K > 0,$ there exist positive
constants $n_{0},\delta$ such that,
\begin{align*}
\sup_{\underset{\Vert u_{i}\Vert_{2} \leq K}{\Vert u_{1} - u_{2}\Vert_{2}
\leq\delta}} \big\vert n R_{n}\big( \beta_{\ast}+ n^{-1/2}u_{1}\big) - n
R_{n}\big( \beta_{\ast}+ n^{-1/2}u_{2}\big) \big\vert \leq
\varepsilon^{\prime},
\end{align*}
with probability exceeding $1-\varepsilon,$ for every $n > n_{0}.$ \label%
{prop:RWPcont}
\end{proposition}

Proofs of Propositions \ref{prop:ERM} - \ref{prop:RWPcont} are furnished in
Section \ref{sec:proofs-props-master-thm} in the supplementary material. With the statements of these
results, we proceed with the proof of Theorem \ref{thm:master} as follows.

\begin{proof}[Proof  of Theorem \ref{thm:master}]
  Since $E\{h(X,\beta _{\ast })\}=0,$ it follows from central limit
  theorem that
$
H_{n}\Rightarrow -H,
$
where
$H\sim \mathcal{N}(0,E\{h(X,\beta _{\ast })h(X,\beta _{\ast })^{
  \T}\}).$ Since inequalities \eqref{eqn:V_DRO} and \eqref{eqn:V_ERM}
are associated with the same $H_n$ , it follows from Propositions
\ref%
{prop:ERM} and \ref{prop:DRO} that,
\begin{equation}
V_{n}^{ERM}(\cdot )\Rightarrow V^{ERM}(\cdot )=V(-H,\cdot )\quad \text{ and
}\quad V_{n}^{DRO}(\cdot )\Rightarrow V^{DRO}(\cdot )=V\{-f_{\eta ,\gamma
}(H),\cdot \}  \label{inter-mas-thm-pf-0}
\end{equation}%
jointly, on the space topologized by uniform convergence on compact sets.

To prove convergence of the third component of the triplet considered in
Theorem \ref{thm:master}, observe from the definitions of $\varphi ^{\ast
}(\cdot )$ and $C$ that,
\begin{equation}
  \varphi ^{\ast }(H-Cu)=\max_{\xi }\big(\xi ^{{\mathrm{\scriptscriptstyle T}}}
  [ H-E\{ D_{\beta }h(X,\beta _{\ast })\} u] - 4^{-1}
  E\Vert  \{ D_{x}h(X,\beta _{\ast })\} ^{{\mathrm{\scriptscriptstyle %
        T}}}\xi \Vert _{p}^{2}\big).  \label{Expr-psi-HCu}
\end{equation}%
Consider any fixed $K\in (0,+\infty ).$ Due to the weak convergence $%
H_{n}\Rightarrow -H,$ applications of continuous mapping theorem to the
bounds in Proposition \ref{prop:RWPUB},~\ref{prop:RWPLB} result in the
conclusions that,
\begin{equation}
  f_{up}(H_{n},u,b,c)\Rightarrow \max_{\Vert \xi \Vert _{p}\leq b}\big\{\xi ^{{\T}}H-
  E\big[ {4}^{-1}\Vert \big\{D_{x}h(X,\beta _{\ast })] ^{{\mathrm{\scriptscriptstyle T}}}\xi
  \Vert_{p}^{2}+\xi ^{{\mathrm{\scriptscriptstyle T}}}D_{\beta }h(X,\beta _{\ast})u\big]
  \mathbb{I}(X \in C_0^{\epsilon_0}) \big\},
\label{inter-mas-up}
\end{equation}%
\begin{equation}
  f_{low}(H_{n},u,b)\Rightarrow \max_{\Vert \xi \Vert _{p}\leq b}\big\{\xi ^{{\T}}H-
  E\big[ {4}^{-1}\Vert  \{D_{x}h(X,\beta _{\ast })\}^{{\mathrm{\scriptscriptstyle T}}}\xi \Vert
  _{p}^{2}+\xi ^{{\mathrm{\scriptscriptstyle T}}}D_{\beta }h(X,\beta _{\ast
  })u\big] \big\},  \label{inter-mas-low}
\end{equation}%
for any $u$ satisfying $\Vert u\Vert _{2}\leq K.$ Since the bounds in
Propositions \ref{prop:RWPUB},~\ref{prop:RWPLB} hold for arbitrarily large
choices for constants $b,c,$ and arbitrarily small choice for constant $\epsilon_0$ combining with the assumption $P_*(\Omega^\circ)=1$, we conclude from the observations %
\eqref{Expr-psi-HCu}, \eqref{inter-mas-up}, and \eqref{inter-mas-low}   that
\begin{equation}
  nR_{n}\big( \beta _{\ast }+n^{-1/2}u\big) \Rightarrow \varphi ^{\ast}(H-Cu),  \label{mas-fdim-conv}
\end{equation}%
for any $u$ satisfying $\Vert u\Vert _{2}\leq K.$ Finally, we have
from Propositions \ref{prop:tightness-1} and \ref%
{prop:RWPcont} that the collection
$\{nR_{n}(\beta_{\ast}+n^{-1/2}\times%
\cdot\,)\}$ is tight; see, for example, \citet[Theorem
7.4]{billingsley2013convergence}. As a consequence of this tightness
and the finite dimensional convergence in \eqref{mas-fdim-conv}, we
have that,
\begin{equation*}
nR_{n}\big( \beta_{\ast}+n^{-1/2}\times\cdot\,\big) \Rightarrow
\varphi^{\ast}(H-C\times\cdot\,).
\end{equation*}
Combining this observation with those in \eqref{inter-mas-thm-pf-0}, we
obtain the desired convergence result in Theorem \ref{thm:master}. Furthermore, since $f_{low}(H_{n},u,b)$ and $f_{up}(H_{n},u,b)$ are associated with the same $H_n$ with inequalities \eqref{eqn:V_DRO} and \eqref{eqn:V_ERM}, we have the three terms converge jointly.
\end{proof}

\subsection{Proof of Theorem \protect\ref{thm:levelsets-master}}

\begin{proof}[Proof  of Theorem \ref{thm:levelsets-master}]
  \label{sec:levelsets-master} Theorem \ref{thm:levelsets-master} is
  proved by considering suitable level sets of the component functions
  in the triplet, $%
  \{V_{n}^{ERM}(\cdot), \ V_{n}^{DRO}(\cdot), \ nR_{n}(
    \beta_{\ast}+ n^{-1/2}\times \cdot\ )\}, $ considered in
  Theorem \ref{thm:master}.  To reduce clutter in expressions, from
  here-onwards we refer the distributionally robust estimator
  \eqref{Wass_DRO_A}, simply as $\beta^{DRO}_n,$ with the dependence
  on the radius $\delta_n$ to be understood from the context. To begin,
  consider the following tightness result whose proof is provided in
  Section \ref{ssec:proofs-props-mast-lev-sets}.

\begin{proposition}
The sequences $\{\arg\min\xspace_{u} \, V_{n}^{ERM}(u): n \geq1\}$ and $%
\{\arg\min\xspace_{u} \, V_{n}^{DRO}(u): n \geq1\}$ are tight. \label%
{prop:argmintightness}
\end{proposition}

Observe that $V_{n}^{ERM}(\cdot)$ and $V_{n}^{DRO}(\cdot)$ are
minimized, respectively, at $n^{1/2}(\beta_{n}^{ERM}-\beta_{\ast})$
and $n^{\bar{\gamma }/2}(\beta_{n}^{DRO}-\beta_{\ast}).$ Furthermore,
due to the positive definiteness of $C$ in Assumption A2.b, we
have that $V^{ERM}(\cdot)$ and $V^{DRO}(\cdot)$ are strongly convex
with respect to $u$ and have unique minimizers, with probability
1. Therefore, due to the tightness of the sequences
$\{n^{1/2}(\beta_{n}^{ERM}-\beta_{\ast})\}_{n\geq1}$ and $\{n^{%
  \bar{\gamma}/2}(\beta_{n}^{DRO}-\beta_{\ast})\}_{n\geq1}$; see
Proposition %
\ref{prop:argmintightness} and the weak convergence of
$V_{n}^{ERM}(\cdot)$ and $V_{n}^{DRO}(\cdot)$ in Theorem
\ref{thm:master}, we have the following convergences:
\begin{align}
  n^{1/2}(\beta_{n}^{ERM}-\beta_{\ast})&\Rightarrow\arg\min\xspace%
  _{u}\,V(-H,u)=C^{-1}H,  \label{inter-mas-thm-pf-1}\\
  n^{\bar{\gamma}/2}(\beta_{n}^{DRO}-\beta_{\ast})
  &\Rightarrow\arg\min \xspace%
  _{u}\,V^{DRO}(u)=C^{-1}f_{\eta,\gamma}(H)
  \nonumber
\end{align}

Finally, to prove the convergence of the sets $\Lambda^+_{\delta_{n}}(P_{n}),$
we proceed as follows. Define
\begin{equation*}
G_{n}(u) = nR_{n}(\beta_{\ast}+ n^{-1/2}u), \quad G(u) = \varphi^{\ast
}(H-Cu), \quad\text{ and } \quad\alpha_{n} = n\delta_{n}.
\end{equation*}
For any function $f:B \rightarrow\mathbb{R}$ and $\alpha \in[%
0,+\infty],$ let lev$(f,\alpha)$ denote the level set
$\{ x \in \mathbb{R}^d: f(x) \leq\alpha\}.$

\begin{proposition}
  If $\delta_{n} = n^{-1}\eta,$ then
  $\mathrm{cl}\{$\textnormal{lev}$(G_{n},\alpha_{n})\} \Rightarrow$
  \textnormal{lev}$(G,\eta).$ \label{prop-levset-eq1}
\end{proposition}

\begin{proposition}
If $\delta_{n} = n^{-\gamma}\eta$ for some $\gamma> 1,$ then $\mathrm{cl}\{$\textnormal{lev}%
$(G_{n},\alpha_{n})\} \Rightarrow\{C^{-1}H\}.$ \label{prop-levset-gthan1}
\end{proposition}

\begin{proposition}
  If $\delta_{n} = n^{-\gamma}\eta$ for some $\gamma< 1,$ then
  $\mathrm{cl}\{$\textnormal{lev}$(G_{n},\alpha_{n})\} \Rightarrow$
  $\mathbb{R}^{d}.$ \label{prop-levset-lthan1}
\end{proposition}

Propositions \ref{prop-levset-eq1} - \ref{prop-levset-lthan1} above, whose
proofs are furnished in Section \ref{ssec:proofs-props-mast-lev-sets}, allow
us to complete the proof of Theorem \ref{thm:levelsets-master} as follows.
It follows from the definition of $R_{n}(\beta )$ that,
\begin{equation*}
\Lambda^+_{\delta _{n}}(P_{n})=\left\{ \beta :R_{n}(\beta )\leq \delta
_{n}\right\} =\beta _{\ast }+n^{-1/2}\left\{ u:G_{n}(u)\leq \alpha
_{n}\right\} .
\end{equation*}%
We have from Propositions \ref{prop-levset-eq1} -
\ref{prop-levset-lthan1} that
\begin{equation*}
n^{1/2}\left( \Lambda^+_{\delta _{n}}(P_{n})-\beta _{\ast }\right) =\left\{
u:G_{n}(u)\leq \alpha _{n}\right\} \Rightarrow
\begin{cases}
\text{lev}(G,\eta )\quad & \text{ if }\gamma =1, \\
\mathbb{R}^{d} & \text{ if }\gamma <1, \\
\{C^{-1}H\} & \text{ if }\gamma >1.%
\end{cases}
\end{equation*}%
Observe that $\varphi ^{\ast }(u)=\varphi ^{\ast }(-u).$ Therefore,
lev$%
(G,\eta )=\{u:\varphi ^{\ast }(H-Cu)\leq \eta \}=C^{-1}H+\{u:\varphi
^{\ast }(Cu)\leq \eta \}.$ Since the three terms in Theorem
\ref{thm:master} converge jointly, we have the three terms in Theorem
\ref%
{thm:levelsets-master} also converge jointly. This completes the proof
of Theorem \ref%
{thm:levelsets-master}.
\end{proof}

Proposition \ref{prop:constrained-support} follows by adopting exactly
the same steps which are used to establish the convergence of
$n^{1/2}\left\{\Lambda^+_{\delta _{n}}(P_{n})-\beta _{\ast }\right\}$ in the proof
of Theorem \ref{thm:levelsets-master}.



\section{Discussions}
\label{sec:discussion}
We discuss the subtleties in
deriving a limit theorem for the distributionally robust estimator
$\beta_n^{DRO}$ when the support of the random vector $X,$ denoted by
$\Omega,$ is constrained to be a strict subset of $\mathbb{R}^m.$
Suppose that the support of $X$ is constrained to be contained in the
set $\Omega = \{x \in \mathbb{R}^m: Ax \leq b\}$ specified in terms of
linear constraints involving an $l \times m$ matrix $A$ and
$b \in \mathbb{R}^l.$ For the sake of clarity, we discuss here only
the non-degenerate case where $\delta_n = \eta/ n.$

Considering the transportation cost $c(x,y) = \Vert x- y\Vert_2^2$ in
Definition \ref{defn:WD}, we demonstrate in Section \ref{ssec:discussion_proof} of the
Supplementary material that the central limit theorem,
$n^{1/2}\{\beta_n^{DRO}(\delta_n) - \beta_\ast\} \Rightarrow
C^{-1}H - \eta^{1/2} C^{-1}D_\beta S(\beta_\ast),$ continues to hold,
for example, in the elementary
case 
where the matrix $A$ has linearly independent rows, $X$ has a
probability density which is absolutely continuous with respect to the
Lebesgue measure on $\mathbb{R}^m$ and the support $\Omega$ is
compact. 
A key element which emerges in the verification (offered in
Proposition \ref{prop:DRO-constrained} in Section \ref{ssec:discussion_proof} of the supplementary material) is that
the fraction of samples which get transported to the boundary of the
set $\Omega$ stays $O_p(n^{-1/2}),$ as $n \rightarrow \infty.$

On the other hand, when the set
$\Omega = \{x \in \mathbb{R}^m: Ax \leq b\}$ has equality constraints
as in, for example,
$\Omega = \{(x_1,x_2,\ldots,x_m) \in \mathbb{R}^2: x_1 - x_2 = 0\},$
the bias term in the limit theorem gets affected due to the constraint
binding all the samples $\{X_1,\ldots,X_n\}$ and the fraction of
samples which get transported to the boundary of the set $\Omega$ is
1. This is easily seen in the linear regression example in Section
\ref{sec:geo_insignts} where $\ell(x,y;\beta) = (y - \beta^\T x)^2$ and
the support is taken as
$\Omega = \{(x_1,x_2) \in \mathbb{R}^2: x_1 = x_2\}.$ For this
elementary example, we instead have,
\begin{align}
    n^{1/2}\big\{\beta_n^{DRO}(\delta_n) - \beta_\ast\big\} \Rightarrow
  C^{-1}H - \eta^{1/2} C^{-1}D_\beta \tilde{S}(\beta_\ast),
  \label{clt-support-constraints}
\end{align}
where $\tilde{S}(\beta)$ is different from the term $S(\beta)$ as in,
$ \tilde{S}(\beta_\ast) = 2^{1/2-1/q} { \vert \beta^\T1 \vert} {
  \Vert \beta \Vert_p^{-1}} S(\beta).$
Here, recall the earlier definition
$S(\beta) = [E\{ \Vert D_x\ell(X;\beta)\Vert_p^2\}]^{1/2}$ in
\eqref{defn:sensitivity-term} for the unconstrained support case. The
computations required to arrive at the above conclusion are presented
in Example A1 in Section A.3 of the supplementary material. {
In the presence of general support constraints of the form $\Omega = \{x \in \mathbb{R}^m:Ax= b\},$ we show with Example A2 in Section A.3 that \eqref{clt-support-constraints} holds with $\tilde{S}(\beta) = \Vert P_{\mathcal{N}(A)} \beta \Vert_2$ for quadratic losses of the form $\ell(x;\beta) = a + \beta^\T x + \beta^\T C \beta;$ here $A$ is taken to be a matrix with linearly independent rows and $P_{\mathcal{N}(A)}$ denotes the projection operator onto the null space of $A.$ The bias term here is again different when compared to the term resulting from $S(\beta) = \Vert \beta \Vert_2$ exhibited in Theorem \ref{thm:levelsets-master}.}
As reasoned above, the presence of equality constraints for the support
$\Omega$ introduces new challenges to be tackled in another study.

\section*{Acknowledgements}
Material in this paper is based upon work supported by the Air Force Office of Scientific Research under award number FA9550-20-1-0397. Additional support is gratefully acknowledged from NSF grants 1915967, 1820942 and 1838576 and  MOE SRG ESD 2018 134.

\bibliographystyle{plainnat}
\bibliography{DR2,DR4}
\ \\

\appendix
\begin{center}
\textbf{\large Supplementary material}
\end{center}

\vspace{-15pt}
\section{Proofs pertaining to limit theorems of $\beta_n^{ERM}$ and
  $\beta_n^{DRO}(\delta_n)$}
  \label{ssec:dro_limit_theorem}
In this section we first present the proofs of Propositions
\ref{prop:ERM} - \ref{prop:DRO} which are useful towards establishing
convergences of the first two components of the triple considered in
Theorem \ref{thm:master}. Following these, we present the proofs of
Propositions \ref{prop:DRO-boundary} - \ref{prop:DRO-property}, both
pertaining to limit theorems for the distributionally robust estimator
under relaxed assumptions. Towards the end of this section, we also
provide the proofs of statements made in Section \ref{sec:discussion}.
%

\begin{proof}[Proof  of Proposition \ref{prop:ERM}]
Recall that $h(x;\beta)=D_{\beta}\ell(x;\beta).$ For $n$ sufficiently large, we have $\beta_{\ast}+n^{-\alpha/2}u \in B^\circ$ for $\| u\| \leq 2$. With $\ell(\cdot)$ being
twice continuously differentiable, employing Taylor expansion up to the
quadratic term, we obtain,
\begin{align*}
n^{\alpha-1}V_{n}^{ERM}\{n^{(1-\alpha)/2}u\} & =n^{\alpha}\left[ E%
_{P_{n}}\left\{ \ell\left( X;\beta_{\ast}+n^{-\alpha/2}u\right) \right\} -%
E_{P_{n}}\left\{ \ell\left( X;\beta_{\ast}\right) \right\} \right] \\
& =n^{\alpha/2}E_{P_{n}}\left\{ h(X;\beta_{\ast})\right\} ^{\T}u+\frac{%
1}{2}u^{\T}E_{P_{n}}\left\{ D_{\beta}h(X;\beta_{\ast })\right\} u+o(1),
\end{align*}
as $n\rightarrow\infty,$ uniformly over $u$ in compact sets. With this
expansion, the statement of Proposition \ref{prop:ERM} follows as a direct
consequence of the definitions, $H_{n}=n^{-1/2}\sum_{i=1}^{n}h(X_{i},%
\beta_{\ast}),$ $C=E\{D_{\beta}h(X;\beta_{\ast})\}$ and an application of the
law of large numbers, $\lim_{n\rightarrow\infty}E_{P_{n}}\{D_{\beta}h(X,%
\beta_{\ast})\}\rightarrow C$ almost surely.
\end{proof}

\subsection{Proof of Proposition \ref{prop:DRO}}
  \label{ssec:proof_prosition_DRO}
This subsection is devoted to the proof of Proposition \ref{prop:DRO}
taking $\Omega = \mathbb{R}^m.$ The following notation will be used in
the sequence of results below used to prove Proposition
\ref{prop:DRO}. Given $q \in (1,\infty),$ let $D_q(v),H_q(v)$ denote
the first derivative (gradient) and second derivative (Hessian) of the
function $f(\Delta) = \Vert \Delta \Vert_q^2$ evaluated at
$\Delta = v.$ Recall that we take $p$ to be such that
$p^{-1} + q^{-1} = 1.$ We also define the map
$T_p: \mathbb{R}^m \rightarrow \mathbb{R}^m$ as,
\begin{align*}
T_p(v) = \Vert v \Vert_p^{1-p/q} \text{sgn}(v)\vert v \vert^{p/q},
\end{align*}
where $\text{sgn}(\cdot)$ denotes the sign function.

Proposition \ref{prop:DRO} is proved via the sequence of results below.
\begin{lemma}
  Letting $\eta_{n} = \delta_{n} n^{\gamma},$  we have, for $\beta \in B$
  \begin{equation*}
    n^{\gamma/2}\left[ \Psi_{n}(\beta) - E_{P_{n}}\{\ell(X;\beta)\} \right]\\
    = \inf_{\lambda \geq 0}\left[ \lambda \eta_{n} +\frac{1}{4\lambda }
    E_{P_{n}}\left\{\left\Vert D_{x}\ell(X;\beta)\right\Vert _{p}^{2}\right\}
    + e_{n}(\beta,\lambda)\right],
\end{equation*}
where the function $e_n(\beta,\lambda)$ is
$e_n(\beta,\lambda) = E_{P_n}\left\{ f_{n}(X,\beta,\lambda)\right\},$ with
$f_n(\cdot)$ defined as,
\begin{align*}
  f_{n}(x,\beta,\lambda)
  = \sup_{\Delta \in \mathbb{R}^m} \left[ n^{\gamma/2}
  \left\{ \ell \big(x + n^{-\gamma/2}\Delta; \beta \big) -
  \ell(x;\beta)\right\}
  - \lambda \Vert \Delta \Vert_q^2\right] - \frac{1}{4\lambda}
  \Vert D_x \ell(x;\beta) \Vert_p^2.
\end{align*}
\label{Lem-1-Val-Fn}
\end{lemma}
\begin{proof}[Proof  of Lemma \ref{Lem-1-Val-Fn}]
  It follows from \citet[Theorem 1]{blanchet2016quantifying} that
\begin{align*}
\Psi_{n}(\beta)=\inf_{\lambda\geq0}\left[ n^{\gamma/2}\lambda\delta_{n} +%
 E_{P_{n}}\left\{ \phi_{\lambda}(X;\beta,\lambda)\right\} \right] , \quad\text{where}
\end{align*}
\begin{align*}
   \phi_{\lambda}(x;\beta,\lambda)
  =\sup_{\Delta \in \{\Delta \in \mathbb{R}^{m}| x+n^{-\gamma
/2} \Delta \in \Omega \}}\left\{ \ell\big( x+ n^{-\gamma
/2}\Delta;\beta\big) -\lambda
n^{-\gamma/2}\Vert\Delta\Vert_{q}^{2}\right\} .
\end{align*}
With $\Omega = \mathbb{R}^m,$ it follows from the definition of
$f_n(\cdot)$ that,
\begin{align*}
  n^{\gamma/2}\left\{\phi_\lambda(x;\beta) - \ell(x;\beta)\right\}
  &= \sup_{\Delta \in \mathbb{R}^m}
  \left[ n^{\gamma/2}\left\{\ell\big(x + n^{-\gamma/2}\Delta;\beta \big)
  - \ell(x;\beta)\right\} - \lambda \Vert \Delta \Vert_q^2\right]\\
  &= f_n(x,\beta,\lambda) + (4\lambda)^{-1}\Vert D_x \ell(x;\beta)
    \Vert_p^2.
\end{align*}
Then, since $e_n(\beta,\lambda) = E_{P_n}[f_n(X,\beta,\lambda)]$ and
$ \delta_n n^{\gamma} = \eta_n,$ we obtain,
\begin{align*}
  n^{\gamma/2} \left[ \Psi_n(\beta) - E_{P_n}\left\{ \ell(x;\beta)\right\}
  \right] = \inf_{\lambda \geq 0} \left[ \lambda \eta_n + \frac{1}{4\lambda}
  E_{P_n}\left\{ \Vert D_x\ell(X;\beta)\Vert_p^2\right\}
  + e_n(\beta,\lambda)\right].
\end{align*}
This completes the verification of the statement of Lemma
\ref{Lem-1-Val-Fn}.
\end{proof}

\begin{lemma}
  For any $\Delta,\Delta_{\ast}\in \mathbb{R}^{d},$ letting
  $\xi=\Delta-\Delta_{\ast},$ we have the following inequalities:
\begin{itemize}
\item[a)] if $q\in(1,2],$ then
  $\left\Vert \Delta\right\Vert _{q}^{2} \geq
  \left\Vert\Delta_{\ast}\right\Vert _{q}^{2}+
  D_q(\Delta_\ast)^\T\xi+(q-1)\left\Vert \xi\right\Vert _{q}^{2}; $ and
\item[b)] if $q>2,$ then
  $\left\Vert \Delta\right\Vert _{q}^{2}\geq
  \left\Vert \Delta_{\ast}\right\Vert _{q}^{2} +
  D_q(\Delta_\ast)^\T \xi+
  C\min\left\{\left\Vert \xi\right\Vert _{2}^{2},
    {\left\Vert \xi\right\Vert_{q}^{q}}
    {\left\Vert \Delta_{\ast}\right\Vert_{q-2}^{-(q-2)}}\right\},$
where $C$ is a positive constant which  depends only on $d$ and $q.$
\end{itemize}
\label{Lem-Delta-ineq}
\end{lemma}
The proof of Lemma \ref{Lem-Delta-ineq} is technical in nature and is
provided in Section \ref{appendix:proof_technical_lemma}.

\begin{lemma}
  For any $v \in \mathbb{R}^d, \lambda > 0, \varepsilon > 0$ and
  $d \times d$ symmetric matrix $B,$ we have that the value of
  optimization,
  \begin{align}
    &\sup_{\Delta \in \mathbb{R}^d} \left( v^\T \Delta -
      \lambda \Vert \Delta \Vert_q^2
    + \varepsilon \Delta^\T B \Delta \right)\label{unconstr-obj}
  \end{align}
  is upper bounded and lower bounded as follows:
  \begin{align*}
    0 \leq \sup_{\Delta \in \mathbb{R}^d} \left\{ v^\T \Delta -
    \lambda \Vert \Delta \Vert_q^2 + \varepsilon \Delta^\T B \Delta
    \right\}  - \frac{\Vert v \Vert_p^2}{4\lambda}
    - \frac{\varepsilon}{4\lambda^2} T_p(v)^\T B T_p(v) \leq
        \frac{c_0 \varepsilon^{\bar{q}} \Vert \Delta_v\Vert_2^2}
    {\min\{ (\lambda - c_1\varepsilon)^+, \lambda^{\frac{1}{q-1}}\}}
  \end{align*}
  where $\Delta_v = (2\lambda)^{-1}T_p(v),$
  $\bar{q} = \min\{2,q/(q-1)\},$ and $c_0,c_1$ are positive constants
  which depends only on $d,q$ and the Frobenius norm of the matrix $B.$
\label{lem:unconstr-opt}
\end{lemma}
\begin{proof}[Proof  of Lemma \ref{lem:unconstr-opt}]
    First, we consider the case where $B$ is the zero matrix. When
  $B = 0,$ we have
  \begin{align*}
    \sup_{\Delta \in \mathbb{R}^d}
    \left( v^\T \Delta - \lambda \Vert \Delta \Vert_q^2
    \right) = (4\lambda)^{-1}\Vert v \Vert_p^2,
  \end{align*}
  in which the maximum is attained at $\Delta = \Delta_v;$ here,
  recall that $\Delta_v$ is
  $\Delta_v = (2\lambda)^{-1}T_p(v).$ The corresponding optimality
  condition is
  \begin{align}
      v - \lambda D_q(\Delta_v) = 0,
    \label{opt-cond}
  \end{align}
  where $D_q(\Delta_v)$ is the first derivative of the function
  $\Vert \Delta \Vert_q^2$ evaluated at $\Delta = \Delta_v.$ Next, for
  the case where the matrix $B$ is not zero, we proceed by changing
  the variable from $\Delta$ to $\xi$ with the relationship,
  $ \Delta = \Delta_v + \varepsilon \xi .$ Then the objective
  $f(\Delta) = \Delta^\T v - \lambda \Vert \Delta \Vert_q^2 +
  \varepsilon\Delta^\T B \Delta$ is rewritten in terms of the variable
  $\xi$ as follows: $f\big( \Delta_v + \varepsilon \xi\big)$ equals
  \begin{align*}
     & \left( v^\T\Delta_v - \lambda \Vert \Delta_v \Vert_q^2 \right)
       + \varepsilon \Delta_v^\T B \Delta_v
       + \varepsilon \xi^\T\left( v + 2\varepsilon B\Delta_v\right)
       - \lambda \left( \Vert \Delta_v + \varepsilon \xi\Vert_q^2
       - \Vert \Delta_v \Vert_q^2  \right) + \varepsilon^3 \xi^\T  B\xi\\
     &\ \quad =
       (4\lambda)^{-1}\Vert v \Vert_p^2 + \varepsilon \Delta_v^\T B \Delta_v
       + \varepsilon \xi^\T\left\{ v - \lambda D_q(\Delta_v)
       + 2\varepsilon B\Delta_v\right\}\\
     &\qquad\qquad\qquad\qquad
       - \lambda  \left\{ \Vert \Delta_v + \varepsilon \xi\Vert_q^2
       - \Vert \Delta_v \Vert_q^2 - \varepsilon \xi^\T D_q(\Delta_v)\right\}
       + \varepsilon^3 \xi^\T  B\xi.
  \end{align*}
  Then, we have from \eqref{opt-cond} that,
  \begin{align}
    &\varepsilon^{-2}\left\{f(\Delta) - (4\lambda)^{-1}\Vert v \Vert_p^2
      - \varepsilon \Delta_v^\T B \Delta_v\right\} \nonumber\\
      &\qquad\qquad=
    2\xi^\T B\Delta_v -  \lambda \varepsilon^{-2}
    \left\{ \Vert \Delta_v + \varepsilon \xi\Vert_q^2
        - \Vert \Delta_v \Vert_q^2 - \varepsilon \xi^\T D_q(\Delta_v)\right\}
        + \varepsilon \xi^\T  B\xi.
    \label{unconstr-second-order}
  \end{align}

  For deriving the upper bound in the statement of Lemma
  \ref{lem:unconstr-opt}, we proceed by utilizing the bound,
  $ \Vert \Delta_v + \varepsilon \xi \Vert_q^2 - \Vert \Delta_v
  \Vert_q^2 - \varepsilon \xi^\T D_q(\Delta_v) \geq P_v(\varepsilon
  \xi)$
  from Lemma \ref{Lem-Delta-ineq}, where
  $P_v:\mathbb{R}^d \rightarrow \mathbb{R}_+$  defined as,
    \begin{align*}
      P_v(x) &= \tilde{c}
      \begin{cases}
        \Vert x  \Vert_2^2  \quad&\text{ if } q \in (1,2]\\
        \min\left( \Vert x \Vert_2^2,
          \Vert x \Vert_2^q  \Vert \Delta_v \Vert_2^{2-q} \right)
        &\text{ if } q > 2\\
      \end{cases}
    \end{align*}
    for a suitable positive constant $\tilde{c}$ that depends only on
    $d$ and $q$; indeed, the existence of consant $\tilde{c}$
    satisfying this requirement follows from the observation that
    $\Vert x \Vert_q \geq \hat{c} \Vert x \Vert_2$ for a suitable
    positive constant $\hat{c}$ which depends only upon $d$ and $q.$
    Then we have the following upper bound from \eqref{unconstr-second-order}:
      \begin{align}
    \varepsilon^{-2}\left\{ \sup_{\Delta} f(\Delta) - (4\lambda)^{-1}\Vert v \Vert_p^2
    - \varepsilon \Delta_v^\T B \Delta_v\right\} \leq  \sup_{\xi}
        \left\{2\xi^\T B\Delta_v - \lambda\varepsilon^{-2} P_v(\varepsilon \xi) +
        \varepsilon \xi^\T  B\xi \right\}.
    \label{unconstr-second-order-ub}
  \end{align}
  The following observations are useful in simplifying the right hand
  side of \eqref{unconstr-second-order-ub}. With $\Vert B \Vert$
  denoting the Frobenius norm of the matrix $B,$ we have
  $\Vert B\xi \Vert_2 \leq \Vert B \Vert \Vert \xi \Vert_2$ and
  $\xi^\T B\xi  \leq \Vert B \Vert \Vert \xi \Vert_2^2.$
  As a consquence, when $q \in (1,2],$
  \begin{align*}
    \sup_{\xi} \left\{2\xi^\T B\Delta_v - \lambda \varepsilon^{-2} P_v(\varepsilon \xi)
    + \varepsilon \xi^\T  B\xi \right\}
    &\leq     \sup_{\xi} \left\{2\xi^\T B\Delta_v - \left(\lambda \tilde{c} -
      \varepsilon \Vert B \Vert \right) \Vert \xi \Vert_2^2 \right\}   =
      \frac{\Vert B \Vert^2 \Vert \Delta_v \Vert_2^2}
      {\left( \lambda \tilde{c}- \varepsilon \Vert B \Vert \right)^+}.
  \end{align*}
  In the above expression, $x^+ = \max\{x,0\}$ denotes the positive
  part of any real number $x.$ Next, when $q > 2,$ we have the
  following as a consequence of Cauchy-Schwarz inequality:
  \begin{align*}
    &\sup_{\xi} \left(2\xi^\T B\Delta_v - \lambda \tilde{c}
      \varepsilon^{-2} \Vert \varepsilon \xi \Vert_2^q\Vert \Delta_v
      \Vert_2^{2-q} + \varepsilon \xi^\T  B\xi \right)\\
    &\quad \leq \sup_{C \geq 0} \left( 2 \Vert B \Vert \Vert \Delta_v \Vert_2  C
      - \lambda \tilde{c} \Vert \Delta_v \Vert_2^{2-q}\varepsilon^{q-2}   C^q
      + \varepsilon \Vert B \Vert C^2 \right)\\
    &\quad \leq \sup_{C \geq 0} \left( 2 \Vert B \Vert \Vert \Delta_v \Vert_2  C
      - 2^{-1}\lambda \tilde{c} \Vert \Delta_v \Vert_2^{2-q}\varepsilon^{q-2}   C^q \right) +
      \sup_{C \geq 0} \left(\varepsilon \Vert B \Vert C^2
      - 2^{-1}\lambda \tilde{c} \Vert \Delta_v \Vert_2^{2-q}\varepsilon^{q-2}   C^q \right)\\
    &\quad \leq c\varepsilon^{-\frac{q-2}{q-1}} \lambda^{-\frac{1}{q-1}}
      \Vert \Delta_v\Vert_2^2  \left\{ 1  + (\varepsilon
      \lambda^{-1})^{\frac{q}{(q-1)(q-2)}}\right\},
  \end{align*}
  where $c$ is a suitable positive constant which depends only upon
  $d,q$ and $\Vert B \Vert.$ Then letting
  $\bar{q} = \min\{2,q/(q-1)\},$ we obtain from
  \eqref{unconstr-second-order-ub} and the above two upper bounds
  that,
  \begin{align*}
    \sup_{\Delta} f(\Delta) - (4\lambda)^{-1}\Vert v \Vert_p^2
    - \varepsilon \Delta_v^\T B \Delta_v \leq
    \frac{c_0 \varepsilon^{\bar{q}} \Vert \Delta_v\Vert_2^2}
    {\min\{ (\lambda - c_1\varepsilon)^+, \lambda^{\frac{1}{q-1}}\}}
  \end{align*}
  where $c_0,c_1$ are positive constants which depends only on $d,q$
  and $\Vert B\Vert.$ With this conclusion proving the upper bound, the
  lower bound is obtained by letting $\Delta = \Delta_v$ in the
  evaluation of $f(\Delta).$ This concludes the proof of Lemma
  \ref{lem:unconstr-opt}.
\end{proof}

\begin{lemma}
  For any $\beta \in B$ and $\lambda > 0,$ we have the following
  approximation for the term $e_n(\beta,\lambda)$ identified in Lemma
  \ref{Lem-1-Val-Fn}: As $n \rightarrow \infty,$
  \begin{align*}
    e_n(\beta,\lambda) = 8^{-1}\lambda^{-2} n^{-\gamma/2}
    a_n(\beta) + O_p(n^{-\bar{q}\gamma/2}),
  \end{align*}
  where
  \[a_n(\beta) = E_{P_n} \left[T_p\left\{ D_x\ell(X;\beta)\right\}^\T
      D_{xx}\ell(X;\beta) T_p\left\{
        D_x\ell(X;\beta)\right\}\right],\] and the convergence is
  uniform over $\beta$ in compact subsets of $B$ and $\lambda$ bounded
  away from zero. Moreover, the $O_p(n^{-\bar{q}\gamma/2})$ term is
  such that $\sup_{\lambda > 0} \lambda^2 O_p(n^{-\bar{q}\gamma/2})$
  is bounded from below by an integral random
  variable.  
\label{lem:en-bnd}
\end{lemma}

\begin{proof}[Proof  of Lemma \ref{lem:en-bnd}]
  Consider any $\beta$ satisfying $\Vert \beta \Vert_2 \leq b$ and
  $\lambda > \lambda_0 \in (0,1].$ For any
  $\varepsilon^\prime >0,$ we have as a consequence of Taylor
  expansion and uniform continuity of $D_{xx}\ell(\cdot;\beta)$ in Assumption A2.c that,
  \begin{equation*}
    \left\vert \ell(x+\Delta n^{-\gamma/2};\beta)-\ell\left( x;\beta\right)
      -n^{-\gamma/2} D_{x}\ell(x;\beta)^{\T}\Delta-
      2^{-1}n^{-\gamma}\Delta^{\T}D_{xx}\ell(x;\beta)\Delta\right\vert \leq\varepsilon^\prime
    n^{-\gamma}{\Vert\Delta \Vert_{q}^{2}},
  \end{equation*}
  for all $n \geq n_0$ where $n_{0}$ is sufficiently large. Then it
  follows from the definition of $f_n(\cdot)$ and Assumption A2.c that
  $f_n(x,\beta,\lambda)$ is upper and lower bounded, respectively, by,
  \begin{align*}
    &\sup_{\Delta} \left\{ D_x \ell(x;\beta)^\T
    \Delta + 2^{-1}n^{-\gamma/2} \Delta^\T D_{xx} \ell(x;\beta) \Delta
    - \big(\lambda - \varepsilon^\prime n^{-\gamma/2} \big)\Vert \Delta
    \Vert_q^2  \right\} - \frac{1}{4\lambda} \Vert D_x \ell(x,\beta) \Vert_p^2
   \text{ and }\\
   &\sup_{\Delta} \left\{ D_x \ell(x;\beta)^\T
    \Delta + 2^{-1} n^{-\gamma/2} \Delta^\T D_{xx} \ell(x;\beta) \Delta
    - \big(\lambda + \varepsilon^\prime n^{-\gamma/2} \big)\Vert \Delta
    \Vert_q^2 \right\} - \frac{1}{4\lambda} \Vert D_x \ell(x,\beta)
    \Vert_p^2.
  \end{align*}
  Letting $\varepsilon = n^{-\gamma/2},$ $v = D_x\ell(x;\beta),$
  $\bar{q} = \min\{2,q/(q-1)\}$ and $B =D_{xx}\ell(x;\beta),$ we
  obtain from the bounds derived for (\ref{unconstr-obj}) in Lemma
  \ref{lem:unconstr-opt} that,
  \begin{align*}
    f_n(x,\beta,\lambda) - 8^{-1}\lambda^{-2} n^{-\gamma/2}
    T_p\left\{ D_x\ell(x;\beta)\right\}^\T D_{xx}\ell(x;\beta)
    T_p\left\{ D_x\ell(x;\beta)\right\}
  \end{align*}
  is upper bounded by,
  \begin{align*}
    c_u (1 + \varepsilon^\prime) n^{-\bar{q}\gamma/2} \Vert     T_p\left\{
    D_x\ell(x;\beta)\right\}  \Vert_2^2
    \ \lambda_0^{-\max\{2,\frac{1}{q-1}\}}
  \end{align*}
  and likewise, lower bounded by,
  \begin{align}
        -c_l \varepsilon^\prime n^{-\gamma} \Vert     T_p\left\{
    D_x\ell(x;\beta)\right\} \Vert_2^2 \lambda^{-2}
    \label{lb-lambda-sq}
  \end{align}
  for suitable positive constants $c_l,c_u$ which are, in turn,
  determined by the constants $b, d$ and $q.$ Since
  $e_n(\beta,\lambda)$ is defined to equal
  $E_{P_n}[f_n(X,\beta,\lambda)],$ due to the finiteness of the second
  moment of
  $\sup[T_p\{D_x \ell(X;\beta)\}: \Vert \beta \Vert_2 \leq b],$ we
  have that
  \begin{align*}
    e_n(\beta,\lambda) = 8^{-1}\lambda^{-2} n^{-\gamma/2}
    E_{P_n} \left[T_p\left\{ D_x\ell(X;\beta)\right\}^\T D_{xx}\ell(X;\beta)
    T_p\left\{ D_x\ell(X;\beta)\right\}\right] + O_p(n^{-\bar{q}\gamma/2}),
  \end{align*}
  where the convergence is uniform over $(\beta,\lambda)$ such that
  $\Vert \beta \Vert_2 \leq b$ and $\lambda > \lambda_0.$ The
  observation that the $O_p(n^{-\bar{q}\gamma/2})$ term satisfies
  $\lambda^2 O_p(n^{-\bar{q}\gamma/2})$ is bounded from below by an
  integral random variable, uniformly over all $\lambda > 0$ and
  $\Vert \beta \Vert \leq b$, follows from the lower bound in
  \eqref{lb-lambda-sq}.
\end{proof}

\begin{proposition}
As $n\rightarrow\infty,$ we have,
\begin{equation*}
  \Psi_{n}(\beta) = E_{P_{n}}[\ell(X;\beta)]+ \delta_{n}^{1/2}
  \left[ E_{P_{n}}\big\{\left\Vert D_{x}\ell(X;\beta)\right\Vert_{p}^{2}\big\} \right]^{1/2}
  + \delta_{n} \frac{a_{n}(\beta)}{2E_{P_{n}}\big\{\left\Vert
      D_{x}\ell(X;\beta)\right\Vert _{p}^{2}\big\}}
  + o_p\left( \delta_{n} \right) ,
\end{equation*}
uniformly over $\beta$ in compact sets. \label{Lem-2-Val-fn}
\end{proposition}

\begin{proof}[Proof  of Proposition \ref{Lem-2-Val-fn}]
  We have from Lemma \ref{Lem-1-Val-Fn} and \ref{lem:en-bnd} that
  $n^{\gamma/2}\left( \Psi_{n}(\beta) - E_{P_{n}}[\ell(X;\beta)]
  \right)$ equals,
  \begin{align}
    \lim_{\lambda_0 \downarrow 0}
    \inf_{\lambda \geq \lambda_0}\left[ \lambda \eta_{n} +\frac{1}{4\lambda }
    E_{P_{n}}\left\{\left\Vert D_{x}\ell(X;\beta)\right\Vert _{p}^{2}\right\}
    +  8^{-1}\lambda^{-2} n^{-\gamma/2}
    a_n(\beta) + O_p(n^{-\bar{q}\gamma/2})\right],
    \label{approx-psin}
  \end{align}
  where the $O_p(n^{-\bar{q}\gamma/2})$ term in the above equation is
  uniform over
  $\{(\beta,\lambda): \Vert \beta \Vert_2 \leq b, \lambda >
  \lambda_0\},$ for any $b,\lambda_0 > 0,$ and
  $\sup_{\lambda > 0} \lambda^2 O_p(n^{-\bar{q}\gamma/2})$ is bounded
  from below by an integral random variable.  To solve this
  minimization, we begin by understanding the solution to the problem
  $\inf_{\lambda \geq 0} g_{1}(\lambda),$ where
\begin{equation*}
 g_{1}(\lambda)=  a\lambda+b/\lambda+c\varepsilon/\lambda^{2},
\end{equation*}
where $a,b,\varepsilon$ are positive constants and $c$ is non-negative.
Changing variable as in $\lambda=(b/a)^{1/2}(1+\varepsilon ua^{1/2})$
results in,
\begin{equation}
  \inf_{\lambda \geq 0} g_{1}(\lambda)
  = 2(ab)^{1/2} + \varepsilon acb^{-1} + \varepsilon^{2}a^{3/2}
  \inf_{u \geq -\varepsilon^{-1} a^{-1/2}} g_2(u),
\label{inter-dro-lempf-1}
\end{equation}
where
\begin{equation*}
  g_{2}(u)=\frac{b^{1/2}u^{2}}{1+\varepsilon ua^{1/2}}-\frac{c}{b}
  \frac{u(2+\varepsilon ua^{1/2})}{(1+\varepsilon ua^{1/2})^{2}}.
\end{equation*}
Since
\begin{equation*}
  g_{2}(u)\geq\frac{b^{1/2}u^{2}}{1+\varepsilon u a^{1/2}}-\frac{c}{b}
  \frac{2u}{(1+\varepsilon u{a}^{1/2})}=\frac{b^{1/2}u^{2}-(2c/b)u}
  {1+\varepsilon ua^{1/2}},
\end{equation*}
for $u\geq0,$ we have that, $\inf_{u\geq0}g_{2}(u)=0$ if $c=0$ and $%
\inf_{u\geq0}g_{2}(u)<0$ if $c>0.$ For the case $c>0,$ for all values
of $u>0$ such that $b^{1/2}u^{2}-(2c/b)u<0$ we have
$g_{2}(u) > b^{1/2}u^{2}-(2c/b)u.$ Since $b^{1/2}u^{2}-(2c/b)u$ is
lower bounded by $-c^{2}b^{-5/2}$ irrespective of the value of $u,$ we
have,
\begin{equation*}
-c^{2}b^{-5/2}\leq\inf_{u\geq-1/(\varepsilon {a}^{1/2})}g_{3}(u)\leq0,
\end{equation*}
for all sufficiently small $\varepsilon.$ Moreover, the infimum is
attained at $u \geq 0.$ Combining this observation with
\eqref{inter-dro-lempf-1}, we obtain that,
\begin{equation}
\left\vert \inf_{\lambda\geq 0}g_{1}(\lambda)-2(ab)^{1/2}-\varepsilon
  ac/b\right\vert \leq\varepsilon^{2}c^{2}a^{3/2}b^{-5/2},
\label{g1-opt-bnd}
\end{equation}
for all sufficiently small $\varepsilon,$ and the infimum is attained
at a choice of $\lambda \geq (b/a)^{1/2}.$ Letting $a = \eta_n,$
$b = 4^{-1} E_{P_n}\{ \Vert D_x\ell(X;\beta) \Vert_p^2\}, \varepsilon
= n^{-\gamma/2}$ and $c = 8^{-1}a_n(\beta) \geq 0,$ we obtain from
\eqref{g1-opt-bnd} that,
\begin{align*}
  \inf_{\lambda\geq 0}
  &\left[ \lambda \eta_{n} +\frac{1}{4\lambda }
    E_{P_{n}}\left\{\left\Vert D_{x}\ell(X;\beta)\right\Vert _{p}^{2}\right\}
    +  8^{-1}\lambda^{-2} n^{-\gamma/2} a_n(\beta) \right]\\
  &\quad = \left[\eta_nE_{P_{n}}\left\{\left\Vert D_{x}\ell(X;\beta)
    \right\Vert _{p}^{2}\right\}\right]^{1/2} + 2^{-1}n^{-\gamma/2}\eta_na_n(\beta)
    \left[E_{P_{n}}\left\{\left\Vert D_{x}\ell(X;\beta)
    \right\Vert _{p}^{2}\right\}\right]^{-1}
    + O_p(n^{-\gamma}),
\end{align*}
as $n \rightarrow \infty,$ and that the limit supremum of the sequence
of minimizers which attain the above infimum is positive.
Consequently, as $n \rightarrow \infty,$ we have that
(\ref{approx-psin})  equals,
\begin{align*}
   \left[\eta_nE_{P_{n}}\left\{\left\Vert D_{x}\ell(X;\beta)
  \right\Vert _{p}^{2}\right\}\right]^{1/2} + 2^{-1}n^{-\gamma/2}b_na_n(\beta)
      \left[E_{P_{n}}\left\{\left\Vert D_{x}\ell(X;\beta)
    \right\Vert _{p}^{2}\right\}\right]^{-1}
    + O_p(n^{-\bar{q} \gamma/2}),
\end{align*}
due to $\bar{q} \leq 2$ and the tightness of the collection
$\{\lambda^2 O_p(n^{-\bar{q}\gamma/2}): \lambda > 0\}.$ Since (\ref{approx-psin}) in
turn equals $n^{\gamma/2}(\Psi_n(\beta) - E_{P_n}\{\ell(X;\beta)\}),$
we obtain the claim in Proposition \ref{Lem-2-Val-fn} by substituting
$\eta_n = \delta_n n^{\gamma}.$
\end{proof}

\begin{proof}[Proof  of Proposition \ref{prop:DRO}]
For ease of notation, define $S_{n}(\beta)=[E_{P_{n}}\{\Vert
D_{x}\ell(X;\beta)\Vert _{p}^{2}\}]^{1/2}.$ Then it follows from the
definitions of $V_{n}^{DRO}(\cdot),V_{n}^{ERM}(\cdot)$ and the conclusion in
Lemma \ref{Lem-2-Val-fn} that,
\begin{align}
  V_{n}^{DRO}(u)
  & =n^{\bar{\gamma}-1}V_{n}^{ERM}\big\{n^{(1-\bar{\gamma})/2}u \big\}+
    n^{\bar{\gamma}}\delta_{n}^{1/2}\left\{ S_{n}(\beta_{\ast}+n^{-\bar{\gamma}%
/2}u)-S_{n}(\beta_{\ast})\right\}  \notag \\
& \hspace{50pt}+\frac{n^{\bar{\gamma}}\delta_{n}}{2}\left\{ \frac{%
a_{n}(\beta_{\ast}+n^{-\bar{\gamma}/2}u)}{S_{n}(\beta_{\ast}+n^{-\bar{\gamma}%
/2}u)}-\frac{a_{n}(\beta_{\ast})}{S_{n}(\beta_{\ast})}\right\} +o\left(
1\right) ,
\end{align}
as $n\rightarrow\infty,$ uniformly over $u$ in compact sets. Since $\bar{%
\gamma}=\min\{\gamma,1\},$ due to the twice continuous differentiability of $%
\ell(\cdot),$ we have that,
\begin{align*}
V_{n}^{DRO}(u) & =n^{\bar{\gamma}-1}V_{n}^{ERM}\big\{n^{(1-\bar{\gamma})/2}u \big\}+
\eta^{1/2}n^{\bar{\gamma}-\gamma/2}\left\{ S_{n}(\beta_{\ast}+n^{-\bar{\gamma%
}/2}u)-S_{n}(\beta_{\ast})\right\} +o(1) \\
& =n^{(\bar{\gamma}-1)/2}H_{n}^{\T}u+\frac{1}{2}u^{\T}Cu+ \eta^{1/2}n^{\left(
\bar{\gamma}-\gamma\right) /2}D_{\beta}S_{n}(\beta_{\ast})^{\T}u+o(1),
\end{align*}
as $n\rightarrow\infty,$ uniformly over $u$ in compact sets. Since $D_{\beta
}S_{n}(\beta_{\ast})$ converges to $D_{\beta}S(\beta_{\ast}),$ combining the
above observation with the statement of Proposition \ref{prop:ERM}, we
obtain the conclusion of Proposition \ref{prop:DRO}.
\end{proof}

\subsection{Proofs of Propositions \ref{prop:DRO-boundary} -
  \ref{prop:DRO-property}}
    \label{ssec:proofs_dro_bound_prop}
\begin{proof}[Proof  of Proposition \ref{prop:DRO-boundary}]
  First, consider the Lagrangian function,
  \begin{align*}
    L_n(\beta,\lambda) = \Psi_n(\beta) + \sum_{i \in I \cup J} \lambda_i g_i(\beta),
  \end{align*}
  and the pointwise maximum function,
  \begin{align*}
    \Phi_n(\beta) = \max\left\{L_n(\beta,\lambda): \lambda \in \Lambda_0 \right\}.
  \end{align*}
  Under the stated Mangasarian-Fromovitz constraint qualification
  conditions, we have that the set $\Lambda_0$ is nonempty, bounded
  convex polytope; see the discussion following Assumption B.3 in
  \citet{shapiro1989}. Therefore, $\Lambda_0$ is a convex
  hull of a finite set of extreme points denoted by $\Lambda_e.$ Then
  from the definition of $L(\beta,\lambda),$
  \begin{align*}
    \Phi_n(\beta)
    &= \max\left[ \Psi_n(\beta)  + L(\beta,\lambda) - E\{\ell(X;\beta)\}
      : \lambda \in \Lambda_e\right]\\
    &= \left[\Psi_n(\beta) -  E_{P_n}\{\ell(X;\beta)\}\right]
      + \left[  E_{P_n}\{\ell(X;\beta)\} - E\{\ell(X;\beta)\}\right] +
      \max\left[ L(\beta,\lambda) : \lambda \in \Lambda_e\right].
  \end{align*}
  Letting
  $H_n = -n^{1/2}\left[E_{P_n}\{h(X;\beta_\ast)\} -
    E\{h(X;\beta_\ast)\}\right]$ and taking $S_{n}(\beta)$ as in the
  proof of Proposition \ref{prop:DRO}, we obtain
  the following from the smoothness properties of $\ell(\cdot)$ in
  Assumptions A2.a, A2.c, expansion for $\Psi_n(\beta)$ in Proposition
  \ref{Lem-2-Val-fn}, and its subsequent application in Proposition
  \ref{prop:DRO}:
  \begin{align*}
    &n\left\{\Phi_n(\beta_\ast + n^{-1/2}u) - \Phi_n(\beta_\ast)\right\}
      = I_n(u) + J_n(u) + K_n(u),
  \end{align*}
  where
  \begin{align*}
    I_n(u) &= n\delta_n^{1/2}\left\{ S_n(\beta_\ast + n^{-1/2}u) -
             S_n(\beta_\ast)\right\} + o_p(n\delta_n)
             = \eta^{1/2}D_\beta S_n(\beta_\ast)^\T u + o_p(1),\\
    J_n(u) &=  n\left[  E_{P_n}\{\ell(X;\beta_\ast + n^{-1/2}u)\} -
             E_{P_n}\{\ell(X;\beta_\ast)\}\right]
             +  n\left[  E_{P}\{\ell(X;\beta_\ast + n^{-1/2}u)\} - E\{\ell(X;\beta)\}\right]\\
           &= -H_n^\T u + o_p(1)\\
    K_n(u) &= \max_{\lambda \in \Lambda_e} L(\beta_\ast + n^{-1/2}u,\lambda) -
             \max_{\lambda \in \Lambda_e} L(\beta_\ast,\lambda)
             =  2^{-1}q(u) + o(1),
  \end{align*}
  uniformly over compact sets of the variable $u.$ While the
  simplications for terms $I_n(u),J_n(u)$ are following the obtained same
  reasoning in the proofs of Propositions \ref{prop:ERM} - \ref{prop:DRO}, the last equality pertaining to $K_n(u)$ follows
  from the finiteness of the set $\Lambda_e,$ taylor expansion for
  $\max_{\lambda \in \Lambda_e} L(\beta_\ast + n^{-1/2}u)$ around
  $u = 0,$ and the Kuhn-Tucker optimality condition that
  $D_\beta L(\beta_\ast,\lambda) = 0$ for all $\lambda \in \Lambda_e.$
  Thus,
  \begin{align}
    &n\left\{\Phi_n(\beta_\ast + n^{-1/2}u) - \Phi_n(\beta_\ast)\right\}
      = \left\{-H_n + \eta^{1/2} D_\beta S_n(\beta_\ast) \right\}^\T u + 2^{-1}q(u) + o_p(1),
      \label{approx-dro-boundary}
  \end{align}
  uniformly in compact sets over the variable $u.$

  Next, we observe that the cone $\mathcal{C}$ of critical directions
  is nonempty under the second-order sufficient conditions stated in
  Proposition \ref{prop:DRO-boundary} (see the discussion following
  Theorem 3.1 in \citet{shapiro1989}.  Following the same
  lines of the reasoning in \cite[Lemma 3.1 -
  3.3]{shapiro1989}, we have a neighborhood $\mathcal{N}$ of
  $\beta_\ast$ such that if $\beta_n^{DRO}(\delta_n) \in \mathcal{N},$
  then
  \begin{align}
    \min_{\beta \in B}  \Psi_n(\beta) = \min_{\beta \in \mathcal{C}} \Phi_n(\beta),
    \label{approx-cone}
  \end{align}
  where $\mathcal{C}$ is the critical cone of directions given in the
  statement of Proposition \ref{prop:DRO-boundary}; here, the
  conditions which are required for applying these results in
  \citet{shapiro1989} are verified as follows: The conditions
  stated in Assumptions A.1, A.4 - A.5, B.4, C.4 and D are direct
  consequences of the continuous differentiability properties of
  $\ell(\cdot),$ compactness of $B,$ and finite moments assumed in the
  statement of Proposition \ref{prop:DRO-boundary} and Assumptions A2.a
  and A2.c in Section \ref{sec:assumption_results}. While the
  conditions in \cite[Assumptions A.2, A.6]{shapiro1989}
  follow from the compactness and aforementioned continuous
  differentiability properties, the conditions stated in Assumptions
  A.3, B.1 - B.3, C.5 and D of \citet{shapiro1989} are
  explicitly mentioned in the statement of Proposition
  \ref{prop:DRO-boundary}. Now, with the tightness of the collection
  $n^{1/2}\{\beta_n^{DRO}(\delta_n) - \beta_\ast\}$ verified as in
  Proposition \ref{prop:argmintightness}, we have that the probability
  of the event $\{\beta_n^{DRO}(\delta_n) \in \mathcal{N}\}$ is
  $1-o_p(1).$ Therefore, we have from (\ref{approx-cone}) and
  (\ref{approx-dro-boundary}) that,
  \begin{align*}
    n^{1/2}\left\{\beta_n^{DRO}(\delta_n) - \beta_\ast \right\}
    &= \arg \min_{u \in \mathcal{C}}
      \left\{\Phi_n(\beta_\ast + n^{-1/2}u) - \Phi_n(\beta_\ast)\right\}\\
    &= \arg \min_{u \in \mathcal{C}}
      \left[\left\{-H_n +  \eta^{1/2}D_\beta S_n(\beta_\ast) \right\}^\T u + 2^{-1}q(u)
      + o_p(1)\right],
  \end{align*}
  with probability $1-o_p(1),$ as $n \rightarrow \infty.$ As noted
  earlier, the small $o_p(1)$ term is uniform over compact sets of the
  variable $u.$ Due to central limit theorem, we have
  $H_n \Rightarrow H,$ where
  $H \sim \mathcal{N}[0,\text{cov}\{h(X,\beta_\ast)\}].$ We also have
  $ D_\beta S_n(\beta_\ast) \rightarrow D_\beta S(\beta_\ast),$ as
  $n \rightarrow \infty.$ With the cone $\mathcal{C}$ being nonempty as reasoned above
  and $\omega(\xi) = \arg\min_{u \in \mathcal{C}} \{ u^\T\xi +
  2^{-1}q(u)\}$ unique, we then obtain
  \begin{align*}
    n^{1/2}\left\{\beta_n^{DRO}(\delta_n) - \beta_\ast \right\}
    \Rightarrow \omega\left\{-H + \eta^{1/2}D_\beta S(\beta_\ast)\right\}.
  \end{align*}
  as a consequence of argmax/argmin continuous mapping theorem; see
  \citet[Corollary 3.2.3a]{van1996weak}.
\end{proof}

\begin{proof}[Proof  of Proposition \ref{prop:DRO-property}]
  Due to the continuous differentiability properties of $\ell(\cdot)$
  in Assumption A2.c and the compactness of the set $B,$ we have from
  \cite[Theorems 2.7.11 and 2.5.6]{van1996weak} that the class
  $\{\ell(X;\beta): \beta \in B\}$ is $P_\ast-$Donsker. Consequently,
  we have the uniform central limit theorem that,
  \begin{align*}
    {n}^{1/2}\left[E_{P_n}\{\ell(X;\beta)\} - E\{\ell(X;\beta)\} \right]
    \Rightarrow Z(\beta),
  \end{align*}
  as $n \rightarrow \infty,$ uniformly over continuous functions
  defined on the set $B.$ Similarly, applying the continuity
  properties of $\Vert D_x \ell(X;\beta) \Vert_p^2$ in Assumption
  A2.c, we have from \citet[Theorems 2.7.11 and 2.4.1]{van1996weak}
  that
  \begin{align*}
    \sup_{\beta \in B} \left\vert S_n(\beta)  - S(\beta) \right\vert
    \rightarrow  0
  \end{align*}
  as $n \rightarrow \infty$.  Since $n\delta_n \rightarrow \eta,$ we
  obtain by combining the above two convergences that,
    \begin{align*}
      \delta_n^{-1/2}\left[E_{P_n}\{\ell(X;\beta)\} + \delta_n^{1/2}S_n(\beta)
      - E\{\ell(X;\beta)\} \right] \Rightarrow \eta^{-1/2} Z(\beta) +
      S(\beta),
    \end{align*}
    uniformly.
    On the other hand, we have from Proposition \ref{Lem-2-Val-fn}
    that the DRO objective $\Psi_n(\beta)$ and
    $E_{P_n}[\ell(X;\beta)] + \delta_n^{1/2}S_n(\beta)$ differ only
    by $O_p(\delta_n).$ Therefore,
    \begin{align*}
      n^{1/2}\left[\Psi_n(\beta) - E\{\ell(X;\beta)\} \right]
      \Rightarrow Z(\beta) + \eta^{1/2} S(\beta),
    \end{align*}
    uniformly.  Recall that $B_\ast$ is the set of minimizers of
    $\min_{\beta \in B}E\{\ell(X;\beta)\}.$ Let us denote the optimal
    value $\min_{\beta \in B}E\{\ell(X;\beta)\}$ as $m.$ Due
    to the above uniform convergence and almost sure finiteness of
    $\sup_{\beta \in B} \left \vert  Z(\beta) +
      \eta^{1/2} S(\beta)\right\vert,$ given $\varepsilon > 0,$ there exists $N$
    large enough such that
    $\min_{\beta \in B_\ast} \Psi_n(\beta) < m + \varepsilon $ for all
    $n > N.$ Therefore, if the right hand side is singleton almost
    surely, we have
    \begin{align*}
      \arg\min_{\beta \in B} \Psi_n(\beta) \Rightarrow
      \arg\min_{\beta \in B_\ast} \left\{ Z(\beta) +
       \eta^{1/2}S(\beta) \right\},
    \end{align*}
    as $n \rightarrow \infty,$ as a consequence of the Argmin/argmax
    continuous mapping theorem; see \citet[Corollary
    3.2.3a]{van1996weak}.
    %
  \end{proof}

  \subsection{Statements and proofs of the results in Section
    \ref{sec:discussion}}
\label{ssec:discussion_proof}
\begin{proposition}
  Suppose that the support of $X$ is constrained to be contained in
  the set $\Omega = \{x \in \mathbb{R}^m: Ax \leq b\}$ specified in
  terms of linear constraints involving an $l \times m$ matrix $A$
  with linearly independent rows and $b \in \mathbb{R}^l.$ Consider
  the Wasserstein distance defined as in Definition \ref{defn:WD} with
  the transportation cost $c(x,y) = \Vert x- y\Vert_2^2.$ Suppose that
  $\delta_n = \eta n^{-1},$ $X$ has a probability density which is
  absolutely continuous with respect to the Lebesgue measure on
  $\mathbb{R}^m$ and the support $\Omega$ is compact. Then we have,
  \begin{equation}
  n^{1/2}\{\beta_n^{DRO}(\delta_n) - \beta_\ast\} \Rightarrow
    C^{-1}H - \eta^{1/2} C^{-1}D_\beta S(\beta_\ast),
    \label{eqn:prop_Omega_conv}
    \end{equation}
  as $n \rightarrow \infty.$

  \label{prop:DRO-constrained}
\end{proposition}

As in the proof of Proposition \ref{prop:DRO}, we first present the
constrained counterpart to Lemma \ref{lem:unconstr-opt} which is
useful for the setting considered in
Proposition~\ref{prop:DRO-constrained}.

\begin{lemma}
  For any $x,v \in \mathbb{R}^m, \lambda > 0, \varepsilon > 0,$
  $d \times d$ symmetric matrix $B,$ $l \times m$ matrix $A,$ and
  $b \in \mathbb{R}^l,$ we have
  \begin{align}
    \sup_{x:A(x + \varepsilon \Delta) \leq b} \left\{ v^\T \Delta -
    \lambda \Vert \Delta \Vert_2^2
    + \varepsilon \Delta^\T B \Delta \right\}
    = \frac{\Vert v \Vert_2^2}{4\lambda} - \frac{1}{4\lambda}
    \xi^\T H \xi,
    \label{constr-obj}
  \end{align}
  where $\xi = \{2\lambda \varepsilon^{-1}(Ax-b) + A \tilde{B} v\}^+,$
  $\tilde{B}$ is an $m \times m$ matrix given by the inverse of
  $(I_m - \varepsilon \lambda^{-1}B)$ with $I_m$ denoting the identity
  matrix, and $H$ is an $l \times l$ matrix given by the inverse of
  $A\tilde{B}A^\T.$
\label{lem:constr-opt}
\end{lemma}

\begin{proof}[Proof  of Lemma \ref{lem:constr-opt}]
  For any $x \in \Omega,$ we have $Ax \leq b.$ Consequently, the
  constrained optimization in (\ref{constr-obj}) is feasible for the
  choice $\Delta = 0.$ Then, due to Lagrange's theorem for convex
  duality, we have that the objective in (\ref{constr-obj}) equals
  \begin{align}
    \inf_{\mu \geq 0}
    &\sup_{\Delta \in \mathbb{R}^d}\left[
    v^\T \Delta - \lambda \Vert \Delta \Vert_2^2
      + \varepsilon \Delta^\T B \Delta - \mu^\T \left\{ A(x+\varepsilon \Delta)
      - b \right\}\right]\nonumber\\
    &\qquad\qquad=\inf_{\mu \geq 0}
      \left\{ -\mu^\T(Ax-b) + \sup_{\Delta \in \mathbb{R}^d} f(\Delta,\mu)
      \right\}
     \label{Lagrange-duality}
  \end{align}
  where, for any $\mu \geq 0, \Delta \in \mathbb{R}^d$ we define
 $f(\Delta,\mu)$ as,
  \begin{align}
    f(\Delta,\mu)=  \left(v - \varepsilon A^\T\mu \right)^\T \Delta
    - \lambda \Vert \Delta \Vert_2^2 + \varepsilon \Delta^\T B \Delta.
    \label{f-Delta-mu}
  \end{align}
  Utilizing the optimality condition that
  $v-\varepsilon A^\T\mu = 2(\lambda + \varepsilon B\Delta),$ we
  obtain
  \begin{align*}
    \sup_{\Delta \in \mathbb{R}^d} f(\Delta,\mu)
    = \frac{1}{4\lambda}(v - \varepsilon A^\T \mu)^\T \tilde{B}
    (v - \varepsilon A^\T \mu).
  \end{align*}
  Then, we obtain from (\ref{Lagrange-duality}) that
  \begin{align*}
    \inf_{\mu \geq 0}
    \left\{ -\mu^T(Ax-b) + \sup_{\Delta \in \mathbb{R}^d} f(\Delta,\mu)
    \right\}
    &= \frac{\Vert v \Vert_2^2}{4\lambda} +
      \inf_{\mu \geq 0} \left\{-\mu^\T \left(Ax-b +
      \frac{\varepsilon}{2\lambda}A \tilde{B}v \right) +
      \frac{\varepsilon}{4\lambda} \mu^\T A\tilde{B}A^\T\mu\right\}\\
    &=  \frac{\Vert v \Vert_2^2}{4\lambda} +
      \frac{\varepsilon}{2\lambda}\inf_{\mu \geq 0} \left( - \mu^\T \xi
      + \frac{\varepsilon}{2} \mu^\T  A\tilde{B}A^\T \mu\right).
  \end{align*}
  where $\xi = \{2\lambda \varepsilon^{-1}(Ax-b) + A \tilde{B} v\}^+$
  denotes the component-wise positive part.
   This is because, for any $\mu = (\mu_1,\ldots,\mu_l)$ which attains
   the infimum in the above left hand side, it is necessarily the case
   that $\mu_i = 0$ whenever the respective $\xi_i < 0$ for any
   $i = 1,\ldots,l.$ Moreover,
   \begin{align*}
     \inf_{\mu \geq 0} \left( - \mu^\T \xi
     + \frac{\varepsilon}{2} \mu^\T A \tilde{B} A^\T \mu\right)
     =      \inf_{\mu \in \mathbb{R}^l} \left( - \mu^\T \xi
      + \frac{\varepsilon}{2} \mu^\T A \tilde{B} A^\T \mu\right),
   \end{align*}
   because of the following reasoning: $\xi \geq 0$ component-wise and
   if any $\mu = (\mu_1,\ldots,\mu_l)$ which attains the optimum in
   the right-hand side is such that $\mu_i < 0$ for some $i,$ then one
   can strictly decrease the objective by increasing $\mu_i$ if the
   respective $\xi_i > 0,$ (or) not change the objective by making
   $\mu_i = 0.$ Consequently,
   \begin{align*}
     \inf_{\mu \geq 0} \left( - \mu^\T \nu
     + \frac{\varepsilon}{2} \mu^\T A \tilde{B} A^\T \mu\right)
          &=      \inf_{\mu \in \mathbb{R}^m} \left( - \mu^\T \xi
       + \frac{\varepsilon}{2} \mu^\T A\tilde{B} A^\T \mu\right)\\
       &= - 2^{-1}\varepsilon^{-1}\xi^\T \left(A\tilde{B}A^\T \right)^{-1}
         \xi,
   \end{align*}
   because $A$ is taken to have linearly independent rows and the
   respective optimality condition is
   $\xi - \varepsilon A\tilde{B}A^T\mu = 0.$ Therefore, we
   have from the Lagrange duality, (\ref{Lagrange-duality}) and the
   above simplication that the objective in (\ref{constr-obj}) equals
   $(4\lambda)^{-1}( \Vert v \Vert_2^2 -
     \xi^\T H \xi ),$
   thus concluding the proof.
\end{proof}

\begin{proof}[Proof  of Proposition \ref{prop:DRO-constrained}]
  Due to the presence of the constraints
  $\Omega = \{x \in \mathbb{R}^m: Ax \leq b\},$ we have
  $\Psi_n(\beta)$ as in the statement of Lemma \ref{Lem-1-Val-Fn} with
  $e_n(\beta,\lambda) = E_{P_n}\left[ f_{n}(X,\beta,\lambda)\right]$
  and
\begin{align*}
  f_{n}(x,\beta,\lambda)
  = \sup_{x: A(x+ n^{-1/2} \Delta) \leq b } \left[ n^{1/2}
  \left\{ \ell \big(x + n^{-1/2}\Delta;\beta\big) - \ell(x;\beta)\right\}
  - \lambda \Vert \Delta \Vert_q^2\right] - \frac{1}{4\lambda}
  \Vert D_x \ell(x;\beta) \Vert_2^2.
\end{align*}
Fixing $b > 0$ and $ \lambda_0 \in (0,1),$ consider any $\beta$ such
that $\Vert \beta \Vert_2 \leq b$ and $\lambda > \lambda_0.$ To apply
Lemma \ref{lem:constr-opt} for evaluating $f_{n}(x,\beta,\lambda)$ as
in the proof of Lemma \ref{lem:en-bnd}, we identify the respective
quantities in (\ref{constr-obj}) in the statement of Lemma
\ref{lem:constr-opt} as follows: Letting
$\varepsilon = n^{-1/2},$ $v = D_x\ell(x;\beta),$
$\bar{q} = \min\{2,q/(q-1)\},$ $B =D_{xx}\ell(x;\beta),$
$H_n(x,\beta,\lambda)$ be the inverse of
$A \left\{I_m - n^{-1/2} \lambda^{-1}
  D_{xx}\ell(X;\beta)\right\}^{-1}A^\T$ and
\[\xi_n(x,\beta,\lambda) = \left[2\lambda n^{1/2}(Ax - b) +
    A \left\{I_m - n^{-1/2} \lambda^{-1}
      D_{xx}\ell(X;\beta)\right\}^{-1} D_x\ell(x;\beta) \right]^+ \]
we have that
$f_n(x,\beta,\lambda) - (4\lambda)^{-1} \xi_n(x,\beta,\lambda)^\T
H_n(x,\beta,\lambda) \xi_n(x,\beta,\lambda)$
is upper and lower bounded, respectively, by
  \begin{align*}
    c_u (1 + \varepsilon^\prime) n^{-\bar{q}/2} \Vert \xi_n(x,\beta,
    \lambda) \Vert_2^2
    \ \lambda_0^{-\max\{2,\frac{1}{q-1}\}} \quad \text{ and } \quad
    -c_l \varepsilon^\prime n^{-1} \Vert \xi_n(x,\beta,\lambda) \Vert_2^2
    \lambda^{-2},
  \end{align*}
  for suitable positive constants $c_l,c_u$ which are, in turn,
  determined by the constants $b, d$ and $q.$

  Next, with $\Omega$ being compact, we have from the expression for
  $\xi_n(\cdot)$ and the uniform boundedness of
  $D_x\ell(x,\beta),D_{xx}\ell(x,\beta)$ (over the set
  $x \in \Omega, \Vert \beta \Vert \leq b$) that,
  \begin{align*}
    {\rm pr}\left\{   \left\Vert \xi_n(X,\beta,\lambda) \right \Vert_2 > 0
    \right\}  \leq {\rm pr}\left[ \min_{i=1,\ldots,l} \{b_i - (Ax)_i\} <
    M\lambda^{-1} n^{-1/2} \right],
  \end{align*}
  for some suitably large constant $M.$ The above right hand side is
  $O_p(\lambda n^{-1/2}),$ as $n \rightarrow \infty,$ since the
  distribution $X$ is absolutely continuous and satisfies
  $\mathrm{pr}(AX \leq b) = 1.$ Then, letting
  \begin{align*}
    a_n(\beta,\lambda) = \lambda^{-1} n^{1/2}E_{P_n}\left\{\xi_n(X, \beta,
    \lambda)^\T H_n(X,\beta,\lambda) \xi_n(X,\beta,\lambda)\right\},
  \end{align*}
  we have
  $\sup_{n, \lambda > \lambda_0, \Vert \beta \Vert_2 \leq b}
  a_n(\beta,\lambda) < \infty$ due to the uniform boundedness of
  $\xi_{n}(x,\beta,\lambda)$ over
  $n \geq 1, x \in \Omega, \Vert \beta \Vert \leq b$ and
  $\lambda > \lambda_0.$ With $e_n(\beta,\lambda)$ defined to equal
  $E_{P_n}[f_n(X,\beta,\lambda)],$ we therefore obtain,
  \begin{align*}
    e_n(\beta,\lambda) = 4^{-1}\lambda^{-2} n^{-1/2}a_n(\beta,\lambda)
    + O_p(n^{-\bar{q}/2}),
  \end{align*}
  where the convergence pertaining to the $O_p(\cdot)$ term is uniform
  over $(\beta,\lambda)$ such that $\Vert \beta \Vert_2 \leq b$ and
  $\lambda > \lambda_0.$ Likewise, due to the above lower bound for
  $f_n(\cdot),$ the $O_p(n^{-\bar{q}1/2})$ term
  $\lambda^2 O_p(n^{-\bar{q}/2})$ is bounded from below by an
  integral random variable, uniformly over all $\lambda > 0$ and
  $\Vert \beta \Vert \leq b.$ Furthermore, due to continuous
  differentiability of $\ell(\cdot)$ over compact $\Omega,$ we have
  that $a_n(\beta,\lambda)$ is lipschitz over
  $\lambda > \lambda_0, \Vert \beta \Vert_2 \leq b.$ Combining this
  with the above expression for $e_n(\beta,\lambda)$ and that of
  $\Psi_n(\beta)$ derived from Lemma \ref{Lem-1-Val-Fn}, we have,
  $n^{1/2}\left[ \Psi_{n}(\beta) - E_{P_{n}}\{\ell(X;\beta)\}
  \right]$ equals,
  \begin{align*}
    \inf_{\lambda \geq 0}\left[ \lambda \eta +\frac{1}{4\lambda }
    E_{P_{n}}\left\{\left\Vert D_{x}\ell(X;\beta)\right\Vert_{p}^{2}\right\}
    +  \frac{a_n(\beta,\lambda)}{4\lambda^2n^{1/2}}
     + O_p(n^{-\bar{q}/2})\right].
  \end{align*}
  The desired conclusion then follows by utilizing the uniform
  boundedness, lipschitzness of $a_n(\beta,\lambda)$ and proceeding
  as in the proofs of Propositions \ref{Lem-2-Val-fn} and
  \ref{prop:DRO} given earlier in this supplementary material.
\end{proof}

{ The following examples show that the convergence \eqref{eqn:prop_Omega_conv} may not hold if the set
$\Omega = \{x \in \mathbb{R}^m: Ax \leq b\}$ has equality constraints.}
\begin{example}
  For the linear regression example in Section \ref{sec:geo_insignts},
  suppose that the support for $X,$ denoted by the set
  $\Omega = \{x \in \mathbb{R}^2: Ax \leq b\},$ where the matrix $A$
  and vector $b$ are such that
  \[\Omega = \{(x_1,x_2) \in \mathbb{R}^2: x_1 - x_2 = 0\}.\]
  Suppose that $\delta_n = \eta n^{-1}.$ With the loss
  $\ell(x,y;\beta) = (y-\beta^\T x)^2$ and the transportation cost
  $c(\cdot)$ given as in \eqref{tr-cost-LinReg}, we have the following
  from the definition of $\phi_{\lambda}(\cdot)$ in the proof of Lemma
  \ref{Lem-1-Val-Fn}: for any $x = (x_1,x_2) \in \Omega,$ with $x_1$
  being equal to $x_2,$
  \begin{align*}
    \phi_{\lambda}(x;\beta,\lambda)
    &= \sup_{\Delta \in \mathbb{R} }
      \left\{ \left(y - \beta^\T x - n^{-1/2}\Delta \beta^\T1  \right)^2
      - \lambda n^{-1/2} 2^{2/q}\Delta^2 \right\}\\
    &= (y-\beta^\T x)^2 + n^{-1/2}\sup_{\Delta \in R}
      \left[ -2(y-\beta^\T x) \beta^\T 1 \Delta - \left\{ \lambda 2^{2/q}
      - (\beta^\T 1)^2n^{-1/2}\right\}\Delta^2\right]\\
    &= (y-\beta^\T x)^2 + n^{-1/2}\frac{(y-\beta^\T x)^2}
      {\lambda 2^{2/q}(\beta^\T 1)^{-2} - n^{-1/2}}
      = \frac{(y-\beta^\T x)^2}
      {1 - \lambda^{-1}2^{-2/q}(\beta^\T 1)^2 n^{-1/2}}.
  \end{align*}
  For the choice $\delta_n = \eta n^{-1},$ the distributionally robust optimization objective
  simplifies as below by exploiting the dual representation for
  $\Psi_n(\beta)$ in Lemma \ref{Lem-1-Val-Fn}:
  \begin{align*}
    \Psi_n(\beta)
    &= \inf_{\lambda \geq 0} \left\{ \lambda \eta n^{-1/2}
      + \frac{E_{P_n}(Y-\beta^\T X)^2}
      {1 - \lambda^{-1}2^{-2/q}(\beta^\T 1)^2 n^{-1/2}} \right\}\\
    &= E_{P_n}(Y-\beta^\T X)^2 +  n^{-1/2} \inf_{\mu \geq 0}
      \left\{\eta \mu + \mu^{-1}2^{-2/q}(\beta^\T 1)^2 E_{P_n}(Y-\beta^\T X)^2
      \right\} + n^{-1} \eta 2^{-2/q}(\beta^\T 1)^2 \\
    &= E_{P_n}(Y-\beta^\T X)^2 +  n^{-1/2} 2^{1-1/q} \eta^{1/2}
      \vert \beta^\T 1 \vert \{E_{P_n}(Y-\beta^\T X)^2\}^{1/2}
      + n^{-1}\eta 2^{-2/q}(\beta^\T 1)^2.
  \end{align*}
  Suppose that $\beta_\ast,$ denoting an optimal parameter minimizing
  $E\{(Y-\beta^\T X)^2\},$ is such that $\beta_\ast^\T 1 \neq 0.$ Then
  \[n \left\{\Psi_n(\beta_\ast + n^{-1/2}u) -
      \Psi_n(\beta_\ast)\right\} = H_n^\T u + u^T E_{P_n}\left( X
      X^\T\right)u + \eta^{1/2} D_\beta\tilde{S}(\beta_\ast)^\T u +
    \eta 2^{-2/q}(\beta_\ast^\T 1)^2 + o(1),\]
  where
  $H_n = -n^{1/2}E_{P_n}\left\{2(Y-\beta_\ast^\T X)X\right\}$ and
  $\tilde{S}(\beta) = 2^{1-1/q}\vert \beta^\T 1 \vert
  \{E_{P_n}(Y-\beta^\T X)^2\}^{1/2}.$ The above convergence happens
  uniformly in compact sets over $u$ and as $n \rightarrow \infty.$
  Consequently, when $C = E[XX^\T]$ is positive definite, we have the
  the following central limit theorem for the distributionally robust estimator
  $\beta_n^{DRO}(\delta_n)$ incorporating support constraint: As
  $n \rightarrow \infty,$
  \begin{align*}
    n^{1/2}\left\{ \beta_n^{DRO}(\delta_n) - \beta_\ast \right\} \Rightarrow
    C^{-1}H - \eta^{1/2} D_\beta \tilde{S}(\beta),
  \end{align*}
  where $H$ is normally distributed as in Theorem
  \ref{thm:levelsets-master}. Comparing this limiting result with that
  in Theorem \ref{thm:levelsets-master}, we see that the limit has
  changed with the introduction of support constraints via the term
  $D_\beta\tilde{S}(\beta),$ instead of $D_\beta S(\beta)$ appearing in
  Theorem \ref{thm:levelsets-master}. In particular, we see that the
  terms $S(\beta)$ and $\tilde{S}(\beta)$ differ as in,
  \begin{align*}
    \tilde{S}(\beta_\ast) = 2^{1/2-1/q} \frac{\vert \beta^T1 \vert}
    {\Vert \beta \Vert_p}S(\beta).
  \end{align*}
  \label{eg1:supp-constr-CLT}
\end{example}

{
\begin{example}
 Suppose that $\ell(x;\beta) = a + \beta^\T x + \beta^\T C \beta$ for some $a \in \mathbb{R}$ and  positive semi-definite $C.$  Let $r \leq m$ be a positive integer and the support for $X$ be given by $\Omega = \{x \in \mathbb{R}^m: Ax = b\},$ where the matrix $A$ is an $(r \times m)$ matrix with linearly independent rows and $b \in \mathbb{R}^r.$  Suppose that $\delta_n = \eta n^{-1}.$ With the transportation cost
  $c(\cdot)$ given by $c(x,x^\prime) = \Vert x - x^\prime \Vert_2^2,$ we have the following
  from the definition of $\phi_{\lambda}(\cdot)$ in the proof of Lemma
  \ref{Lem-1-Val-Fn}: for any $x \in \Omega,$ we have $Ax = b$ and
  \begin{align*}
    \phi_{\lambda}(x;\beta,\lambda)
    &= \ell(x;\beta) +  n^{-1/2} \sup_{\Delta}
      \left\{  \beta^\T \Delta
      - \lambda \Vert \Delta \Vert_2^2   : A(x + n^{-1/2}\Delta) = b \right\}\\
          &= \ell(x;\beta) +  n^{-1/2} \sup_{\Delta}
      \left\{  \beta^\T \Delta
      - \lambda \Vert \Delta \Vert_2^2   : A\Delta = 0\right\}\\
      &= \ell(x;\beta) +  n^{-1/2} \inf_{\mu \in \mathbb{R}^r} \sup_{\Delta}
      \left\{  (\beta - A^\T\mu)^\T \Delta
      - \lambda \Vert \Delta \Vert_2^2  \right\},
  \end{align*}
  as a consequence of convex duality. Then
\begin{align*}
    \phi_{\lambda}(x;\beta,\lambda) =  \ell(x;\beta) + \frac{n^{-1/2}}{4\lambda} \inf_{\mu \in \mathbb{R}^r} \Vert \beta - A^\T \mu\Vert_2^2 = \ell(x;\beta) + \Vert (\mathbb{I}_m - A^\T(AA^\T)^{-1}A)\beta\Vert_2^2,
\end{align*}
where $I_m$ is the $m \times m$ identity matrix. For the choice $\delta_n = \eta n^{-1},$
we obtain the following from the dual representation in Lemma \ref{Lem-1-Val-Fn}:
  \begin{align*}
    \Psi_n(\beta)
    &= E_{P_n} \left\{ \ell(X;\beta) \right\} + \inf_{\lambda \geq 0}
    \left\{ \lambda \eta n^{-1/2}  + \frac{n^{-1/2}}{4\lambda}  \Vert (\mathbb{I}_m - A^\T(AA^\T)^{-1}A) \beta \Vert_2^2 \right\}\\
    &= E_{P_n} \left\{ \ell(X;\beta) \right\} + \eta^{1/2}n^{-1/2} \Vert (\mathbb{I}_m - A^\T(AA^\T)^{-1}A) \beta \Vert_2\\
     &= E_{P_n} \left\{ \ell(X;\beta) \right\} + \delta_n^{1/2} \Vert P_{\mathcal{N}(A)} \beta  \Vert_2,
  \end{align*}
  where $P_{\mathcal{N}(A)} = I_m - A^\T(AA^\T)^{-1}A$ is the matrix for projecting onto the null space of A.
Letting $H_n = n^{1/2}E_{P_n}\left\{h(X;\beta)\right\}$ and $\tilde{S}(\beta) = E_{P_n} \left\{\Vert P_{\mathcal{N}(A)} \beta \Vert_2^2 \right\}^{1/2},$
  \[n \left\{\Psi_n(\beta_\ast + n^{-1/2}u) -
      \Psi_n(\beta_\ast)\right\} = H_n^\T u + u^T C u + \eta^{1/2} D_\beta\tilde{S}(\beta_\ast)^\T u +
     + o(1),\]
     as $n \rightarrow \infty$ and uniformly in compact sets over $u.$ Consequently,
  \begin{align*}
    n^{1/2}\left\{ \beta_n^{DRO}(\delta_n) - \beta_\ast \right\} \Rightarrow
    C^{-1}H - \eta^{1/2} D_\beta \tilde{S}(\beta),
  \end{align*}
  where $H$ is normally distributed as in Theorem
  \ref{thm:levelsets-master}. With $S(\beta) = \Vert \beta \Vert_2$ in this example, we see that the introduction of support constraint results in a bias term that differs from that in Theorem \ref{thm:levelsets-master} by,
  \begin{align*}
      \tilde{S}(\beta) = \frac{\Vert P_{\mathcal{N}(A)} \beta \Vert} {\Vert \beta \Vert_2 } S(\beta),
  \end{align*}
  where $P_{\mathcal{N}(A)}$ is the projection matrix for projecting onto the null space of the matrix $A.$ 
\end{example}
}

\color{black}
\section{Proofs of Propositions \protect\ref{prop:RWPEq} - \protect\ref%
{prop:RWPcont}}

\label{sec:proofs-props-master-thm} In this section we present the proofs of
Propositions \ref{prop:RWPEq} - \ref{prop:RWPcont}, which are useful towards
establishing the convergence of the last component of the triple considered
in Theorem \ref{thm:master}.

\begin{proof}[Proof  of Proposition \ref{prop:RWPEq}]
By utilizing the duality for linear semi-infinite programs as in the proof
of Proposition 3 of \citet{blanchet2016robust}, for $\beta _{\ast
}+n^{-1/2}u\in \Theta ,$ we obtain that
\begin{align*}
nR_{n}(\beta _{\ast }+n^{-1/2}u)&= \max_{\xi }\left( -\sum_{i=1}^{n}\xi ^{{%
\mathrm{\scriptscriptstyle T}}}h(X_{i},\beta _{\ast }+n^{-1/2}u)\right. \\
&-\left. \sum_{i=1}^{n}\max_{\Delta :X_{i}+\Delta \in \Omega
}\left[ \xi ^{{\mathrm{\scriptscriptstyle T}}}\left\{ h(X_{i}+\Delta ,\beta
_{\ast }+n^{-1/2}u)-h(X_{i},\beta _{\ast }+n^{-1/2}u)\right\} -\left\Vert
\Delta \right\Vert _{q}^{2}\right] \right) .
\end{align*}%
As a result,
\begin{align*}
& nR_{n}(\beta _{\ast }+n^{-1/2}u)=\max_{\xi }\left[
-\sum_{i=1}^{n}\max_{\Delta :X_{i}+\Delta \in \Omega }\left\{ \xi ^{{\mathrm{%
\scriptscriptstyle T}}}h(X_{i}+\Delta ,\beta _{\ast }+n^{-1/2}u)-\left\Vert
\Delta \right\Vert _{q}^{2}\right\} \right] \\
& \quad =\max_{\xi }\left( -\sum_{i=1}^{n}\xi ^{{\mathrm{\scriptscriptstyle T%
}}}h(X_{i},\beta _{\ast })-\sum_{i=1}^{n}\max_{X_{i}+\Delta \in \Omega
}\left[ \xi ^{{\mathrm{\scriptscriptstyle T}}}\left\{ h(X_{i}+\Delta ,\beta
_{\ast }+n^{-1/2}u)-h(X_{i},\beta _{\ast })\right\} -\left\Vert \Delta
\right\Vert _{q}^{2}\right] \right) .
\end{align*}%
By rescaling $\xi =n^{1/2}\xi ,\Delta =n^{1/2}\Delta $ and letting $H_{n}=%
{n^{-1/2}}\sum_{i=1}^{n}h(X_{i},\theta _{\ast }),$ we obtain,
\begin{equation*}
nR_{n}(\beta _{\ast }+n^{-1/2}u)=\max_{\xi }\left\{ -\xi ^{{\mathrm{%
\scriptscriptstyle T}}}H_{n}-M_{n}(\xi ,u)\right\} ,
\end{equation*}%
where
\begin{align}
& M_{n}(\xi ,u) \notag \\
=&\frac{1}{n}\sum_{i=1}^{n}\max_{\Delta :X_{i}+n^{-1/2}\Delta
\in \Omega }\left[ n^{1/2}\xi ^{{\mathrm{\scriptscriptstyle T}}}\left\{
h(X_{i}+n^{-1/2}\Delta ,\beta _{\ast }+n^{-1/2}u)-h(X_{i},\beta _{\ast
})\right\} -\left\Vert \Delta \right\Vert _{q}^{2}\right]
\label{Mn-rep-for-lb} \\
  =&\frac{1}{n}\sum_{i=1}^{n}\max_{\Delta :X_{i}+n^{-1/2}\Delta \in \Omega
}\left\{ \xi ^{{\mathrm{\scriptscriptstyle T}}}\int_{0}^{1}
D_{x}h\left( X_{i} +n^{-1/2}t\Delta,\beta _{\ast }+n^{-1/2}tu%
\right)\Delta\mathrm{d}t  \right. \notag \\
&\left. + \xi ^{{\mathrm{\scriptscriptstyle T}}}\int_{0}^{1} D_{\beta }h\left( X_{i}+n^{-1/2}t\Delta ,\beta
_{\ast }+n^{-1/2}tu\right) u \mathrm{d}t-\Vert \Delta \Vert
_{q}^{2}\right\} ,  \notag
\end{align}%
where the latter equality follows from the fundamental theorem of calculus.
This completes the proof of the first part of Proposition~\ref{prop:RWPEq}.

For the second part, we first show $\beta _{\ast }\in \Theta .$ For any
non-zero $\xi \in \mathbb{R}^{d},$ we have $E\left\{ \xi ^{{\mathrm{%
\scriptscriptstyle T}}}h(X,\beta _{\ast })\right\} =0$, due to Assumption
A2.b. We claim $0$ lies in the interior of $\mathrm{conv}\left( \left\{ \xi
^{{\mathrm{\scriptscriptstyle T}}}h(x,\beta _{\ast }),x\in \Omega \right\}
\right) .$ Otherwise, we must have $h(X,\beta ^{\ast })=0$, almost
surely, and $\xi ^{{\mathrm{\scriptscriptstyle T}}}h(x,\beta ^{\ast })$ have
the same sign, for all $x\in \Omega .$ Without loss of generality, we assume
$\xi ^{{\mathrm{\scriptscriptstyle T}}}h(x,\beta ^{\ast })\geq 0$ for all $%
x\in \Omega .$ Then, we have $D_{x}\left\{ \xi ^{{\mathrm{\scriptscriptstyle %
T}}}h(x,\beta ^{\ast })\right\} =0$, almost surely, which leads to a
contradiction to $E\left\{ D_{x}h(X,\beta _{\ast })D_{x}h(X,\beta _{\ast })^{%
{\mathrm{\scriptscriptstyle T}}}\right\} \succ 0.$ Therefore, there exists $%
\underline{x}_{\xi },\overline{x}_{\xi }\in $ $\Omega $ such as%
\begin{equation*}
\xi ^{{\mathrm{\scriptscriptstyle T}}}h(\underline{x}_{\xi },\beta _{\ast
})<0<\xi ^{{\mathrm{\scriptscriptstyle T}}}h(\overline{x}_{\xi },\beta
_{\ast }).
\end{equation*}%
If $\beta _{\ast }\notin \Theta ,$ which means $0$ lies on the boundary of $%
\mathrm{conv}\left[ \left\{ h(x,\beta _{\ast }),x\in \Omega \right\} \right]
,$ by applying the supporting hyperplane theorem; see, for example, \citet[section
2.5.2]{boyd2004convex}, there exists a non-zero $\xi $ such that for
all $x\in \Omega ,$
\begin{equation*}
\xi ^{{\mathrm{\scriptscriptstyle T}}}h(x,\beta _{\ast })\leq 0,
\end{equation*}%
which leads to a contradiction.

Since $\beta _{\ast }\in \Theta ,$ there exists $\epsilon >0$ such as $%
B_{\epsilon }\left( 0\right) \subset \mathrm{conv}\left[ \left\{ h(x,\beta
_{\ast }),x\in \Omega \right\} \right] .$ Consider basis points $%
e_{i}=(0,\ldots ,1,\ldots ,0)^{{\mathrm{\scriptscriptstyle T}}},$ whose
coordinates are all zero, except the $i$-th entry that equals one. So, $%
\mathrm{conv}\left[ \left\{ \epsilon e_{i}\right\} _{i=1}^{d}\cup \left\{
-\epsilon e_{i}\right\} _{i=1}^{d}\right] \subset \mathrm{conv}\left[
\left\{ h(x,\beta _{\ast }),x\in \Omega \right\} \right] $ is a neighborhood
of $0.$ To simplify the notation, let $y_{i}=\epsilon e_{i}$ for $%
i=1,2,\ldots ,d$, and $y_{i}=-\epsilon e_{i-d}$ for $i=d+1,d+2,\ldots ,2d.$
By Carath\'{e}odory's theorem; see, for example, \citet[Theorem 17.1]{rockafellar1970convex},
we have for each $y_{i},$ there exists $x_{i,1}\ldots x_{i,d+1}$ such that $%
y_{i}$ is a convex combination of $h(x_{i,1},\beta )\ldots h(x_{i,d+1},\beta
).$ Then, due to the continuity of $D_{\beta }h(x,\beta )$
around $\beta _{\ast }$ in Assumption A2.c, there exists a neighborhood of $%
\beta _{\ast },$ $B_{\epsilon }\left( \beta _{\ast }\right) $, such that for
all $\beta \in B_{\epsilon }\left( \beta _{\ast }\right) ,$ for $%
i=1,2,\ldots ,2d$ and $j=1,2,\ldots ,d+1,$
\begin{equation*}
\left\Vert h(x_{i,j},\beta )-h(x_{i,j},\beta )\right\Vert _{2}<\epsilon /2.
\end{equation*}%
Then by applying the same convex combination to obtain $y_{i}^{\beta }$, we
have for all $i=1,2,\ldots ,2d,$ $\left\Vert y_{i}^{\beta }-y_{i}\right\Vert
_{2}<\epsilon /2.$ Therefore, $\mathrm{conv}\left( \left\{ y_{i}^{\beta
}\right\} _{i=1}^{2d}\right) \subset \mathrm{conv}\left[\left\{ h(x,\beta
),x\in \Omega \right\} \right] $ is a neighborhood of $0,$ which completes
the proof. \end{proof}

A key component of the proofs of the upper and lower bounds for $%
nR_{n}(\beta_{\ast}+n^{-1/2}u)$ is the following tightness result.

\begin{lemma}
For any $\varepsilon,K>0,$ there exists $n_{0}>0$ and $b\in(0,\infty)$ such
that
\begin{equation*}
{\rm pr}\left[ \max_{\left\Vert \xi\right\Vert _{q}\geq
b}\{-\xi^{\T}H_{n}-M_{n}(\xi,u)\}>0\right] \leq\varepsilon,
\end{equation*}
for all $n\geq n_{0}$ and uniformly over $u$ such that $\Vert u\Vert_{2}\leq
K$. \label{lemmaksai}
\end{lemma}

\begin{lemma}
For any positive constants $b,c_{0}$ and any bounded set $C \in \mathbb{R}^d$, we have
\begin{align*}
& \frac{1}{n}\sum_{i=1}^{n}\left[ \left\Vert \left\{ D_{x}h(X,\beta _{\ast
})\right\} ^{\T}\xi \right\Vert _{p}^{2}+\xi ^{\T}D_{\beta }h(X_{i},\beta
_{\ast })u\right] \mathbb{I}(X_{i}\in C) \\
& \hspace{50pt}\rightarrow E\left( \left[ \left\Vert \left\{
D_{x}h(X,\beta _{\ast })\right\} ^{\T}\xi \right\Vert _{p}^{2}+\xi
^{\T}D_{\beta }h(X_{i},\beta _{\ast })u\right] \mathbb{I}(X_{i}\in C)%
\right) ,
\end{align*}%
uniformly over $\left\Vert \xi \right\Vert _{q}\leq b$ and $\Vert u\Vert
_{2}\leq K$ in probability as $n\rightarrow \infty .$ \label{Lem-LLN-uniform}
\end{lemma}

Proofs of Lemmas \ref{lemmaksai} and \ref{Lem-LLN-uniform} are presented in
Section \ref{appendix:proof_technical_lemma}. The following definitions are useful in the
proofs of Proposition \ref{prop:RWPUB} and Lemma \ref{lemmaksai}. For a
fixed $u,\Delta ,$ let
\begin{equation}
I(X_{i},\Delta ,u)=I_{1}(X_{i},\Delta ,u)+I_{2}(X_{i},\Delta ,u),
\label{defn-I-terms}
\end{equation}%
where $i\in \{1,\ldots ,n\},$
\begin{align*}
I_{1}(X_{i},\Delta ,u)& =\int_{0}^{1}\left\{ D_{x}h\left( X_{i}+n^{-1/2}t%
\Delta ,\beta _{\ast }+n^{-1/2}tu\right) -D_{x}h\left(
X_{i},\beta _{\ast }\right) \right\} \Delta \mathrm{d}t\quad \text{ and } \\
I_{2}(X_{i},\Delta ,u)& =\int_{0}^{1}\left\{ D_{\beta }h\left( X_{i}+n^{-1/2}t
\Delta ,\beta _{\ast }+n^{-1/2}tu\right) -D_{\beta
}h\left( X_{i},\beta _{\ast }\right) \right\} u\mathrm{d}t.
\end{align*}%
Then, we have
\begin{equation*}
M_{n}(\xi ,u)=\frac{1}{n}\sum_{i=1}^{n}\left[ \xi ^{{\mathrm{%
\scriptscriptstyle T}}}D_{\beta }h\left( X_{i},\beta _{\ast }\right)
u+\max_{\Delta :X_{i}+n^{-1/2}\Delta \in \Omega }\left\{ \xi ^{{\mathrm{%
\scriptscriptstyle T}}}D_{x}h\left( X_{i},\beta _{\ast }\right) \Delta +\xi
^{{\mathrm{\scriptscriptstyle T}}}I(X_{i},\Delta ,u)-\Vert \Delta \Vert
_{q}^{2}\right\} \right] .
\end{equation*}

In addition,  for $\xi \neq 0,$ we write $%
\bar{\xi}=\xi /\left\Vert \xi \right\Vert _{p}.$ Let us define the vector $%
V_{i}(\bar{\xi})=D_{x}h(X_{i},\beta _{\ast })^{{\mathrm{\scriptscriptstyle T}%
}}\bar{\xi}$ and put
\begin{equation}
\Delta _{i}^{\prime }=\Delta _{i}^{\prime }(\bar{\xi})=
\begin{cases}
|V_{i}(\bar{\xi}%
)|^{p/q}\text{sgn}\{V_{i}(\bar{\xi})\} \quad &q \in (1,\infty) \\
 V_{i}(\bar{\xi})\mathbb{I}[{|V_{i}(\bar{\xi})|=\max_j\{|V_{j}(\bar{\xi})|\}}]\quad &q = 1 \\
\text{sgn}\{V_{i}(\bar{\xi})\}\quad &q = \infty.
\end{cases}
\label{delta_prime}
\end{equation}

\begin{proof}[Proof  of Proposition \ref{prop:RWPUB}]
First observe that $R_{n}(\cdot) \geq0$ (consider the choice $\xi= 0
$). Given $K,\varepsilon> 0,$ define the event,
\begin{align*}
\mathcal{A}_{n} = \left\{ nR_{n}\left(\beta_{\ast}+ n^{-1/2}u\right) = \max_{\Vert
\xi\Vert_{p} \leq b} \left\{ -\xi^{{ \mathrm{\scriptscriptstyle T} }}H_{n}
- M_{n}(\xi,u)\right\} \text{ for all } u \text{ such that } \Vert u
\Vert_{2} \leq K\right\} .
\end{align*}
where $b > 0$ is such that ${\rm pr}(\mathcal{A}_{n}) \geq1-\varepsilon$ for $n \geq n'$.  Such a $b
\in(0,\infty)$ exists because of Lemma \ref{lemmaksai} and the fact that the set $\{\beta_* + n^{-1/2} u \mid \|u\|_2 \leq K\}$ will eventually become a subset of $\Theta$ when $n$ is sufficiently large.

Next, for any $c_{0}>0,\epsilon _{0}>0$ define
\begin{equation*}
M_{n}^{\prime }(\xi ,u,c_{0},\epsilon _{0})=\frac{1}{n}\sum_{i=1}^{n}\left\{
\xi ^{{\mathrm{\scriptscriptstyle T}}}D_{x}h(X,\beta _{\ast })\bar{\Delta}%
_{i}-\Vert \bar{\Delta}_{i}\Vert _{q}^{2}+\xi ^{{\mathrm{\scriptscriptstyle T%
}}}I(X_{i},\bar{\Delta}_{i},u)+\xi ^{{\mathrm{\scriptscriptstyle T}}%
}D_{\beta }h(X,\beta _{\ast })u\right\} \mathbb{I}\left( X_{i}\in
C_{0}^{\epsilon _{0}}\right) ,
\end{equation*}%
where $C_{0}=\{w\in \Omega :\left\Vert w\right\Vert _{p}\leq c_{0}\},$ $%
I(X_{i},\Delta ,u)$ is defined as in \eqref{defn-I-terms} and $\bar{\Delta}%
_{i}=c_{i}\Delta _{i}^{\prime }$, which is defined in (\ref{delta_prime})
with $c_{i}$ chosen so that
\begin{equation*}
\left\Vert \bar{\Delta}_{i}\right\Vert _{q}=\frac{1}{2}\left\Vert
D_{x}h(X_{i},\beta _{\ast })^{{\mathrm{\scriptscriptstyle T}}}\xi
\right\Vert _{p}.
\end{equation*}%
Since $D_{x}h(X_{i},\beta _{\ast })$ is continuous, $\left\Vert \xi
\right\Vert _{p}$ is bounded, and $C_{0}$ is compact, we have
\begin{equation*}
\sup_{x\in C_{0}}\left\{ \frac{1}{2}\left\Vert D_{x}h(X_{i},\beta _{\ast })^{%
{\mathrm{\scriptscriptstyle T}}}\xi \right\Vert _{p}\right\} <\infty .
\end{equation*}%
Therefore, there exists $n_{1}>0$ such that for all $n\geq n_{1}$ and $X_i \in C_0^{\epsilon_0}$,  we have $%
X_{i}+n^{-1/2}\bar{\Delta}_{i}\in C_{0},$ and thus $M_{n}(\xi ,u)\geq
M_{n}^{\prime }(\xi ,u,c_{0},\epsilon _{0}),$ for every $u$ and $n\geq n_{1}.
$ With these definitions, observe that
\begin{align}
\max_{X_{i}+n^{-1/2} \Delta \in \Omega }\left\{ \xi ^{{\mathrm{\scriptscriptstyle T}}%
}D_{x}h(X_{i},\beta _{\ast })\Delta -\left\Vert \Delta \right\Vert
_{q}^{2}\right\} & =\xi ^{{\mathrm{\scriptscriptstyle T}}}D_{x}h(X_{i},\beta
_{\ast })\bar{\Delta}_{i}-\Vert \bar{\Delta}_{i}\Vert _{q}^{2}  \notag \\
& =\frac{1}{4}\left\Vert \left\{ D_{x}h(X,\beta _{\ast })\right\} ^{{\mathrm{%
\scriptscriptstyle T}}}\xi \right\Vert _{p}^{2}  \label{inter-rwp-ub1}
\end{align}%
and
\begin{equation}
\max_{\left\Vert \xi \right\Vert _{q}\leq b}\left\{ -\xi ^{{\mathrm{%
\scriptscriptstyle T}}}H_{n}-M_{n}(\xi ,u)\right\} \leq \max_{\left\Vert \xi
\right\Vert _{q}\leq b}\left\{ -\xi ^{{\mathrm{\scriptscriptstyle T}}%
}H_{n}-M_{n}^{\prime }(\xi ,u,c_{0},\epsilon _{0})\right\} .
\label{inter-rwp-ub2}
\end{equation}

Next, define
\begin{align*}
\hat{M}_{n}(\xi ,u,c_{0},\epsilon _{0})& =\frac{1}{n}\sum_{i=1}^{n}\left\{
\xi ^{{\mathrm{\scriptscriptstyle T}}}D_{x}h(X_{i},\beta _{\ast })\bar{\Delta%
}_{i}-\left\Vert \bar{\Delta}_{i}\right\Vert _{q}^{2}+\xi ^{{\mathrm{%
\scriptscriptstyle T}}}D_{\beta }h(X_{i},\beta _{\ast })u\right\} \mathbb{I}%
(X_{i}\in C_{0}^{\epsilon _{0}}) \\
& =\frac{1}{n}\sum_{i=1}^{n}\left\{ \frac{1}{4}\left\Vert D_{x}h(X_{i},\beta
_{\ast })^{{\mathrm{\scriptscriptstyle T}}}\xi \right\Vert _{p}^{2}+\xi ^{{%
\mathrm{\scriptscriptstyle T}}}D_{\beta }h(X_{i},\beta _{\ast })u\right\}
\mathbb{I}(X_{i}\in C_{0}^{\epsilon _{0}}),
\end{align*}%
where the equality follows from \eqref{inter-rwp-ub1}. Due to Lemma \ref%
{Lem-LLN-uniform}, we have
\begin{equation*}
\hat{M}_{n}(\xi ,u,c_{0},\epsilon _{0})\rightarrow E\left( \left[ \frac{1}{4}%
\left\Vert \left\{ D_{x}h(X,\beta _{\ast })\right\} ^{{\mathrm{%
\scriptscriptstyle T}}}\xi \right\Vert _{p}^{2}+\xi ^{{\mathrm{%
\scriptscriptstyle T}}}D_{\beta }h(X,\beta _{\ast })u\right] \mathbb{I}(X\in
C_{0}^{\epsilon _{0}})\right) .
\end{equation*}%
in probability, uniformly over $\Vert \xi \Vert _{p}\leq b$ and $\Vert
u\Vert _{2}\leq K.$ Furthermore,
\begin{equation}
\sup_{\left\Vert \xi \right\Vert _{p}\leq b}\left\vert \hat{M}_{n}(\xi
,u,c_{0},\epsilon _{0})-M_{n}^{\prime }(\xi ,u,c_{0},\epsilon
_{0})\right\vert \rightarrow 0,  \label{inter-rwp-ub-3}
\end{equation}%
because, from the uniform continuity of $D_{\beta }h(\cdot )$ and $%
D_{x}h(\cdot )$ in compact sets, we have that
\begin{equation}
|\xi ^{{\mathrm{\scriptscriptstyle T}}}I(X_{i},\bar{\Delta}_{i},u)|\mathbb{I}%
(X_{i}\in C_{0}^{\epsilon _{0}})\rightarrow 0,  \label{inter-rwp-ub-4}
\end{equation}%
uniformly over $\Vert \xi \Vert _{p}\leq b$ and $\Vert u\Vert _{2}\leq K.$
Combining the observations in \eqref{inter-rwp-ub-3} and %
\eqref{inter-rwp-ub-4}, we obtain that for any $\varepsilon ^{\prime }>0$
there exists $n_{0}\geq n_{1}$ sufficiently large such that,
\begin{align*}
& \max_{\Vert \xi \Vert _{p}\leq b}\left\{ -\xi ^{{\mathrm{\scriptscriptstyle
T}}}H_{n}-M_{n}^{\prime }(\xi ,u,c_{0},\epsilon _{0})\right\}  \\
& \hspace{30pt}\leq \max_{\Vert \xi \Vert _{p}\leq b}\left\{ -\xi ^{{\mathrm{%
\scriptscriptstyle T}}}H_{n}-E\left( \left[ \frac{1}{4}\left\Vert \left\{
D_{x}h(X,\beta _{\ast })\right\} ^{\T}\xi
\right\Vert _{p}^{2}+\xi ^{{\mathrm{\scriptscriptstyle T}}}D_{\beta
}h(X,\beta _{\ast })u\right] \mathbb{I}(X\in C_{0}^{\epsilon _{0}})\right)
\right\} +\varepsilon ^{\prime }.
\end{align*}%
Then the statement of Proposition \ref{prop:RWPUB} follows from %
\eqref{inter-rwp-ub2}, the definition of the event $\mathcal{A}_{n}$ and the
observation that ${\rm pr}(\mathcal{A}_{n})\geq 1-\varepsilon .$ \end{proof}

\begin{proof}[Proof  of Proposition \ref{prop:RWPLB}]
For the lower bound, we reexpress the expression for $M_{n}(\xi ,u)$ in %
\eqref{Mn-rep-for-lb} as follows:
\begin{align}
M_{n}(\xi ,u)& \leq \frac{1}{n}\sum_{i=1}^{n}\max_{\Delta \in \mathbb{R}%
^{d}}\left[ n^{1/2}\xi ^{{\mathrm{\scriptscriptstyle T}}}\left\{
h(X_{i}+n^{-1/2}\Delta ,\beta _{\ast }+n^{-1/2}u)-h(X_{i},\beta _{\ast
}+n^{-1/2}u)\right\} -\left\Vert \Delta \right\Vert _{q}^{2}\right]  \notag
\\
& \hspace{30pt}+\frac{1}{n}\sum_{i=1}^{n}n^{1/2}\xi ^{{\mathrm{%
\scriptscriptstyle T}}}\left\{ h(X_{i},\beta _{\ast
}+n^{-1/2}u)-h(X_{i},\beta _{\ast })\right\} .  \label{inter-lb-1}
\end{align}


Employing the fundamental theorem of calculus, we obtain that
\begin{align*}
& \frac{1}{n}\sum_{i=1}^{n}n^{1/2}\xi ^{{\mathrm{\scriptscriptstyle T}}%
}\left\{ h(X_{i},\beta _{\ast }+n^{-1/2}u)-h(X_{i},\beta _{\ast })\right\} =%
\frac{1}{n}\sum_{i=1}^{n}\int_{0}^{1}\xi ^{{\mathrm{\scriptscriptstyle T}}%
}D_{\beta }h(X_{i},\beta _{\ast }+tn^{-1/2}u)u\mathrm{d}t \\
& \quad \quad =\xi ^{{\mathrm{\scriptscriptstyle T}}}\left\{ \frac{1}{n}%
\sum_{i=1}^{n}D_{\beta }h(X_{i},\beta _{\ast })\right\} u+\frac{1}{n}%
\sum_{i=1}^{n}\int_{0}^{1}\xi ^{{\mathrm{\scriptscriptstyle T}}}\left\{
D_{\beta }h(X_{i},\beta _{\ast }+tn^{-1/2}u)-D_{\beta }h(X_{i},\beta _{\ast
})\right\} u\mathrm{d}t \\
& \quad \quad \leq \xi ^{{\mathrm{\scriptscriptstyle T}}}\left\{\frac{1}{n}%
\sum_{i=1}^{n}D_{\beta }h(X_{i},\beta _{\ast })\right\} u+\Vert \xi \Vert _{p}%
\frac{1}{n}\sum_{i=1}^{n}\int_{0}^{1}\left\Vert \left\{ D_{\beta
}h(X_{i},\beta _{\ast }+tn^{-1/2}u)-D_{\beta }h(X_{i},\beta _{\ast })\right\}
u\right\Vert _{q}\mathrm{d}t.
\end{align*}%
Then, given $\varepsilon ,\varepsilon ^{\prime }>0,$ due to continuity of $%
D_{\beta }h(\cdot )$ in Assumption A2.c, finiteness $E\{\bar{\kappa}(X_{i})\}$
and the law of large numbers, there exists $n_{0}$ sufficiently large such
that for all $n\geq n_{0},\Vert \xi \Vert _{p}\leq b,\Vert u\Vert _{2}\leq K,
$ we have
\begin{equation}
\frac{1}{n}\sum_{i=1}^{n}n^{1/2}\xi ^{{\mathrm{\scriptscriptstyle T}}%
}\left\{ h(X_{i},\beta _{\ast }+n^{-1/2}u)-h(X_{i},\beta _{\ast })\right\}
\leq \xi ^{{\mathrm{\scriptscriptstyle T}}}E\{D_{\beta }h(X,\beta _{\ast
})\}u+\varepsilon ^{\prime }/2,
\label{inter-lb-term2}
\end{equation}%
with probability exceeding $1-\varepsilon /2.$

Next, given $\nu ,\varepsilon ^{\prime \prime },b,K\in (0,\infty )$, it
follows from Assumption A2 and the same line of reasoning in the proof of
Proposition 5 in \citet{blanchet2016robust} that there exists $n_{0}$ such
that,
\begin{equation}
\sup_{\left\Vert \Delta \right\Vert _{q}\geq \nu n^{1/2}}\left[ n^{1/2}%
\xi ^{{\mathrm{\scriptscriptstyle T} }}\left\{ h(X_{i}+n^{-1/2}\Delta
,\beta _{\ast }+n^{-1/2}u)-h(X_{i},\beta _{\ast }+n^{-1/2}u)\right\}
-\left\Vert \Delta \right\Vert _{q}^{2}\right] \leq 0,
\label{inner-opt-inside-compact-set}
\end{equation}%
for all $n\geq n_{0},\Vert \xi \Vert _{p}\leq b,$ $\Vert u\Vert _{2}\leq K,$
and consequently, the first term in the right hand side of \eqref{inter-lb-1}
is bounded from above by
\begin{equation*}
\frac{1}{n}\sum_{i=1}^{n}\min \left\{ \frac{1}{4\left( 1-\varepsilon
^{\prime \prime }\right) }\Vert \xi ^{{\mathrm{\scriptscriptstyle T} }%
}D_{x}h(X_{i},\beta _{\ast }+n^{-1/2}u)\Vert _{p}^{2},c_{n}\right\} +\nu ,
\end{equation*}%
for some sequence $(c_{n}:n\geq 1)$ satisfying $c_{n}\rightarrow \infty $ as
$n\rightarrow \infty $ (the exact value of $c_{n}$ is not important). It
follows from Assumption A2.c that
\begin{align*}
& \frac{1}{n}\sum_{i=1}^{n}\min \left[ \frac{1}{4\left( 1-\varepsilon
^{\prime \prime }\right) }\left\Vert \left\{ D_{x}h(X_{i},\beta _{\ast
}+n^{-1/2}u)\right\} ^{{\mathrm{\scriptscriptstyle T} }}\xi \right\Vert
_{p}^{2},c_{n}\right] +\nu \\
& \hspace{50pt}\leq \frac{1}{n}\sum_{i=1}^{n}\min \left[ \frac{1}{4\left\{
1-\varepsilon ^{\prime \prime }\right\} }\Vert \left( D_{x}h(X,\beta _{\ast
})\right) ^{{\mathrm{\scriptscriptstyle T} }}\xi \Vert
_{p}^{2},c_{n}\right] +n^{-1/2}\Vert \xi \Vert _{p}\Vert u\Vert _{q}%
\frac{1}{n}\sum_{i=1}^{n}\kappa ^{\prime }(X_{i})+\nu .
\end{align*}%
Then, for $n_{0}$ suitably large, a similar application of Lemma \ref%
{Lem-LLN-uniform} as in Proposition \ref{prop:RWPUB} results in,
\begin{align*}
& \frac{1}{n}\sum_{i=1}^{n}\min \left[ \frac{1}{4\left( 1-\varepsilon
^{\prime \prime }\right) }\left\Vert \left\{ D_{x}h(X_{i},\beta _{\ast
}+n^{-1/2}u)\right\} ^{{\mathrm{\scriptscriptstyle T} }}\xi \right \Vert
_{p}^{2},c_{n}\right] +\nu \\
& \hspace{50pt}\leq \frac{1}{4\left( 1-\varepsilon ^{\prime \prime }\right) }%
E\left\Vert \left\{ D_{x}h(X,\beta _{\ast })\right\} ^{{\mathrm{%
\scriptscriptstyle T} }}\xi \right\Vert _{p}^{2}+\frac{\varepsilon ^{\prime }%
}{4}+\nu ,
\end{align*}%
for all $n\geq n_{0},\Vert \xi \Vert _{p}\leq b,\Vert u\Vert _{2}\leq K,$
with probability exceeding $1-\varepsilon /2.$ Choosing $\nu ,\varepsilon
^{\prime \prime }$ suitably small, we combine the above observation with
that in \eqref{inter-lb-term2} to obtain that,
\begin{align*}
& nR_{n}\left( \beta _{\ast }+n^{-1/2}u\right) =\max_{\xi }\left\{ -\xi ^{{
\mathrm{\scriptscriptstyle T} }}H_{n}-M_{n}(\xi ,u)\right\} \\
& \hspace{50pt}\geq \max_{\underset{\Vert u_{i}\Vert _{2}\leq K}{\Vert \xi
\Vert _{p}\leq b}}\left[ -\xi ^{{\mathrm{\scriptscriptstyle T} }}H_{n}-E%
\left\{\frac{1}{4}\left\Vert D_{x}h(X,\beta _{\ast })^{{\mathrm{%
\scriptscriptstyle T} }}\xi \right\Vert _{p}^{2}+\xi ^{{\mathrm{%
\scriptscriptstyle T} }}D_{\beta }h(X,\beta _{\ast })u\right\} \right]
-\varepsilon ^{\prime },
\end{align*}%
for all $n\geq n_{0},$ $\Vert u\Vert _{2}\leq K,$ with probability exceeding
$1-\varepsilon .$ \end{proof}

\begin{proof}[Proof  of Proposition \ref{prop:tightness-1}]
Due to equation \eqref{inter-mas-up}, given $\varepsilon>0$, there exists $a$ and $n_1$ such as
\[
{\rm pr}\{f_{up}(H_{n},u,b,c) >a\} < \varepsilon/2.
\]
for all $n>n_1$.
Recall Proposition \ref{prop:RWPUB} and apply specifically with $u = 0$, we have
\[\sup_{n\geq
\max\{n_{0},n_{1}\}} {\rm pr}\{nR_{n}(\beta_{\ast})\geq a +1\}\leq\varepsilon.\]
\end{proof}

Lemma \ref{lem-lip-cont-Mn} below is useful to prove Proposition \ref%
{prop:RWPcont}. Proof of Lemma \ref{lem-lip-cont-Mn} is presented in Section \ref{appendix:proof_technical_lemma}.

\begin{lemma}
Given any $K,b \in(0,\infty)$ and $\varepsilon\in(0,1),$ there exist
positive constants $n_{0},L$ such that,
\begin{align*}
\sup_{\Vert\xi\Vert_{p} \leq b}\left\vert M_{n}(\xi,u_{1}) - M_{n} (\xi
,u_{2})\right\vert \leq L \Vert u_{1} - u_{2} \Vert_{q},
\end{align*}
with probability exceeding $1-\varepsilon.$ \label{lem-lip-cont-Mn}
\end{lemma}

\begin{proof}[Proof  of Proposition \ref{prop:RWPcont}]
For any $u_{j}, j = 1,2,$ satisfying $\Vert u_{j} \Vert_{2} \leq K,$ let $%
\xi_{j}$ attain the supremum in the relation $n^{\rho/2}R_{n}(\beta_{\ast}+
n^{-1/2}u_{j}) = \sup_{\xi} \left\{ -\xi^{\T}H_{n} -
M_{n}(\xi,u_{j})\right\} .$ Then we have,
\begin{align}
\left\vert nR_{n}(\beta_{\ast}+ n^{-1/2}u_{1}) - nR_{n}(\beta_{\ast}+
n^{-1/2}u_{2}) \right\vert \leq\max_{j=1,2} \left\vert M_{n}(\xi_{j}, u_{1})
- M_{n}(\xi_{j},u_{2})\right\vert.  \label{inter-RWPpf-cont}
\end{align}
For the given choices of $\varepsilon, K,$ we have from Lemma \ref{lemmaksai}
that there exist positive constants $b$ and $n_{0}$ such that the optimal
choices $\xi_{j},$ $j = 1,2,$ satisfy $\Vert\xi_{j} \Vert_{p} \leq b,$ each
with probability exceeding $1-\varepsilon/3.$ Consequently, we have from %
\eqref{inter-RWPpf-cont} and Lemma \ref{lem-lip-cont-Mn} that
\begin{align*}
\sup_{\Vert u_{j} \Vert_{2} \leq K, j =1,2}\left\vert nR_{n}(\beta_{\ast}+
n^{-1/2}u_{1}) - nR_{n}(\beta_{\ast}+ n^{-1/2}u_{2}) \right\vert \leq L
\Vert u_{1} - u_{2} \Vert_{q},
\end{align*}
with probability exceeding $1-\varepsilon,$ for all $n$ suitably large. This
completes the proof of Proposition \ref{prop:RWPcont}.
\end{proof}

\section{Proofs of Propositions \protect\ref{prop:argmintightness} -
\protect\ref{prop-levset-lthan1}}
\label{ssec:proofs-props-mast-lev-sets} In this section, we provide proofs
of Propositions \ref{prop:argmintightness} - \ref{prop-levset-lthan1}, which
are key in the proof of Theorem \ref{thm:levelsets-master}.

\begin{proof}[Proof  of Proposition \ref{prop:argmintightness}]
Due to the convexity of $\ell(\cdot)$, we have that $V_{n}^{ERM}(\cdot)$ is
convex. In addition, for $\beta_{1},\beta_{2} \in\mathbb{R}^{d}$ and $%
\alpha\in(0,1),$ we have
\begin{align*}
\Psi_{n}\{\alpha\beta_{1} + (1-\alpha)\beta_{2}\} & = \sup_{P:D_{c}(P,P_{n})
\,\leq\, \delta_{n}} E_{P} \left\{ \ell(X;\alpha\beta_{1} +
(1-\alpha)\beta_{2})\right\} \\
& \leq\sup_{P:D_{c}(P,P_{n}) \,\leq\, \delta_{n}} \left[ \alpha E%
_{P} \left\{ \ell(X; \beta_{1})\right\} + (1-\alpha) E_{P} \left\{
\ell(X; \beta_{2})\right\} \right] \\
& \leq\alpha\sup_{P:D_{c}(P,P_{n}) \,\leq\, \delta_{n}} E_{P} \left\{
\ell(X; \beta_{1})\right\} + (1-\alpha) \sup_{P:D_{c}(P,P_{n}) \,\leq\,
\delta_{n}} E_{P} \left\{ \ell(X; \beta_{2})\right\} \\
& = \alpha\Psi_{n}(\beta_{1}) + (1-\alpha)\Psi_{n}(\beta_{2}).
\end{align*}
Due to the convexity of $\Psi_{n}(\cdot)$, we have that $V_{n}^{DRO}(\cdot)$
is also convex. Furthermore, due to the positive definiteness of $C = E %
\left\{ D_{\beta}h(X,\beta_{\ast}) \right\} $ in Assumption A2.b, the
smallest eigen value of $C,$ denoted by $\lambda_{\min}(C),$ is positive.
Equipped with these observations, we proceed as follows:

For a given $\varepsilon ,\varepsilon ^{\prime }>0,$ let $K_{1}$ be such
that $\sup_{n}{\rm pr}(\Vert H_{n}\Vert _{2}>K_{1})\leq \varepsilon ,$
\begin{equation*}
K_{2}=\max \left\{ K_{1},\sup_{v:\Vert v\Vert _{2}\leq K_{1}}f_{\eta
,\gamma }(v)\right\} ,\quad \text{ and }\quad K_{3}=2\frac{2K_{2}+\left\{
2\varepsilon ^{\prime }\lambda _{\min }(C)\right\}^{1/2}}{\lambda _{\min }(C)}.
\end{equation*}%
Observe from the definition of $f_{\eta ,\gamma }(\cdot )$ that $K_{2}\in
(0,+\infty ).$ Due to Propositions \ref{prop:ERM} and \ref{prop:DRO}, there
exists $n_{0}$ such that,
\begin{align*}
V_{n}^{ERM}(u)& \geq H_{n}^{\T}u+\frac{1}{2}u^{\T}Cu-\varepsilon ^{\prime }%
\text{ and } \\
V_{n}^{DRO}(u)& \geq f_{\eta ,\gamma }(H_{n})^{\T}u+\frac{1}{2}%
u^{\T}Cu-\varepsilon ^{\prime },
\end{align*}%
for all $u$ such that $\Vert u\Vert _{2}\leq K_{3}$ and $n\geq n_{0}.$ On
the event, $\Vert H_{n}\Vert _{2}\leq K_{1},$ we have that both $%
V_{n}^{ERM}(u)$ and $V_{n}^{DRO}(u)$ are bounded from below by,
\begin{equation*}
V_{l}(u)=-K_{2}\Vert u\Vert _{2}+\frac{\lambda _{\min }(C)}{2}\Vert u\Vert
_{2}^{2}-\varepsilon ^{\prime },
\end{equation*}%
when $n\geq n_{0}.$ Since $V_{l}(u)>0$ for all $u$ such that $\Vert u\Vert
_{2}\geq K_{3}/2,$ we have that
\begin{equation*}
\sup_{n\geq n_{0}}{\rm pr}\left\{ V_{n}^{ERM}(u)>0\right\} \geq
1-\varepsilon \quad \text{ and }\quad \sup_{n\geq n_{0}}{\rm pr}\left\{
V_{n}^{DRO}(u)>0\right\}\geq 1-\varepsilon ,
\end{equation*}%
for all $u$ such that $\Vert u\Vert _{2}\in \lbrack K_{3}/2,K_{3}].$ Define $%
U=\{u\in \mathbb{R}^{d}:\Vert u\Vert _{2}\leq K_{3}/2\}.$ As $%
\min_{u}V_{n}^{ERM}(u)\leq 0$ and $\min_{u}V_{n}^{DRO}(u)\leq 0,$ it follows
from the convexity of $V_{n}^{ERM}(\cdot )$ and $V_{n}^{DRO}(\cdot )$ that,
\begin{align*}
\sup_{n\geq n_{0}}{\rm pr}\left\{ \arg \min \xspace_{u}V_{n}^{ERM}(u)%
\subseteq U\right\} & \geq 1-\varepsilon \quad \text{ } \\
\text{and }\quad \sup_{n\geq n_{0}}{\rm pr}\big\{\arg \min \xspace%
_{u}V_{n}^{DRO}(u)& \subseteq U\big\}\geq 1-\varepsilon ,
\end{align*}%
thus verifying the claim.
\end{proof}

\begin{proof}[Proof  of Proposition \ref{prop-levset-eq1}]
Due to Theorem \ref{thm:master}, we have that $G_{n}(\cdot) \Rightarrow
G(\cdot),$ uniformly in compact sets. Then it follows from Skorokhod
representation theorem that there exists a probability space where the
convergence,
\begin{align}
\sup_{\Vert u \Vert\leq K} \left\vert G_{n}(u) - G(u) \right\vert
\rightarrow0,  \label{as-conv}
\end{align}
happen almost surely, for every $K \in(0,\infty).$ Since $G(\cdot)$ is
continuous, we have from Theorem 7.14 of \citet{rockafellar2009variational}
that $G_{n}$ converges continuously to $G,$ almost surely. A simple
consequence of this observation; see \citet[Theorem 7.11]%
{rockafellar2009variational}, is that the epigraphs of $G_{n}$ converges to
the epigraph of $G$ (alternatively, $G_{n}$ epiconverges to $G$) almost
surely. Observe that $G(\cdot)$ is convex and $\inf_{u} G(u) = 0;$ this is
because $\varphi^{\ast}(0) = 0.$ Moreover, since $\alpha_{n} =
n\delta_{n} = \eta\in(0,+\infty)$, we have that
\begin{align*}
\text{lev}(G_{n},\alpha_{n}) = \text{lev}(G_{n},\eta) \rightarrow \text{lev}%
(G,\eta),
\end{align*}
almost surely, in the Painelev\'{e}-Kuratowski sense; see, for example, \citet%
[Theorem 5.1]{doi:10.1080/02331938908843480}, \citet[Theorem 7.1]%
{10.2307/1994611}, or \citet{10.2307/2159443}. Then, by \citet[Proposition 4.4]{rockafellar2009variational}, we have $\mathrm{cl}\{\text{lev}%
(G_{n},\alpha_{n})\}\rightarrow\text{lev}
(G,\eta)$.
Consequently, $\mathrm{cl}\{\text{lev}%
(G_{n},\alpha_{n})\} \Rightarrow \text{lev}(G,\eta).$
\end{proof}

\begin{proof}[Proof  of Proposition \ref{prop-levset-gthan1}]
Following the same reasoning used in the proof of Proposition \ref%
{prop-levset-eq1} to arrive at \eqref{as-conv}, we have a probability space
where the convergence,
\begin{align}
\sup_{\Vert u \Vert\leq K} \left\vert G_{n}(u) - G(u) \right\vert
\rightarrow0 \quad\text{ and } \quad n^{1/2}\left( \beta_{n}^{ERM} -
\beta_{\ast}\right) \rightarrow C^{-1}H,  \label{as-conv-2}
\end{align}
happen almost surely, for every $K \in(0,\infty);$ here, the latter
convergence follows from \eqref{inter-mas-thm-pf-1}.

Next, observe that $\alpha_{n} = n\delta_{n} \rightarrow0.$ Then, we have
from \citet[Proposition 7.7a]{rockafellar2009variational} that
\begin{align}
\text{Ls}_{n \rightarrow \infty} \ \text{lev}(G_{n},\alpha_{n}) \subseteq%
\text{lev}(G,0) = \left\{ u: \varphi^{\ast}(H-Cu) = 0\right\} = \{C^{-1}H\},
\label{limsup-setsgthan1-inter}
\end{align}
where the latter equality follows from the strict convexity of $%
\varphi^{\ast }(\cdot)$ and the positive definiteness of $C$ in Assumption
A2.b.

Furthermore, since $R_{n}(\beta_{n}^{ERM}) = 0,$ we have,
\begin{equation*}
n^{1/2}(\beta_{n}^{ERM}-\beta_{\ast}) \in\text{lev}\left\{nR_{n}(\beta_{\ast}+
n^{-1/2}\times\cdot\ ), \ \alpha_{n}\right\}= \text{lev}(G_{n},\alpha_{n}),
\end{equation*}
for every $n.$ Therefore, from the second convergence in \eqref{as-conv-2},
we obtain
\begin{equation*}
C^{-1}H \in \text{Li}_{n \rightarrow \infty} \ \text{lev}(G_{n},\alpha_{n}).
\end{equation*}
Combining this observation with that in \eqref{limsup-setsgthan1-inter} and \citet[Proposition 4.4]{rockafellar2009variational}, we
obtain that PK-$\lim_{n}$ $\mathrm{cl}\{\text{lev}%
(G_{n},\alpha_{n})\} = \{C^{-1}H\}$ almost
surely. As a result, $\mathrm{cl}\{\text{lev}%
(G_{n},\alpha_{n})\} \Rightarrow\{C^{-1}H\}.$
\end{proof}

\begin{proof}[Proof  of Proposition \ref{prop-levset-lthan1}]
Following the same reasoning in the proof of Proposition \ref%
{prop-levset-eq1}, we have a probability space where the convergence in %
\eqref{as-conv} happen almost surely, for every $K \in(0,+\infty).$ Consider
any fixed $u \in\mathbb{R}^{d}.$ Since $G_{n}(u) \rightarrow G(u) < \infty$
almost surely and $\alpha_{n} = n\delta_{n} \rightarrow\infty,$ there exists
a random variable $N_{u},$ defined on the same probability space, such that,
$G_{n}(u) < \alpha_{n},$ with probability 1, for all $n \geq N_{u}.$ As a
result, we have $u \in$ lev$(G_{n},\alpha _{n}),$ for all but finitely many $%
n,$ with probability 1. Then it follows from the definition of inner limit ($%
\text{Li}_n$) of sets that $u \in\liminf_{n} \text{lev}(G_{n},\alpha_{n}).$
Since the choice of $u \in\mathbb{R}^{d}$ is arbitrary, we have that
\begin{equation*}
\mathbb{R}^{d} \subseteq \text{Li}_{n \rightarrow \infty}\ \text{lev}%
(G_{n},\alpha_{n}).
\end{equation*}
As $\limsup_{n} \text{lev}(G_{n},\alpha_{n})$ is essentially a subset of $%
\mathbb{R}^{d},$ it follows that PK-$\lim_{n} \mathrm{cl}\{\text{lev}%
(G_{n},\alpha_{n})\} =
\mathbb{R}^{d},$ almost surely. Consequently, $\mathrm{cl}\{\text{lev}%
(G_{n},\alpha_{n})\}
\Rightarrow\text{lev}(G,\alpha).$
\end{proof}

\section{Proofs of Theorem \protect\ref{minmax_prop} and Corollary
\protect\ref{Cor_Sensitivity}}
\label{ssec:proof_prop_corr_1}
\begin{proof}[Proof  of Theorem \ref{minmax_prop}]
We define
\begin{align*}
&\mathcal{K}_{N}=\Omega \cap \left\{ x\in \mathbb{R}%
^{m}:\left\Vert x\right\Vert _{2}\leq N\right\} ,\ \mathcal{U}_{\delta
}^{N}\left( P_{n}\right) =\left\{ P\in \mathcal{P}(\mathcal{K}_{N}):W(P,P_{n})\leq \delta^{1/2}\right\},
\\
&g_{N}(\beta )=\sup_{P\in \mathcal{U}_{\delta }^{N}\left( P_{n}\right)
}E_{P}\left\{ \ell \left( X,\beta \right) \right\} \text{, and } g(\beta )=\sup_{P\in \mathcal{U}_{\delta }\left( P_{n}\right)
}E_{P}\left\{ \ell \left( X,\beta \right) \right\} .
\end{align*}
As slight abuse of the notation, we define $\ell(x,\beta)=+\infty$ for $\beta \notin B$, and thus $g(\beta)=g_N(\beta)=+\infty$ for $\beta \notin B$. Since $B$ is closed, we have $g(\beta)$ and $g_N(\beta)$ are lower semi-continuous.

 Now, we divide the proof of equation \eqref{minmax_eqn}
into three steps.

{Step 1: }we show that $g_{N}(\beta )\rightarrow g(\beta )$
pointwisely as $N\rightarrow \infty .$

{Step 2: }we show that the sequence $\left\{ g_{N}(\beta )\right\} $
epi-converges to $g(\beta ) $; see, for example, \citet[Definition 7.1]{rockafellar2009variational}.

{Step 3: }we finally show that $\inf_{\beta \in B}\lim_{N\rightarrow \infty }g_{N}(\beta )=\lim_{N\rightarrow \infty
}\inf_{\beta \in B}g_{N}(\beta ).$

By Fatou's lemma and the non-negativity of $\ell(\cdot)$, we have $E_{P}\left\{ \ell \left( X,\cdot \right) \right\}$ is lower semi-continuous on $B$ for any $P \in\mathcal{U}_{\delta }^{N}$. By the weak convergence of probability measure, we have $E_{\cdot}\left\{ \ell \left( X,\beta \right) \right\}$ is continuous in weak topology on $\mathcal{U}_{\delta }^{N}(P_n)$ for any $ \beta \in B$.  By Sion's minimax theorem
\citep{sion1958general} and the compactness of $\mathcal{U}_{\delta }^{N}\left( P_{n}\right) $
in weak topology, we have
\begin{equation*}
\lim_{N\rightarrow \infty }\inf_{\beta \in B}g_{N}(\beta
)=\sup_{N\geq 1}\inf_{\beta \in B}g_{N}(\beta )=\sup_{N\geq
1}\inf_{\beta \in B}\sup_{P\in \mathcal{U}_{\delta }^{N}\left(
P_{n}\right) }E_{P}\left\{ \ell \left( X,\beta \right) \right\} =\sup_{N\geq
1}\sup_{P\in \mathcal{U}_{\delta }^{N}\left( P_{n}\right) }\inf_{\beta \in
B}E_{P}\left\{ \ell \left( X,\beta \right) \right\} .
\end{equation*}%
By the weak duality, we have
\begin{eqnarray*}
\inf_{\beta \in B}\lim_{N\rightarrow \infty }g_{N}(\beta )&=& \inf_{\beta \in B}\sup_{P\in \mathcal{U}_{\delta }\left(
P_{n}\right) }E_{P}\left\{ \ell \left( X,\beta \right) \right\}\\
  &\geq
&\sup_{P\in \mathcal{U}_{\delta }\left( P_{n}\right) }\inf_{\beta \in
B}E_{P}\left\{ \ell \left( X,\beta \right) \right\}  \\
&\geq &\sup_{N\geq 1}\sup_{P\in \mathcal{U}_{\delta }^{N}\left( P_{n}\right)
}\inf_{\beta \in B}E_{P}\left\{ \ell \left( X,\beta \right)
\right\}  \\
&=&\lim_{N\rightarrow \infty }\inf_{\beta \in B}g_{N}(\beta ).
\end{eqnarray*}%
Therefore, all the inequalities above should be equalities, which completes
the proof.

Now, we execute proofs of {Steps 1 - 3}.

{Proof of Step 1: }Since $g_{N}(\beta )$ is increasing, we
have $g_{N}(\beta )\ $converges. Assume $\lim_{N\rightarrow \infty
}g_{N}(\beta )=g^{\ast }(\beta ),$ where $g^{\ast }(\beta )$ could be $%
+\infty .$ We have $g^{\ast }\left( \beta \right) \leq g(\beta ).$ We use
proof by contradiction. If $g^{\ast }\left( \beta \right) <g(\beta ),$ we
consider two cases: $g(\beta )<+\infty $ and $g(\beta )=+\infty .$

Case 1: $g(\beta)<\infty$. Let $\epsilon =g(\beta )-g^{\ast }\left( \beta
\right)>0 .$ Let $P^{\prime }\in \mathcal{U}_{\delta }\left( P_{n}\right) $
such as $E_{P^{\prime }}\left\{ \ell \left( X,\beta \right) \right\}
>g(\beta )-\epsilon /2.$ There exists $N$ sufficiently large such that
\begin{equation*}
E_{P^{\prime }}\left\{ \ell \left( X,\beta \right) \mathbb{I}(\left\Vert
X\right\Vert _{2}>N)\right\} <\epsilon /2.
\end{equation*}%
Then, we construct a measure $P_{N}^{\prime }\in \mathcal{U}_{\delta
}^{N}\left( P_{n}\right) $ that for any Borel set $A\subset \mathcal{K}_{N},$
\begin{equation*}
P_{N}^{\prime }(A)=P^{\prime }(A)+\left\{ 1-P^{\prime }(\mathcal{K}%
_{N})\right\} P_{n}\left( A\right) .
\end{equation*}%
Therefore, we have
\begin{equation*}
g^{\ast }(\beta )\geq g_{N}(\beta )\geq E_{P_{N}^{\prime }}\left\{ \ell
\left( X,\beta \right) \right\} >E_{P^{\prime }}\left\{ \ell \left( X,\beta
\right) \right\} -\epsilon /2>g(\beta )-\epsilon ,
\end{equation*}%
which leads a contradiction.

Case 2: $g(\beta)=+\infty$ and $g^*(\beta)<+\infty$.  Let $P^{\prime }\in \mathcal{U}_{\delta }\left( P_{n}\right) $
such as $E_{P^{\prime }}\left\{ \ell \left( X,\beta \right) \right\}
>g^*(\beta )+1.$ There exists $N$ sufficiently large such that
\begin{equation*}
E_{P^{\prime }}\left\{ \ell \left( X,\beta \right) \mathbb{I}(\left\Vert
X\right\Vert _{2}>N)\right\} <1.
\end{equation*}%
Then, we construct a measure $P_{N}^{\prime }\in \mathcal{U}_{\delta
}^{N}\left( P_{n}\right) $ that for any Borel set $A\subset \mathcal{K}_{N},$
\begin{equation*}
P_{N}^{\prime }(A)=P^{\prime }(A)+\left\{ 1-P^{\prime }(\mathcal{K}%
_{N})\right\} P_{n}\left( A\right) .
\end{equation*}%
Therefore, we have
\begin{equation*}
g^{\ast }(\beta )\geq g_{N}(\beta )\geq E_{P_{N}^{\prime }}\left\{ \ell
\left( X,\beta \right) \right\} >E_{P^{\prime }}\left\{ \ell \left( X,\beta
\right) \right\} -\epsilon /2>g^*(\beta) ,
\end{equation*}%
which leads a contradiction.


{Proof of Step 2: } By \citet[Proposition 7.2]{rockafellar2009variational}, we need to check two conditions:

(i) For every $\beta \in \mathbb{R}^d$ and for every sequence $\left\{
\beta _{N}\right\} _{N=1}^{\infty }$ converging to $\beta $, we claim $\liminf_{N\rightarrow \infty }g_{N}(\beta _{N})\geq g(\beta ).$ Recalling $%
g_{N}(\beta _{N})\geq g_{M}(\beta _{N})$ by monotonicity for $N>M$ and the
lower semi-continuity of $g_{M}(\beta ),$ we have
\begin{equation*}
\liminf_{N\rightarrow \infty }g_{N}(\beta _{N})\geq \liminf_{N\rightarrow \infty }g_{M}(\beta _{N})\geq g_{M}(\beta ).
\end{equation*}%
By taking $M$ to the infinity, we have $\liminf_{N\rightarrow \infty
}g_{N}(\beta _{N})\geq g(\beta ).$

(ii) for every $\beta \in \mathbb{R}^d,$ we pick a sequence $\beta
_{N}=\beta ,$ then
\begin{equation*}
\lim_{N\rightarrow \infty }g_{N}(\beta _{N})=\lim_{N\rightarrow \infty
}g_{N}(\beta )=g(\beta ).
\end{equation*}%

{Proof of Step 3: } We claim $g(\cdot)$ is level-bounded. Since $E_{P_{\ast }}\left\{ \ell \left(
X,\beta \right) \right\} $ has a
unique minimizer, its level set $\left\{\beta \in \mathbb{R}^d : E%
_{P_{\ast }}\left\{ \ell \left( X,\beta \right) \right\} \leq b\right\} $ is
bounded. For every $P_{n}$ and $\delta ,$ there always exists $%
\epsilon \in (0,1),$ such that
\begin{equation*}
(1-\epsilon )P_{n}+\epsilon P_{\ast }\in\mathcal{U}_\delta(P_n).
\end{equation*}
Then, we have $g(\beta )\geq \epsilon
E_{P_{\ast }}\left\{ \ell \left( X,\beta \right) \right\}$.
Therefore, the level set $\left\{\beta\in \mathbb{R}^d: g(\beta )\leq
b\right\} \subset \left\{\beta\in \mathbb{R}^d:E_{P_{\ast }}\left\{
\ell \left( X,\beta \right) \right\} \leq b/\epsilon \right\}$  is bounded.

By \citet[Exercise 7.32(c)]{rockafellar2009variational}, we have the sequence $\{g_N(\cdot)\}$ is eventually level-bounded.  Further, since  $g_N,g$ are lower semi-continuous and proper, by \citet[Theorem 7.33]{rockafellar2009variational}, we have the desired result.

Then, we proceed with the claim that there exists $\beta_{n}^{DRO}(\delta) \in \Lambda^+_{\delta}(P_{n})$.  First, $\mathcal{U}_{\delta
}^{N}\left( P_{n}\right)$ is compact and $\inf_{\beta \in B}g_{N}(\beta)$ is upper-semicontinuous on $P$, and thus there exists $P_N$ such that
\[
P_N\in\mathop{\arg\max}_{P \in \mathcal{U}_{\delta}^{N}\left( P_{n}\right)}\inf_{\beta \in B}g_{N}(\beta).
\]
Further, for any $\beta_N \in \arg\min_{\beta \in B}g_N(\beta)$, when $N$ is sufficiently large, whose existence is guaranteed by \citet[Theorem 7.33]{rockafellar2009variational}.
Then, we have
\[\inf_{\beta \in B}\sup_{P\in \mathcal{U}_{\delta }^{N}\left(
P_{n}\right) }E_{P}\left\{ \ell \left( X,\beta \right) \right\} \geq E_{P_N} \left\{\ell(X,\beta_N) \right\} \geq \sup_{P\in \mathcal{U}_{\delta }^{N}\left( P_{n}\right) }\inf_{\beta \in
B}E_{P}\left\{ \ell \left( X,\beta \right) \right\} .
\]
By Sion's minimax theorem, we have all the inequalities above are equalities. Therefore, $\beta_N \in
\arg \min_{\beta }E_{P_N}\left\{ \ell (X;\beta )\right\}$ and thus $\beta_N \in \Lambda _{\delta}(P_{n})$. Finally, since the sequence $\{\beta_N\}_{N=1}^{\infty}$ is bounded and all its cluster points belong to $\arg \min g(\beta)$ by \citet[Theorem 7.33]{rockafellar2009variational}, combining with the closedness of $\Lambda^+_\delta(P_n)$, we have the desired result.
\end{proof}
\begin{proof}[Proof  of Corollary \ref{Cor_Sensitivity}]
Define $\bar{\Psi}_n(\beta)= E_{P_n}[\ell(X;\beta)] + {\eta}^{1/2}n^{-1/2}
\{E_{P_n}\Vert D_x\ell(X;\beta)\Vert^2\}^{1/2}$ and
\begin{align*}
\bar{V}_n(u) = n^{1/2}\left\{ \bar{\Psi}_n\left(\beta_\ast +
n^{-1/2}u\right) - \bar{\Psi}_n(\beta_\ast)\right\}.
\end{align*}
Following the lines of the proof of Proposition \ref{prop:DRO}, we have $%
\bar{V}_n(u) = V_n^{DRO}(u) + o_p(n^{-1/2}).$ Consequently, since the
collection $\{V_n^{DRO}(\cdot)\}_{n \geq 1}$ is tight and strongly convex
(see the proof of Proposition \ref{prop:argmintightness}), we have that the
sequences $\{\bar{V}_n(\cdot)\}_{n \geq 1}$ and $\{\arg%
\min \xspace_u \bar{V}_n(u): n \geq 1\}$ are tight. Then, as a
consequence of Theorem \ref{thm:master}, we have that $\bar{V}_n(\cdot)
\Rightarrow V\{-f_{\eta,1}(H),\cdot\,\},$ uniformly in compact sets. Since the
functions $\bar{V}_n(\cdot)$ and $V\{-f_{\eta,1}(H),\cdot\,\}$ are minimized,
respectively, at $n^{1/2}\left(\bar{\beta}_n^{DRO} - \beta_\ast\right)$ and $%
C^{-1}f_{\eta,1}(H),$ we have that
\begin{align*}
n^{1/2}\left(\bar{\beta}_n^{DRO} - \beta_\ast \right) \Rightarrow
C^{-1}f_{\eta,1}(H),
\end{align*}
as $n \rightarrow \infty.$ Then the conclusion that
\begin{align*}
n^{1/2}\left(\bar{\beta}_n^{DRO} - \beta_n^{DRO}\right) \rightarrow 0,
\end{align*}
in probability, follows automatically from the convergence $n^{1/2}\left({%
\beta}_n^{DRO} - \beta_\ast\right) \Rightarrow C^{-1}f_{\eta,1}(H)$; see
Theorem \ref{thm:levelsets-master} as a consequence of the continuous
mapping theorem. This verifies the statement of Corollary \ref%
{Cor_Sensitivity}. 
\end{proof}

\section{Proofs of Proposition \protect\ref{Prop_Reg_Plug} and Proposition \ref{prop:level-sets-char}}
\label{ssec:proof_prop_reg}
\begin{proof}[Proof  of Proposition \protect\ref{Prop_Reg_Plug}]
i).
Since $E\left[ D_{x}h(X,\beta_{\ast})D_{x}h(X,\beta_{\ast})^{\T} %
\right] \succ0,$ we have $\varphi(\xi)\geq c\left\Vert
\xi\right\Vert _{2}^{2}$ for some numerical constant $c>0$ and thus $%
\varphi^{\ast}\left( \cdot\right) $ is continuous$.$

ii).
Since $\beta _{n}^{ERM}{\rightarrow }\beta _{\ast }$ almost surely, we have $%
\varphi _{n}\left( \xi \right) \overset{a.s.}{\rightarrow }$ $\varphi \left(
\xi \right) $ for any $\xi .$ Then, since $\left\Vert \xi ^{\T}D_{x}h(X,\beta
_{\ast })\right\Vert _{p}^{2}$ is Lipschitz in $\xi $ for $\left\Vert \xi
\right\Vert _{p}\leq b$, i.e., for $\left\Vert \xi _{1}\right\Vert _{p}\leq
b,\left\Vert \xi _{2}\right\Vert _{p}\leq b$,
\begin{equation*}
\left\vert \left\Vert \left\{ D_{x}h(X,\beta _{\ast })\right\} ^{\T}\xi
_{1}\right\Vert _{p}^{2}-\left\Vert \left\{ D_{x}h(X,\beta _{\ast })\right\}
^{\T}\xi _{2}\right\Vert _{p}^{2}\right\vert \leq 2b\left\Vert D_{x}h(X,\beta
_{\ast })\right\Vert _{q}^{2}\left\Vert \xi _{1}-\xi _{2}\right\Vert _{p}
\end{equation*}%
and $E\left[ \left\Vert D_{x}h(X,\beta _{\ast })\right\Vert _{q}^{2}%
\right] <\infty ,$ we have the uniform law of large numbers that
\begin{equation*}
\sup_{\left\Vert \xi \right\Vert _{p}\leq b}\left\vert \varphi _{n}\left(
\xi \right) -\varphi \left( \xi \right) \right\vert {%
\rightarrow }0
\end{equation*}%
almost surely, uniformly over $\left\Vert \xi \right\Vert _{p}\leq b.$ Then, we have
\begin{equation*}
\sup_{\left\Vert \xi \right\Vert _{p}\leq b}\left\{ \xi ^{\T}\zeta
-\varphi _{n}\left( \xi \right) \right\}{\rightarrow }%
\sup_{\left\Vert \xi \right\Vert _{p}\leq b}\left\{ \xi ^{\T}\zeta
-\varphi \left( \xi \right) \right\}
\end{equation*}%
almost surely, uniformly over $\zeta $ in compact sets as $n\rightarrow \infty .$ Finally,
since $b$ is chosen arbitrarily, we conclude $\varphi _{n}^{\ast }\left(
\cdot \right) \overset{p}{\rightarrow }$ $\varphi ^{\ast }\left( \cdot
\right) $ as $n\rightarrow \infty $ uniformly on compact sets.

iii).
Observe that $\varphi_{n}^{\ast}\left( \bar{\Xi}_{n}Z\right) =\varphi
_{n}^{\ast}\left( \bar{\Xi}_{n}Z\right) -\varphi^{\ast}\left( \bar{\Xi}%
_{n}Z\right) +\varphi^{\ast}\left( \bar{\Xi}_{n}Z\right) .$ The continuous
mapping theorem and $\bar{\Xi}_{n}Z\Rightarrow H$ give us $\varphi^{\ast
}\left( \bar{\Xi}_{n}Z\right) \Rightarrow\varphi^{\ast}\left( H\right) .$
And ii) gives us $\varphi_{n}^{\ast}\left( \bar{\Xi}_{n}Z\right)
-\varphi^{\ast}\left( \bar{\Xi}_{n}Z\right) \overset{p}{\rightarrow}0.$
\end{proof}
\begin{proof}[Proof  of Proposition \ref{prop:level-sets-char}]
For any convex function $f(\cdot)$ with $\inf f<0,$ it is well-known \citep[Exercise 11.6]{rockafellar2009variational} that the
support function of the level set $A=\{u:f(u)\leq0\}$ is $%
h_{A}(v)=\inf_{\lambda>0}\lambda f^{\ast }(\lambda^{-1}v),$ where $f^{\ast}$
is the convex conjugate of $f.$ Since the convex conjugate of $%
\varphi^{\ast}(C\times\cdot\ )-\eta$ is $\varphi(C^{-1}\times\cdot\
)+\eta,$ the support function of $\Lambda_{\eta
}=\{u:\varphi^{\ast}(Cu)-\eta\leq0\}$ is
\begin{equation*}
  h_{\Lambda_{\eta}}(v)=\inf_{\lambda>0}\lambda\left\{ \varphi\left( {
        \lambda^{-1}C^{-1}v}{}\right) +\eta\right\} =\inf_{\lambda>0}\left\{ \lambda^{-1}
    \varphi\left( C^{-1}v\right) +\lambda\eta\right\} =2
  \{\eta\varphi(C^{-1}v)\}^{1/2}.
\end{equation*}
This completes the proof of Proposition \ref{prop:level-sets-char}.
\end{proof}
\section{Proofs of technical results}
\label{appendix:proof_technical_lemma}

\begin{proof}[Proof  of Lemma \protect\ref{Lem-Delta-ineq}]
A proof of the conclusion in Part a) of Lemma \ref{Lem-Delta-ineq} can be
found in Appendix A of \citet{shalev2007online}. For the proof of Part b), we
proceed as follows.

For brevity, let $D$ denote the derivative of the function
$\Vert \Delta \Vert_q^2$ evaluated at $\Delta = \Delta_\ast$ and
$H(\Delta)$ denote the hessian matrix
of function $\frac{1}{2}\left\Vert \Delta\right\Vert _{q}^{2}.$ Then,
for any $x\in R^{d},$
\begin{align*}
& x^{T}H(\Delta)x \\
& =\frac{1}{q}\left( \frac{2}{q}-1\right) \left( \sum_{i=1}^{d}\left\vert
\Delta_{i}\right\vert ^{q}\right) ^{2/q-2}\left\{
q\sum_{i=1}^{d}\text{sgn}(\Delta_{i})|\Delta_{i}|^{q-1}x_{i}\right\} ^{2} \\
& +(q-1)\left( \sum_{i=1}^{d}\left\vert \Delta_{i}\right\vert ^{q}\right)
^{2/q-1}\sum_{i=1}^{d}|\Delta_{i}|^{q-2}x_{i}^{2} \\
& =\left( \sum_{i=1}^{d}\left\vert \Delta_{i}\right\vert ^{q}\right)
^{2/q-2} \left[ \left( 2-q\right) \left\{
\sum_{i}\text{sgn}(\Delta_{i})|\Delta_{i}|^{q-1}x_{i}\right\} ^{2}+(q-1)\left(
\sum_{i=1}^{d}\left\vert \Delta_{i}\right\vert ^{q}\right) \left(
\sum_{i}|\Delta_{i}|^{q-2}x_{i}^{2}\right) \right] \\
& =\left( \sum_{i=1}^{d}\left\vert \Delta_{i}\right\vert ^{q}\right)
^{2/q-2} \left[ \left\{
\sum_{i=1}^{d}\text{sgn}(\Delta_{i})|\Delta_{i}|^{q-1}x_{i}\right\}
^{2}+(q-1)\sum_{i=1}^{d}\sum_{j=i+1}^{d}|\Delta_{i}|^{q-2}\left\vert
\Delta_{j}\right\vert ^{q-2}\left( \Delta_{i}x_{j}-\Delta_{j}x_{i}\right)
^{2}\right] .
 \end{align*}
Since $q-1>1,$ we obtain that,
\begin{equation*}
x^{T}H(\Delta)x\geq\left( \sum_{i=1}^{d}\left\vert \Delta_{i}\right\vert
^{q}\right) ^{2/q-2}\left[ \left\{ \sum_{i=1}^{d}\text{sgn}(\Delta_{i})|\Delta
_{i}|^{q-1}x_{i}\right\}
^{2}+\sum_{i=1}^{d}\sum_{j=i+1}^{d}|\Delta_{i}|^{q-2}\left\vert
\Delta_{j}\right\vert ^{q-2}\left( \Delta_{i}x_{j}-\Delta_{j}x_{i}\right)
^{2}\right] .
\end{equation*}
Considering only non-zero entries among $\{\Delta_{i}:i=1,\ldots,d\},$ we
re-express the right hand side as,
\begin{align*}
& \left( \sum_{i=1}^{d}\left\vert \Delta_{i}\right\vert ^{q}\right) ^{2/q-2}
\left\{ \left( \sum_{\substack{ i=1  \\ \Delta_{i}\neq0}}^{d}|\Delta _{i}|^{q}%
\frac{x_{i}}{\Delta_{i}}\right) ^{2}+\sum_{\substack{ i=1  \\ \Delta
_{i}\neq0 }}^{d}\sum_{\substack{ j=i+1  \\ \Delta_{j}\neq0}}%
^{d}|\Delta_{i}|^{q}\left\vert \Delta_{j}\right\vert ^{q}\left( \frac{x_{j}}{%
\Delta_{j}}-\frac{x_{i}}{\Delta_{i}}\right) ^{2}\right\} \\
& \quad\quad=\left( \sum_{i=1}^{d}\left\vert \Delta_{i}\right\vert
^{q}\right) ^{2/q-2}\left[ \sum_{\substack{ j=1  \\ \Delta_{j}\neq0}}%
^{d}|\Delta_{j}|^{q}\left\{ \sum_{\substack{ i=1  \\ \Delta_{i}\neq0}}%
^{d}|\Delta _{i}|^{q}\left( \frac{x_{i}}{\Delta_{i}}\right) ^{2}\right\} %
\right] \\
& \quad\quad=\left( \sum_{i=1}^{d}\left\vert \Delta_{i}\right\vert
^{q}\right)^{2/q-1}\left\{\sum_{i=1}^{d}|\Delta_{i}|^{q-2}\left(
x_{i}\right) ^{2}\right\} \\
& \quad\quad=\frac{\sum_{i=1}^{d}|\Delta_{i}|^{q-2}\left( x_{i}\right) ^{2}}{%
\left\Vert \left\vert \Delta\right\vert ^{q-2}\right\Vert _{\frac {q}{q-2}}},
\end{align*}
where $\left\vert \Delta\right\vert ^{q-2}$ is defined as a vector $%
(|\Delta_{1}|^{q-2},|\Delta_{2}|^{q-2},\ldots,|\Delta_{d}|^{q-2})^{\T}.$
Then,
\begin{equation}
x^{T}H(\Delta)x\geq\frac{\sum_{i=1}^{d}|\Delta_{i}|^{q-2}\left( x_{i}\right)
^{2}}{\left\Vert \left\vert \Delta\right\vert ^{q-2}\right\Vert _{\frac {q}{%
q-2}}}\geq\frac{\sum_{i=1}^{d}|\Delta_{i}|^{q-2}\left( x_{i}\right) ^{2}}{%
\left\Vert \left\vert \Delta\right\vert ^{q-2}\right\Vert _{1}}=\frac{%
\sum_{i=1}^{d}|\Delta_{i}|^{q-2}\left( x_{i}\right) ^{2}}{\sum
_{i=1}^{d}|\Delta_{i}|^{q-2}}.  \label{inter-delta-ineq-1}
\end{equation}
We can regard the right hand side as the weighted average of $\left\{
x_{i}^{2}\right\} _{i=1}^{d}.$

Next, by applying Taylor's theorem, we have
\begin{equation*}
\left\Vert \Delta\right\Vert _{q}^{2}=\left\Vert \Delta^{\ast}\right\Vert
_{q}^{2}+\left( D\left\Vert \Delta^{\ast}\right\Vert _{q}^{2}\right)
^{\T}\xi+2\int_{0}^{1}(1-t)\xi^{\T}H(\Delta^{\ast}+t\xi)\xi {\rm d}t.
\end{equation*}
We focus on the last term, which is
\begin{equation*}
\int_{0}^{1}(1-t)\xi^{\T}H(\Delta^{\ast}+t\xi)\xi {\rm d}t\geq\int_{0}^{1}(1-t)
\frac{\sum_{i=1}^{d}|\Delta_{i}^{\ast}+t\xi_{i}|^{q-2}(\xi_{i})^{2}}{%
\sum_{i=1}^{d}|\Delta_{i}^{\ast}+t\xi_{i}|^{q-2}}{\rm d}t,
\end{equation*}
due to the inequality deduced earlier in \eqref{inter-delta-ineq-1}. As the
denominator of the right hand side in the above expression is bounded by,
\begin{equation*}
\sum_{i=1}^{d}|\Delta_{i}^{\ast}+t\xi_{i}|^{q-2}\leq\max(2^{q-3},1)\sum
_{i=1}^{d}\left( |\Delta_{i}^{\ast}|^{q-2}+|\xi_{i}|^{q-2}\right) ,
\end{equation*}
we obtain that,
\begin{align*}
\int_{0}^{1}(1-t)\frac{\sum_{i=1}^{d}|\Delta_{i}^{\ast}+t\xi_{i}|^{q-2}(%
\xi_{i})^{2}}{\sum_{i=1}^{d}|\Delta_{i}^{\ast}+t\xi_{i}|^{q-2}}{\rm d}t \geq \frac{%
1}{\max(2^{q-3},1)}\frac{\sum_{i=1}^{d}\left\{
\int_{0}^{1}(1-t)|\Delta_{i}^{\ast}+t\xi_{i}|^{q-2}{\rm d}t\right\}
(\xi_{i})^{2}}{\sum_{i=1}^{d}\left(
|\Delta_{i}^{\ast}|^{q-2}+|\xi_{i}|^{q-2}\right) }.
\end{align*}
Then, we only need to bound the integral
\begin{equation*}
\int_{0}^{1}(1-t)|\Delta_{i}^{\ast}+t\xi_{i}|^{q-2}{\rm d}t.
\end{equation*}
If $\Delta_{i}^{\ast}$ and $\xi_{i}$ have the same sign, then
\begin{equation*}
\int_{0}^{1}(1-t)|\Delta_{i}^{\ast}+t\xi_{i}|^{q-2}{\rm d}t\geq%
\int_{0}^{1}(1-t)t^{q-2}|\xi_{i}|^{q-2}{\rm d}t=\frac{1}{(q-1)q}|\xi_{i}|^{q-2}.
\end{equation*}
On the other hand, if $\Delta_{i}^{\ast}$ and $\xi_{i}$ have  different
signs, then we obtain that,
\begin{equation*}
\int_{0}^{1}(1-t)|\Delta_{i}^{\ast}+t\xi_{i}|^{q-2}{\rm d}t\geq|\xi_{i}|^{q-2}%
\left\{ \int_{0}^{a}(1-t)\left( a-t\right) ^{q-2}{\rm d}t+\int_{a}^{1}(1-t)\left(
t-a\right) ^{q-2}{\rm d}t\right\} ,
\end{equation*}
where $a=\min\left( \left\vert \frac{\Delta_{i}^{\ast}}{\xi_{i}} \right\vert
,1\right) .$ Computing the integrals in the right hand side of the above
inequality, we obtain,
\begin{equation*}
\left\{ \int_{0}^{a}(1-t)\left( a-t\right) ^{q-2}{\rm d}t+\int_{a}^{1}(1-t)\left(
t-a\right) ^{q-2}{\rm d}t\right\}=\frac{(1-a)^{q}+a^{q-1}(q-a)}{q(q-1)}
\end{equation*}
Since $2a<q,$ we have
\begin{equation*}
\frac{(1-a)^{q}+a^{q-1}(q-a)}{q(q-1)}\geq\frac{(1-a)^{q}+a^{q}}{q(q-1)}\geq%
\frac{1}{2^{q-1}q(q-1)}.
\end{equation*}
Then by combining the above observations, we obtain that,
\begin{equation*}
2\int_{0}^{1}(1-t)\xi^{\T}H(\Delta^{\ast}+t\xi)\xi {\rm d}t\geq C^{\prime}\frac {%
\sum_{i=1}^{d}|\xi_{i}|^{q}}{\sum_{i=1}^{d}|\Delta_{i}^{\ast}|^{q-2}+%
\sum_{i=1}^{d}|\xi_{i}|^{q-2}}.
\end{equation*}
where
\begin{equation*}
C^{\prime}=\frac{1}{2^{q-2}q(q-1)\max(2^{q-3},1)}.
\end{equation*}
Moreover, we have,
\begin{equation}
\frac{\sum_{i=1}^{d}|\xi_{i}|^{q}}{\sum_{i=1}^{d}|\Delta_{i}^{\ast}|^{q-2}+%
\sum_{i=1}^{d}|\xi_{i}|^{q-2}}\geq\frac{1}{2}\min\left( \frac {%
\sum_{i=1}^{d}|\xi_{i}|^{q}}{\sum_{i=1}^{d}|\Delta_{i}^{\ast}|^{q-2}},\frac{%
\sum_{i=1}^{d}|\xi_{i}|^{q}}{\sum_{i=1}^{d}|\xi_{i}|^{q-2}}\right) .
\end{equation}
Due to Chebyshev's sum inequality, we also obtain,
\begin{equation}
d\sum_{i=1}^{d}|\xi_{i}|^{q}\geq\left( \sum_{i=1}^{d}|\xi_{i}|^{q-2}\right)
\left( \sum_{i=1}^{d}|\xi_{i}|^{2}\right) .
\end{equation}
Letting $C=\frac{1}{2d}C^{\prime},$ the desired result follows.
\end{proof}

\color{black}
\begin{proof}[Proof  of Lemma \protect\ref{lemmaksai}]
For a fixed $u,\Delta ,$ recall the definitions of $I(X_{i},\Delta
,u)$, $I_{1}(X_{i},\Delta ,u)$, $I_{2}(X_{i},\Delta ,u)$, $\Delta _{i}^{\prime }$
from \eqref{defn-I-terms}-\eqref{delta_prime} and
\begin{equation}
M_{n}(\xi ,u)=\frac{1}{n}\sum_{i=1}^{n}\left( \xi ^{{\mathrm{%
\scriptscriptstyle T}}}D_{\beta }h\left( X_{i},\beta _{\ast }\right)
u+\max_{\Delta :X_{i}+n^{-1/2}\Delta \in \Omega }\left\{ \xi ^{{\mathrm{%
\scriptscriptstyle T}}}D_{x}h\left( X_{i},\beta _{\ast }\right) \Delta +\xi
^{{\mathrm{\scriptscriptstyle T}}}I(X_{i},\Delta ,u)-\Vert \Delta \Vert
_{q}^{2}\right\} \right) .  \label{inter-tightness-ksai}
\end{equation}%
By taking $\Delta =0$ and recalling Assumption A2.c, we have%
\begin{eqnarray*}
&&\max_{\Delta :X_{i}+n^{-1/2}\Delta \in \Omega }\left\{ \xi ^{{\mathrm{%
\scriptscriptstyle T}}}D_{x}h\left( X_{i},\beta _{\ast }\right) \Delta +\xi
^{{\mathrm{\scriptscriptstyle T}}}I(X_{i},\Delta ,u)-\Vert \Delta \Vert
_{q}^{2}\right\}  \\
&\geq &\xi ^{{\mathrm{\scriptscriptstyle T}}}I_{2}(X_{i},0,u) \\
&\geq &-\Vert \xi \Vert _{p}\int_{0}^{1}\left\Vert D_{\beta }h\left(
X_{i},\beta _{\ast }+t\frac{u}{n^{1/2}}\right) -D_{\beta }h\left(
X_{i},\beta _{\ast }\right) \right\Vert _{q}\left\Vert u\right\Vert _{q}\
\mathrm{d}t \\
&\geq &-\frac{1}{2}n^{-1/2}\Vert \xi \Vert _{p}\left\Vert u\right\Vert
_{q}^{2}\bar{\kappa}(X_{i})\mathrm{.}
\end{eqnarray*}%
Since $\left\Vert u\right\Vert _{2}\leq K,$ for any $\epsilon _{1}>0,$ there
exists $n_{1}>0$ such as for all $n>n_{1},$ $\frac{1}{2}n^{-1/2}\left\Vert
u\right\Vert _{q}^{2}<1.$

Then for any $c>0,$ plugging in $\Delta =c\Delta _{i}^{\prime },$ we have
that $\xi ^{{\mathrm{\scriptscriptstyle T}}}D_{x}h(X_{i},\beta _{\ast
})\Delta =c\Vert D_{x}h(X_{i},\beta _{\ast })^{{\mathrm{\scriptscriptstyle T}%
}}\xi \Vert _{p}\Vert \Delta _{i}^{\prime }\Vert _{q},$ and thus,
\begin{align*}
& \max_{\Delta :X_{i}+n^{-1/2}\Delta \in \Omega }\left\{ \xi ^{{\mathrm{%
\scriptscriptstyle T}}}D_{x}h\left( X_{i},\beta _{\ast }\right) \Delta +\xi
^{{\mathrm{\scriptscriptstyle T}}}I(X_{i},\Delta ,u)-\Vert \Delta \Vert
_{q}^{2}\right\}  \\
&  \geq \left\{ c\Vert D_{x}h\left( X_{i},\beta _{\ast }\right) ^{%
{\mathrm{\scriptscriptstyle T}}}\xi \Vert _{p}\Vert \Delta _{i}^{\prime
}\Vert _{q}-c^{2}\Vert \Delta _{i}^{\prime }\Vert _{q}^{2}+\xi ^{{\mathrm{%
\scriptscriptstyle T}}}I_{1}(X_{i},c\Delta _{i}^{\prime },u)\right\} \mathbb{I}\left(
X_{i}+cn^{-1/2}\Delta _{i}^{\prime }\in \Omega \right) -\bar{\kappa}%
(X_{i})\Vert \xi \Vert _{p}.
\end{align*}%
As a consequence of H\"{o}lder's inequality, $|\xi ^{{\mathrm{%
\scriptscriptstyle T}}}I_{1}(X_{i},c\Delta _{i}^{\prime },u)|$ is bounded
from above by
\begin{equation*}
c\Vert \xi \Vert _{p}\int_{0}^{1}\left\Vert \left\{ D_{x}h\left(
X_{i}+cn^{-1/2}\Delta _{i}^{\prime },\beta _{\ast }+t\frac{u}{n^{1/2}}%
\right) -D_{x}h(X_{i},\beta _{\ast })\right\} \Delta _{i}^{\prime
}\right\Vert _{q}{\rm d}t.
\end{equation*}%
Define the set $C_{0}=\{w\in \Omega :\left\Vert w\right\Vert _{p}\leq
c_{0}\},$ where $c_{0}$ will be chosen large momentarily. Then, due to the
continuity of $D_{x}h(\cdot )$ in Assumption A2.c, we have that
\begin{equation*}
\lim_{n\rightarrow \infty }|\xi ^{{\mathrm{\scriptscriptstyle T}}%
}I_{1}(X_{i},c\Delta _{i}^{\prime },u)|\mathbb{I}(X_{i}\in C_{0})=0,
\end{equation*}%
uniformly over all $i$ such that $X_{i}\in C_{0},$ $\xi $ in compact sets,
and $\Vert u\Vert _{2}\leq K.$ Therefore, for given positive constants $%
\varepsilon ^{\prime },c$ there exists $n_{2}$ such that for all $n\geq
n_{2},$
\begin{equation*}
\sup_{i}|\xi ^{{\mathrm{\scriptscriptstyle T}}}I_{1}(X_{i},c\Delta
_{i}^{\prime },u)|\mathbb{I}(X_{i}\in C_{0})\leq c\varepsilon ^{\prime
}\Vert \xi \Vert _{p}.
\end{equation*}%
Further, notice that $\left\Vert \Delta _{i}^{\prime }\right\Vert $ is
bounded when $X_{i}\in C_{0}$ due to the compactness of $C_{0}$ and the
continuity of $D_{x}\left\{ h(X_{i},\beta _{\ast }\right) \},$ Let $M_{\Delta
}=\sup_{x\in C_{0}}\left\Vert \Delta _{i}^{\prime }\left( x\right)
\right\Vert _{2}.$ As a result, we obtain from \eqref{inter-tightness-ksai}
that, $M_{n}(\xi ,u)$ is bounded from below by,
\begin{eqnarray}
\frac{1}{n} &&\xi ^{{\mathrm{\scriptscriptstyle T}}}\sum_{i=1}^{n}D_{\beta
}h(X_{i},\beta _{\ast })u-\frac{1}{n}\sum_{i=1}^{n}\bar{\kappa} (X_{i})\Vert \xi
\Vert _{p}+  \label{inter-tightness-lb} \\
&&\frac{1}{n}\sum_{i=1}^{n}\left\{ c\Vert D_{x}h\left( X_{i},\beta _{\ast
}\right) ^{{\mathrm{\scriptscriptstyle T}}}\xi \Vert _{p}\Vert \Delta
_{i}^{\prime }\Vert _{q}-c^{2}\Vert \Delta _{i}^{\prime }\Vert
_{q}^{2}-c\varepsilon ^{\prime }\Vert \xi \Vert _{p}\right\} \mathbb{I}%
(X_{i}\in C_{0}^{cn^{-1/2}M_{\Delta }}).  \notag
\end{eqnarray}%
As in the proof of Lemma 2 in \citet{blanchet2016robust} and $E%
\left[ D_{x}h(X,\beta _{\ast })D_{x}h(X,\beta _{\ast })^{{\mathrm{%
\scriptscriptstyle T}}}\right] \succ 0$ in Assumption A2.b,
there exists $\epsilon _{0}>0$, $\delta >0,$ and $c_{0}$ sufficiently large
such that for all $n\geq N^{\prime }(\delta ),$
\begin{equation*}
\frac{1}{n}\sum_{i=1}^{n}\Vert D_{x}h\left( X_{i},\beta _{\ast }\right) ^{{%
\mathrm{\scriptscriptstyle T}}}\xi \Vert _{p}\Vert \Delta _{i}^{\prime
}\Vert _{q}\mathbb{I}(X_{i}\in C_{0}^{\epsilon _{0}})>\frac{\delta }{2}%
\left\Vert \xi \right\Vert _{p}.
\end{equation*}%
Further, let $c_{1}=\sup_{x\in C_{0}}\Vert \Delta _{i}^{\prime }\left(
x\right) \Vert _{q}^{2}<\infty .$ By following the proof of Lemma 2 in %
\citet{blanchet2016robust}, if $n\geq \max \left\{ N^{\prime }(\delta
),n_{1},n_{2,}\left( M_{\Delta }c/\epsilon _{0}\right) ^{2}\right\} ,$ we
have
\begin{equation*}
\sup_{\left\Vert \xi \right\Vert _{p}>b}\left\{ \xi ^{{\mathrm{%
\scriptscriptstyle T}}}H_{n}-M_{n}(\xi ,u)\right\} \leq \sup_{\left\Vert \xi
\right\Vert _{p}>b}\left\Vert \xi \right\Vert _{p}\left[ b^{\prime }-\left\{
c\left( \frac{\delta }{2}-\varepsilon ^{\prime }\right) -\frac{\left(
cc_{1}\right) ^{2}}{b}\right\} \right] ,
\end{equation*}%
on the set $\left\{ \Vert H_{n}\Vert _{q}+\left\Vert \frac{1}{n}%
\sum_{i=1}^{n}D_{\beta }h(X_{i},\beta _{\ast })u\right\Vert _{q}+\left\Vert
\frac{1}{n}\sum_{i=1}^{n}\bar{\kappa}(X_{i})\right\Vert _{q}\leq b^{\prime
}\right\} .$ Therefore, we can pick $c=4(b^{\prime }+1)/\delta +1,$ $%
\epsilon ^{\prime }=\delta /4$ and $b=\left( cc_{1}\right) ^{2}+1,$ then
\begin{equation*}
b^{\prime }-\left\{ c\left( \frac{\delta }{2}-\varepsilon ^{\prime }\right) -%
\frac{\left( cc_{1}\right) ^{2}}{b}\right\} <0.
\end{equation*}%
Notice that, there exists $b^{\prime }$ and $n_{3}$ such that for all $%
n>n_{3}$ and $\Vert u\Vert _{2}\leq K$
\begin{equation}
\mathrm{pr}\left\{ \Vert H_{n}\Vert _{q}+\left\Vert \frac{1}{n}%
\sum_{i=1}^{n}D_{\beta }h(X_{i},\beta _{\ast })u\right\Vert +\left\Vert
\frac{1}{n}\sum_{i=1}^{n}\bar{\kappa}(X_{i})\right\Vert _{q}>b^{\prime
}\right\} <\varepsilon /2.  \label{inter-tightness-comapact set}
\end{equation}%
Denote $n_{4}=\left( M_{\Delta }\left( 4(b^{\prime }+1)/\delta +1\right)
/\epsilon _{0}\right) ^{2}.$ Therefore, there exists $n_{0}$ such as%
\begin{equation*}
\mathrm{pr}\left[ \max \left\{ N^{\prime }(\delta
),n_{1},n_{2},n_{3},n_{4}\right\} >n_{0}\right] <\varepsilon /2.
\end{equation*}%
Finally, we have the statement of Lemma \ref{lemmaksai} as a consequence of
the union bound. \end{proof}

\begin{proof}[Proof  of Lemma \protect\ref{Lem-LLN-uniform}]
Lemma \ref{Lem-LLN-uniform} follows as a consequence of the continuity
properties of $D_{x}h(\cdot),D_{\beta}h(\cdot)$ and the strong law of large
numbers. The proof of Lemma \ref{Lem-LLN-uniform} is similar to the proof of
Lemma 3 in \citet{blanchet2016robust}.
\end{proof}

\begin{proof}[Proof  of Lemma \protect\ref{lem-lip-cont-Mn}]
For $i=1,\ldots,n$ and $j=1,2,$ let $\Delta_{ij}$ attain the inner supremum
in
\begin{equation*}
\max_{\Delta}\left[n^{1/2}\xi^{\T}\left\{ h(X_{i}+n^{-1/2}\Delta
,\beta_{\ast}+n^{-1/2}u_{j})-h(X_{i},\beta_{\ast})\right\}-\Vert\Delta
\Vert_{q}^{2}\right] .
\end{equation*}
Then
\begin{align*}
& \left\vert \max_{\Delta}\left[ n^{1/2}\xi^{\T}\left\{
h(X_{i}+n^{-1/2}\Delta,\beta_{\ast}+n^{-1/2}u_{1})-h(X_{i},\beta_{\ast})%
\right\} -\Vert\Delta\Vert_{q}^{2}\right] \right. \\
& \quad\quad\left. -\max_{\Delta}\left[n^{1/2}\xi^{\T}\left\{
h(X_{i}+n^{-1/2}\Delta,\beta_{\ast}+n^{-1/2}u_{2})-h(X_{i},\beta_{\ast
})\right\} -\Vert\Delta\Vert_{q}^{2}\right] \right\vert \\
& \quad\leq\max_{j=1,2}\left[ n^{1/2}\left\vert \xi^{\T}\left\{
h(X_{i}+n^{-1/2}\Delta_{ij},\beta_{\ast}+n^{-1/2}u_{1})-h(X_{i}+n^{-1/2}%
\Delta_{ij},\beta_{\ast}+n^{-1/2}u_{2})\right\} \right\vert \right] ,
\end{align*}
and consequently, it follows from the definition of $M_{n}(\xi,u)$ that,
\begin{align}
& \left\vert M_{n}(\xi,u_{1})-M_{n}(\xi,u_{2})\right\vert  \notag \\
& \,\leq\frac{1}{n}\sum_{i=1}^{n}\max_{j=1,2}\left[ n^{1/2}\left\vert
\xi^{\T}\left\{
h(X_{i}+n^{-1/2}\Delta_{ij},\beta_{\ast}+n^{-1/2}u_{1})-h(X_{i}+n^{-1/2}%
\Delta_{ij},\beta_{\ast}+n^{-1/2}u_{2})\right\} \right\vert \right] .
\label{inter-lip-Mn}
\end{align}

Next, due to fundamental theorem of calculus, we have that,
\begin{align}
& n^{1/2}\xi^{\T}\left\vert h(X_{i}+n^{-1/2}\Delta_{ij},\,\beta_{\ast}+
n^{-1/2}u_{1}) - h(X_{i}+n^{-1/2}\Delta_{ij},\,\beta_{\ast}+n^{-1/2}
u_{2})\right\vert  \notag \\
& \quad= \left\vert \int_{0}^{1} \xi^{\T} D_{\beta}h\left[
X_{i}+n^{-1/2}\Delta_{ij},\,\beta_{\ast}+n^{-1/2}\left\{
u_{1}+(u_{2}-u_{1})t\right\} \right] (u_{2}-u _{1}) {\rm d}t \right\vert  \notag \\
& \quad\leq\left\Vert u_{1}-u_{2}\right\Vert _{q} \Vert\xi\Vert_{p}
\int_{0}^{1}\left\Vert D_{\beta}h\left\{ X_{i}+ n^{-1/2}\Delta_{ij},\
\beta_{\ast }+ n^{-1/2}\left( u_{1}+(u_{2}- u_{1})t\right) \right\}
\right\Vert _{q} {\rm d}t  \notag \\
& \quad\leq\left\Vert u_{1}-u_{2}\right\Vert _{q} \Vert\xi\Vert_{p} \left[
\Vert D_{\beta}h(X_{i},\beta_{\ast})\Vert_{q} + \bar{\kappa}(X_{i}) \left\{
n^{-1/2}\Vert\Delta_{ij} \Vert+ 2c_{q} n^{-1/2} K \right\} \right] ,
\label{inter-Mn-Lip-2}
\end{align}
where $c_{q}$ is a fixed positive constant such that $\Vert x \Vert_{q} \leq
c_{q} \Vert x \Vert_{2}.$ The last inequality follows from Assumption A2.c).
Moreover, for a given $b,\nu,K > 0,$ we have from %
\eqref{inner-opt-inside-compact-set} that there exists $n_{0}$ such that $%
\Delta_{ij} \leq\nu n^{1/2},$ for all $i \leq n, n \geq n_{0},$ $\Vert
\xi\Vert_{p} \leq b,$ $\Vert u \Vert_{2} \leq K.$ Combining this observation
with those in \eqref{inter-lip-Mn} and \eqref{inter-Mn-Lip-2}, we obtain
that
\begin{align*}
\sup_{\Vert\xi\Vert_{p} \leq b} \left\vert M_{n}(\xi,u_{1}) - M_{n}(\xi
,u_{2})\right\vert \leq\Vert u_{1} - u_{2} \Vert_{q} b \left\{  E%
_{P_{n}}\left\Vert D_{\beta}h(X,\beta_{\ast})\right\Vert _{q} +  E%
_{P_{n}}\left\{ \bar{\kappa}(X)\right\} \left( \nu+ 2c_{q}n^{-1/2}K\right)
\right\} ,
\end{align*}
for all $n \geq n_{0}.$ For any random variable $Z,$ let $\text{CV}(Z) =
\mathrm{var}(Z)/E(Z)^{2}$ denote the coefficient of variation of $Z.$
If $n_{0}$ is also taken to be larger than both $2\varepsilon^{-1}\text{CV}%
\left\{\left\Vert D_{\beta}h(X,\beta_{\ast})\right\Vert _{q}\right\}$ and $%
2\varepsilon^{-1}\text{CV}\{\bar{\kappa}(X)\},$ then we have
\begin{align*}
{\rm pr}\left\{ E_{P_{n}}\left\Vert D_{\beta}h(X,\beta_{\ast})\right\Vert
_{q} \leq2 E\left\Vert D_{\beta}h(X,\beta_{\ast})\right\Vert _{q}
\right\} & \geq1-\varepsilon/2 \text{ and } \\
{\rm pr}\left[  E_{P_{n}}\left\{ \bar{\kappa}(X)\right\} \leq2 E\left\{
\bar{\kappa}(X)\right\}\right] & \geq1-\varepsilon/2.
\end{align*}
With these observations, if we take $L = 4b\left[E\left\Vert D_{\beta
}h(X,\beta_{\ast})\right\Vert _{q} + E\{\bar{\kappa}(X)\}(\nu+ 2
c_{q} K)\right],$ then
\begin{align*}
\sup_{\Vert\xi\Vert_{p} \leq b} \left\vert M_{n}(\xi,u_{1}) - M_{n}(\xi
,u_{2})\right\vert \leq L \Vert u_{1} - u_{2} \Vert_{q},
\end{align*}
with probability exceeding $1-\varepsilon.$
\end{proof}

\end{document}